\PassOptionsToPackage{dvipsnames}{xcolor}
\documentclass[11pt,a4paper]{amsart}

\usepackage[all]{xy}
\usepackage[utf8]{inputenc}
\usepackage[T1]{fontenc}      
\usepackage{geometry}[2.5cm]         
\usepackage[english]{babel}  
\usepackage{yfonts}
\usepackage{tikz}
\usepackage{tikz-cd}
\usetikzlibrary{arrows.meta, calc, positioning, shapes}

\usepackage{a4wide}
\usepackage{amsfonts,amsthm,amsmath,amssymb}
\usepackage{amscd,color}
\usepackage{psfrag,graphicx,subfigure,hyperref,enumerate}
\usepackage{import} 
\usepackage[shortlabels]{enumitem}
\usepackage[dvipsnames]{xcolor}
\usepackage{verbatim}
\usepackage[all]{xy}
\usepackage{soul}
\makeatletter
\@namedef{subjclassname@2020}{\textup{2020} Mathematics Subject Classification}
\makeatother

\theoremstyle{plain}
\newtheorem{theo}{Theorem}[section]
\newtheorem{thm}[theo]{Theorem}
\newtheorem{prop}[theo]{Proposition}
\newtheorem{coro}[theo]{Corollary}
\newtheorem{cor}[theo]{Corollary}
\newtheorem{lem}[theo]{Lemma}
\newtheorem{propdef}[theo]{Proposition/Definition}
\newtheorem*{lem*}{Lemma}
\newtheorem{conj}[theo]{Conjecture}
\newtheorem*{thmA}{Theorem A}
\newtheorem*{thmE}{Theorem E}
\newtheorem*{thmB}{Theorem B}
\newtheorem*{thmC}{Theorem C}

\newtheorem*{corD}{Corollary D}
\newtheorem*{corDp}{Corollary D'}

\theoremstyle{definition}
\newtheorem{defi}[theo]{Definition}
\newtheorem{defn}[theo]{Definition}

\newtheorem*{obs*}{Observation}
\newtheorem{rem}[theo]{Remark}

\newtheorem*{rem*}{Remark}

\numberwithin{equation}{section}

\DeclareMathOperator{\diam}{diam}
\DeclareMathOperator{\dist}{dist}

\DeclareMathOperator{\id}{Id}
\DeclareMathOperator{\area}{area}

\DeclareMathOperator{\inter}{int}
\DeclareMathOperator{\Log}{Log}

\newcommand{\R}{\mathbb{R}}

\newcommand{\N}{\mathbb{N}}
\newcommand{\Z}{\mathbb{Z}}
\newcommand{\C}{\mathbb{C}}

\newcommand{\D}{\mathbb{D}}

\newcommand{\chat}{\widehat{\mathbb{C}}}

\newcommand{\jac}{\mathrm{Jac}\,}

\newcommand{\dbar}{\overline{\partial}}
\newcommand{\rs} {\widehat{\mathbb{C}}}

\newcommand{\limn}{\lim_{n \rightarrow \infty}}

\newcommand{\eps}{\epsilon}

\newcommand{\acal}{\mathcal{A}}

\newcommand{\mcal}{\mathcal{M}}
\newcommand{\tcal}{\mathcal{T}}

\newcommand{\htil}{\widetilde{h}}

\newcommand{\re}{\mathrm{Re}}

\newcommand{\bif}{\mathrm{Bif}}

\newcommand{\lam}{\lambda}

\newcommand{\B}{\mathbb{B}}

\renewcommand{\H}{\mathbb{H}}

\newcommand{\beq}{\begin{equation}}
\newcommand{\eeq}{\end{equation}}

\renewcommand{\r}{\color{red}}

\definecolor{Violet}{cmyk}{0.79,0.88,0,0}
\definecolor{Lavender}{cmyk}{0,0.48,0,0}

\newcommand{\la}{\lambda}
\renewcommand{\phi}{\varphi}

\newcommand{\ra}{\rightarrow}

\newcommand{\JJ}{\mathcal{J}}

\renewcommand{\SS}{\mathcal{S}}

\renewcommand{\emptyset}{\varnothing}

\newcommand{\ov}{\overline}
\newcommand{\Tf}{\mathcal{T}(\mathbf f)}
\newcommand{\TT}{\mathcal{T}}
\newcommand{\ind}{\operatorname{ind}}
\renewcommand{\tilde}{\widetilde}

\renewcommand{\Im}{\operatorname{Im}}
\newcommand{\Id}{\operatorname{Id}}

\renewcommand{\hat}{\widehat}

\newcommand{\wring}{\mathcal{W}}
\newcommand{\holom}{\mathfrak{h}}
\renewcommand{\conj}{\Phi}
\newcommand{\riem}{\omega} 

\title{A universal model for the bifurcations of asymptotic values}

\author{Matthieu Astorg$^{\dag}$$^*$}
\author{Anna Miriam benini$^{\dag}$$^{**}$}
\author{N\'uria Fagella$^\ddag$}

\thanks{$^{\dag}$ Partially supported  by the Ministero dell’Universit\`a e della Ricerca, the  Minist\`ere de l’Europe et des Affaires Etrang\`eres, the Minist\`ere de l’Enseignement Sup\'erieur, de la Recherche et de l’Innovation and the  French Italian University and Campus France through the Galileo program, under the project G24-123 {\sl From rational to transcendental: complex dynamics and parameter spaces}}
\thanks{$^*$  Partially supported by the ANR PADAWAN /ANR-21-CE40-0012-01, ANR DynAtrois / ANR-24-CE40-1163, ANR TIGerS/ANR-24-CE40-3604.}
\thanks{$^{**}$ Partially supported  by GNAMPA, INDAM.}
\thanks{$^\ddag$  Partially supported by the Spanish State Research Agency, through the grant  PID2023-147252NB-I00, and the Severo Ochoa and María de Maeztu Program for Centers and Units of Excellence in R\&D (CEX2020-001084-M); and by the Catalan Government through the excellence award  ICREA Acadèmica 2020.  }
\address{ M. Astorg: IRMAR,  Université de Rennes\\
  France \newline
  ORCID 0000-0001-5138-0021} \email{{ matthieu.astorg@univ-rennes.fr}}
\address{ A.M. Benini: Dipartimento  di   Matematica Fisica e Informatica\\
Universit\'a di Parma\\ Italy \newline
ORCID 0000-0001-9014-5235} \email{ambenini@gmail.com}
\address{ N.  Fagella$^{1,2}$:  \newline 
$1$ Dep. de Matem\`atiques i Inform\`atica\\ Universitat de Barcelona\\ Barcelona\\  Spain \newline
$2$ Centre de Recerca Matem\`atica\\ Barcelona\\ Spain \newline
ORCID 0000-0002-5466-0579} \email{nfagella@ub.edu}


\date{\today}

\begin{document}

\begin{abstract}
We study a universal model for the bifurcation of asymptotic values, based on the notion of tangent-like maps and families thereof. In analogy with polynomials and polynomial-like maps, we prove that every tangent-like map can be associated to a member of the model family. In particular, the filled Julia sets of tangent-like maps are quasiconformally homeomorphic to the ones for the model family. Likewise, the bifurcation locus of a family of tangent-like maps is homeomorphic to the boundary of the {\em Tandelbrot set} $\mathcal T$, the bifurcation locus of the model. Moreover $\mathcal T$ is universal, in the sense that homeomorphic copies of it can be found densely in the activity loci of  asymptotic values of any appropriate natural family of finite type meromorphic maps. When applied to the iterates of a transcendental meromorphic map, these results make it possible to build a transcendental renormalization theory, a question that has been open for the last 25 years. 
\end{abstract}
\maketitle

\tableofcontents

\section{Introduction} \label{sec:intro}

The quadratic family $f_c(z):=z^2+c$, $c \in \C$ is  the simplest non-trivial holomorphic family of rational maps in one complex variable.
 These maps have only one {critical point}, located at  $z=0$, which is the only point at which $f_c$ fails to be a local homeomorphism. The  orbit of $z=0$ governs the global dynamics of $f_c$. The  Mandelbrot set $\mathcal M$,  defined as the set of parameters $c\in\C$ for which the orbit of $z=0$ does not escape to infinity, has been the subject of extensive research over the last decades and still is the object of one of the main  open questions in the field:  the MLC conjecture, which states that $\mcal$ is locally connected. It is known that MLC would imply a positive answer to several other important conjectures, such as the Fatou conjecture on the genericity of hyperbolicity in the quadratic family.
 
A remarkable feature of the  family of quadratic polynomials and of the Mandelbrot set is its {\em universality}. In the early 80's,
many numerical experiments showed the existence of topological copies of quadratic filled Julia sets (points with bounded orbit) in the dynamical plane of polynomial or rational maps, or copies of the Mandelbrot set in the parameter spaces of most holomorphic families of rational maps; or, in particular,  copies within itself.
Both of these phenomena are explained by the seminal work of Douady and Hubbard on polynomial-like mappings.

Introduced in  \cite{douhub}, polynomial-like mappings provide a powerful conceptual framework that allows one to transfer results from  polynomial dynamics to a much broader class of holomorphic dynamical systems.  A polynomial-like map is a triple $(f, U, V)$ where $U$ and $V$ are topological disks, $\overline{U}\subset V$, and $f:U\to V$ is a proper holomorphic map of degree $d\geq 2$.  
The fundamental result of this theory is the so-called Straightening Theorem (\cite[Theorem 1]{douhub}), which asserts that every polynomial-like map is (quasiconformally) conjugate to a genuine polynomial $P$ in a neighborhood of their respective filled Julia sets. 
On the other hand, given a rational map $f: \hat{\C} \to \hat{\C}$, the flexibility of this definition allows in many situations to construct restrictions $f_{|U}^n: U \to V$ of some iterate of $f$ 
which are quadratic-like, or more generally polynomial-like.
The Straightening  Theorem then  explains why copies of polynomial Julia sets are encountered in the phase space of many rational or transcendental maps, near a critical point.
In the particular case when $f$ itself is a quadratic polynomial with connected Julia set, if such a restriction exists, then $f_c^n: U \to V$ is  hybrid  conjugate  (i.e. conjugate by a quasiconformal homeomorphism $\Phi$ such that $\overline{\partial}\Phi=0$ a.e. on the filled Julia set) to a unique $f_{c'}$, $c' \in \mcal$.  This gives rise to a \emph{renormalization operator} $\mathcal{R}: c \mapsto c'$. 
The study of the fine properties of this renormalization operator  are closely related to the MLC conjecture, by the famous work of Yoccoz (\cite{hubbard1992local}). We refer the reader  to the recent survey \cite{dudko2025mlc} for an overview of this topic, and the related works of Avila, Kahn, Lyubich, McMullen, Shen and many others.

There is also a parameter version of the Straightening Theorem, which applies only to families of unicritical polynomial-like maps (in particular, to families of quadratic-like maps).
It allowed McMullen \cite{mcmullen2000mandelbrot} to prove the universality property of the Mandelbrot set mentioned above: in any holomorphic family of rational maps with non-empty bifurcation locus and simple critical points, there exist quasiconformal copies of the Mandelbrot set (or of its higher degree analogue, in the case of critical points with persistence higher multiplicity).
\vspace{0.2cm}

In a different setting, there has been remarkable progress during the recent years in the understanding of the dynamics of maps with essential singularities, an area known as {transcendental} dynamics. In the presence of critical points, transcendental functions behave locally like polynomials and we see polynomial Julia sets in their dynamical plane, as explained above. But a different type of singularity of the inverse map may exist for transcendental functions, namely {\em asymptotic values}. Heuristically speaking, asymptotic values $v$ are points that have some of their preimages "at the essential singularity" (see  the formal Definition \ref{def:asymptotic}). The presence of critical values (that is, images of critical points) or asymptotic values is the only obstruction to the existence of well-defined local inverse branches of the map $f$. The exponential family $z\mapsto e^z +\la$ with two asymptotic values, $\la$ and $\infty$, and no critical points, is a possible transcendental entire analogue of $f_c(z)=z^2+c$.

 For a long time there was hope for a renormalization theory for the exponential family, despite the absence of compactness (see \cite[Sect.VI.6]{sch_hab} for a formulation of this question). At least topologically, it was thought that, by collapsing certain dynamical rays and the regions between them, the resulting projection would be conjugate to another exponential map. However, this was proved to be false even for the simplest case of attracting exponential maps, as explained by Rempe in \cite[Thm.1.3, or Sect.~10]{rem-escaping}.

The goal of this paper is to show that a transcendental-like theory leading to renormalization {\em is} possible, provided that we consider a larger class of maps, namely  transcendental {\em meromorphic} maps or, in other words, provided that we allow the existence of poles (preimages of the essential singularity). 

 The starting point is the notion of a \emph{tangent-like map}, which is the analogue of a quadratic-like map, but for asymptotic values instead of simple critical points. A  holomorphic map $f: U \to V$ is tangent-like if there exists $u\in \partial U$ and $v \in V$ such that $f: \overline{U}\setminus \{u\} \to \overline{V} \setminus \{v\}$ is a universal cover, where $U, V$ are Jordan domains and $U \Subset V$ is a quasidisc. Then $v$ is an asymptotic value, and $u$ is  an essential singularity. See Figure \ref{fig:TLmap} and Definition \ref{def:tlmap}.
This notion was introduced by Galazka and Kotus in \cite{GaKo}.

\begin{figure}[hbt!]
	\begin{center}
		\includegraphics[width=0.3\textwidth]{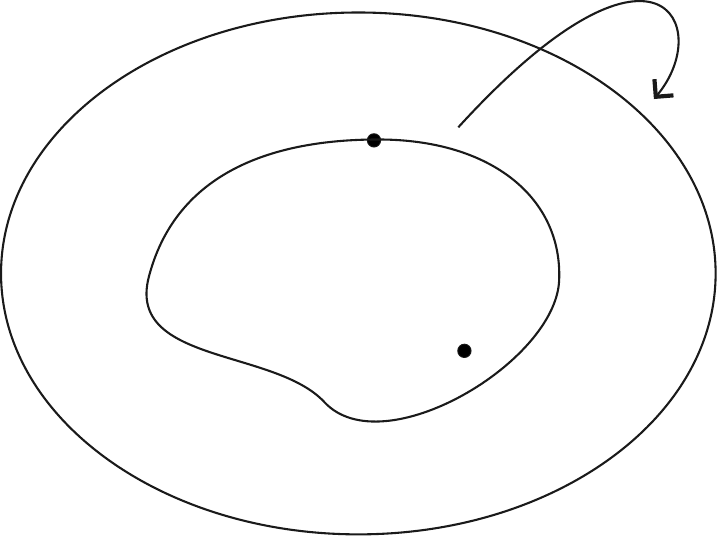}
		\setlength{\unitlength}{0.3\textwidth} 
		\put(-0.85, 0.65){\large$V$}
		\put(-0.7, 0.55){\large$U$}
		\put(-0.03, 0.7){$f$}
		\put(-0.48, 0.58){$u$}
		\put(-0.41, 0.25){$v$}
\caption{\label{fig:TLmap} The definition of a tangent-like map. The point $u \in \partial U$ is an essential singularity. The asymptotic value $v$ can be located anywhere in $V$.}
	\end{center}
\end{figure}

Similarly to polynomial-like maps, one can define the {\em filled Julia set} $K(f)$ of a tangent-like map $f$ as the points that either never leave $U$ under iteration, or eventually land on the essential singularity $u$. The Julia set  $J(f)$ is defined  as the boundary of $K(f)$.  The connectedness of $K(f)$ or $J(f)$ can be proved to be  equivalent to the property   that $v \in K(f)$ (Proposition \ref{prop:Kconnected}).  
Likewise, the concept of quasiconformal conjugacy and hybrid conjugacy can be appropriately defined in this setting (see Section \ref{sec:TLmaps} for details).

Note that,  unless $u=v$, tangent-like maps have points which are mapped to an essential singularity.
Therefore, the class of entire maps is somehow too restrictive to admit different tangent-like restrictions; instead, 
the natural setting is the class of transcendental meromorphic maps, and their iterates.

We now introduce a model family, which will be our analogue of the quadratic family. Set

\begin{equation}\label{def:T}
	T_\alpha(z) :=  \frac{e^{(\alpha-1)z/8}-1}{\alpha e^{(\alpha-1)z/8}-1}, \quad \quad \alpha \in \C \setminus \{1\}.
\end{equation}

For $\alpha\in\C\setminus\{1\}$, the map $T_\alpha$ has exactly two asymptotic values (which are $1$ and $\frac{1}{\alpha}$), no critical points and it also fixes $0$ with multiplier equal to $\frac{1}{8}$. These properties characterize the family, up to affine conjugation (see Lemma \ref{lem:rigidity}).
Putting aside the choice of multiplier $\frac{1}{8}$,  which is  convenient, 
this is in a way the simplest possible family of meromorphic maps with a free asymptotic value. Observe that $T_0$ is the exponential map $1-e^{-z/8}$, where the omitted value $1/\alpha$ coincides with the essential singularity, hence one can think of the family $T_\alpha$ as the one generated by the  splitting of these two elements, while keeping the fixed point $0$ and its multiplier. Maps in this slice   (or a reparametrization of it) were previously studied in \cite{GaKo,GaKo2,chen2022accessible} and  \cite{chen2023slices}.

This family exhibits a symmetry with respect to the involution $\alpha \mapsto \frac{1}{\alpha}$, in the sense that $T_{\alpha}$ is linearly conjugate to $T_{1/\alpha}$, and the conjugacy switches the two asymptotic values.
We shall therefore be mostly concerned with values of $\alpha$ in  $\overline{\D}\setminus \{1\}$, and work with the free asymptotic value $\frac{1}{\alpha}$. Indeed, for $\alpha \in \overline{\D} \setminus \{1\}$, one can show  (Lemma \ref{lem:disk mapping inside}) that the asymptotic value 1 is always in the basin of $0$, while $1/\alpha$ is free and may or may not be captured by this basin. The basin of $0$, which we denote by $\Omega_\alpha$, will play the same role as  the basin of infinity for polynomials. 

We define the {\em filled Julia set} of $T_\alpha$ as $K(T_\alpha):=\rs \setminus \Omega_\alpha$.
The Julia set is the usual set of non-normality. As we will see in Section \ref{sec:model}, the maps 
$T_\alpha$ all admit tangent-like restrictions, and the definitions of filled-Julia sets and Julia sets agree 
with those given above for tangent-like maps.

In contrast to the polynomial case and the quadratic family, hybrid conjugacy classes (see Section \ref{sec:hybrid classes} ) within the set of $\alpha\in \D$ for which $K(T_\alpha)$ is connected, contain countably many elements. These elements are related by the action of a map in parameter space, which we call the \emph{wringing map} $\wring$, whose construction exploits a classical surgery procedure used to change the multiplier of attracting cycles  (see section~\ref{sec:wringing}). Succintly, $\alpha$ and $\wring (\alpha)$ are endpoints of a continuous path of hybrid conjugate maps (in a 2-dimensional parameter space) for which, as we follow the path,  the multiplier of zero runs once from $\frac18$ to itself, around the circle of radius $\frac18$.

Our main motivation for introducing the family $(T_\alpha)_{\alpha \in \C \setminus \{1\}}$ is the following straightening theorem. See Figure \ref{fig:ST}. 

\begin{thmA}[Straightening Theorem]
	Let $(f,U,V,u,v)$ be a tangent-like map. Then 
 there exists $\alpha \in \C \setminus \{1\}$ such that $f$ is hybrid conjugate to $T_\alpha$, with the hybrid conjugacy mapping $v$ to $\frac{1}{\alpha}$. Moreover, if $K(f)$ is connected, then $\alpha$ is unique  up to the action of the wringing map.
\end{thmA}

The first part of this result was already anticipated by Galazka and Kotus in \cite{GaKo}, though the argument given there appears to be incomplete.

\begin{figure}[hbt!]
	\begin{center}
		\includegraphics[width=0.4\textwidth]{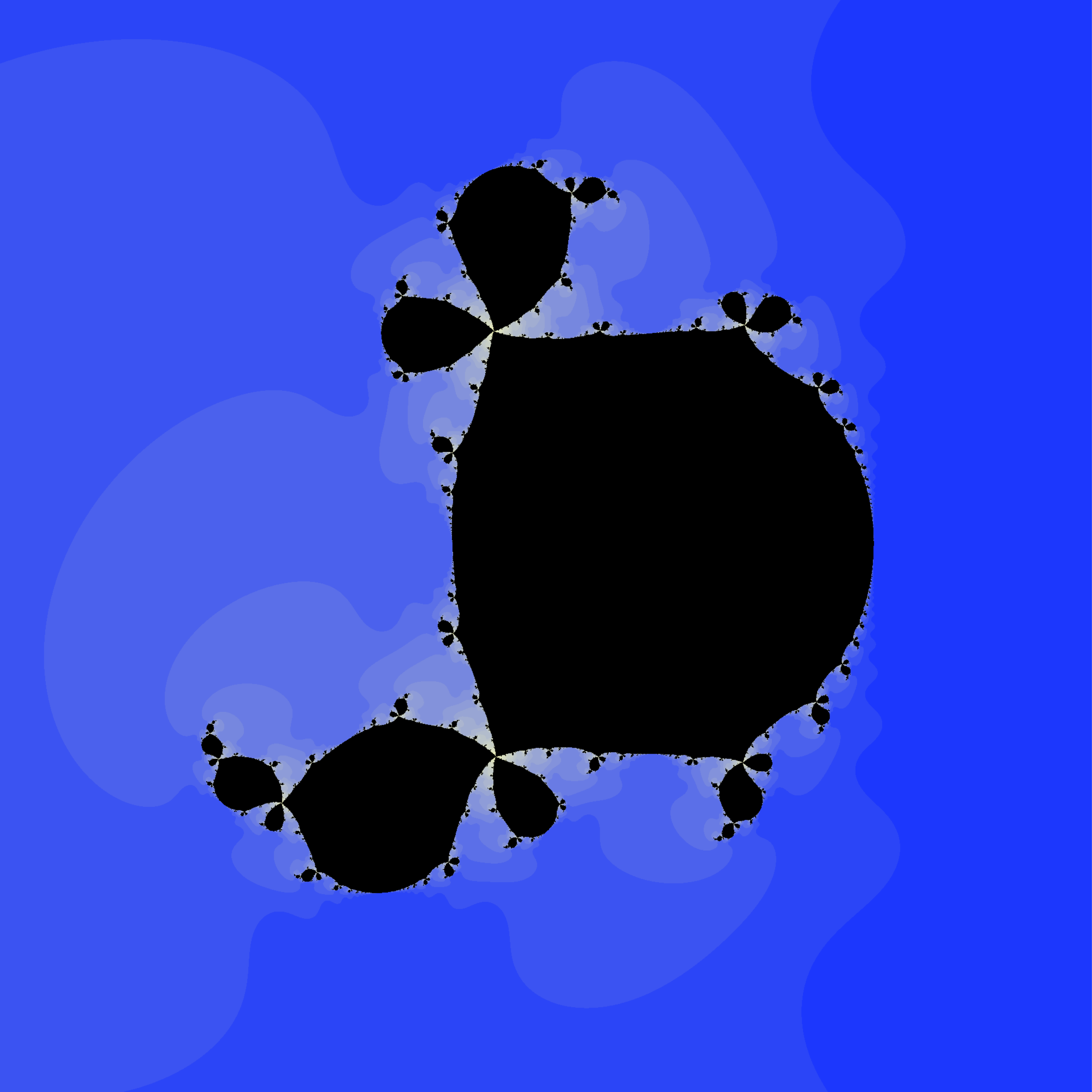}
		\hfil
		\includegraphics[width=0.4\textwidth]{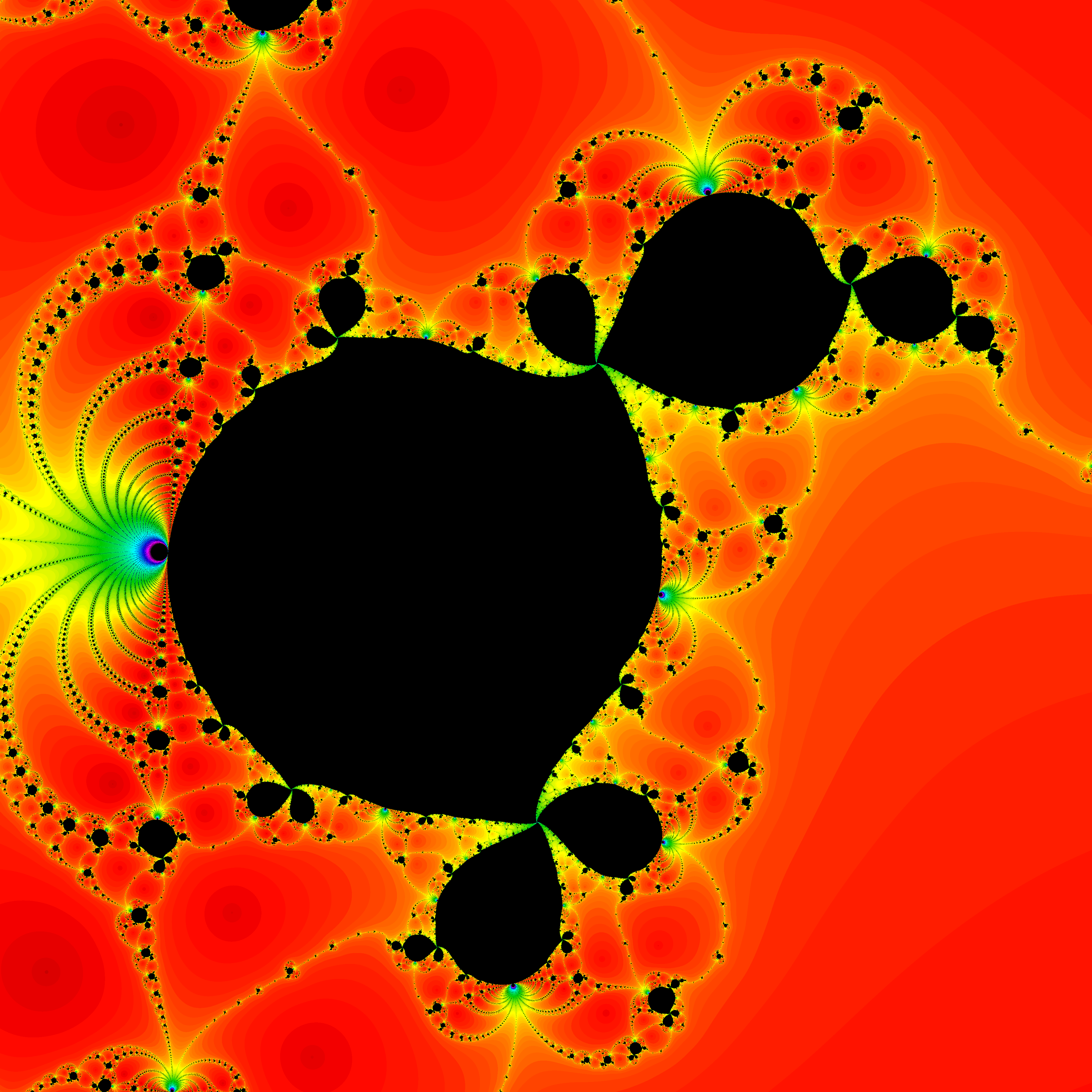}		\setlength{\unitlength}{\textwidth}
		\put(-0.55,0.2){\color{white}{\circle*{0.01}}}
		\put(-0.545,0.195){\color{white}{\small $\infty$}}
		\put(-0.63,0.2){\color{white}{\circle*{0.007}}}
		\put(-0.68,0.31){\color{white}{\circle*{0.007}}}
		\put(-0.71,0.28){\color{white}{\circle*{0.007}}}
		\put(-0.72,0.28){\color{white}{\circle*{0.007}}}
		\put(-0.728,0.267){\color{white}{\tiny $v$}}
		\put(-0.335,0.197){\color{white}{\circle*{0.01}}}
		\put(-0.329,0.195){\color{white}{\scriptsize $u$}}
		\put(-0.208,0.07){\color{white}{\circle*{0.007}}}
		\put(-0.186,0.095){\color{white}{\circle*{0.007}}}
		\put(-0.172,0.095){\color{white}{\circle*{0.007}}}
		\put(-0.176,0.099){\color{white}{\scriptsize $v$}}
		\put(-0.25,0.19){\color{white}{\circle*{0.007}}}
		\caption{\label{fig:ST} \small Left: The filled Julia set $K(T_\alpha)$  of $T_\alpha$ for $\alpha\simeq-0.021+i 0.009$, showing the basin of attraction of an attracting 3-periodic orbit. The asymptotic value is $v=1/\alpha$. Right: A quasiconformal copy of $K(T_\alpha)$ in the dynamical plane of a family $N_a$, corresponding to the Newton's method of $f_a(z)=z+ a e^z$ for $a\simeq-1.1627+i 0.1143$. There is a 3-periodic orbit under $N_a^2$. The point $u$ is a pole of $N_a$ which becomes an essential singularity under the second iterate. The asymptotic values of $N_a^2$ are  $v_1=0$ and $v_2=-a/(1+a)$.} 
	\end{center} 
\end{figure}

In  analogy to  the Mandelbrot set, for the model family we define the \emph{Tandelbrot set} $\tcal$ as (see Figure \ref{fig:Tandelbrot})
\begin{figure}[h]
	\begin{center}
		\includegraphics[width=0.5\textwidth]{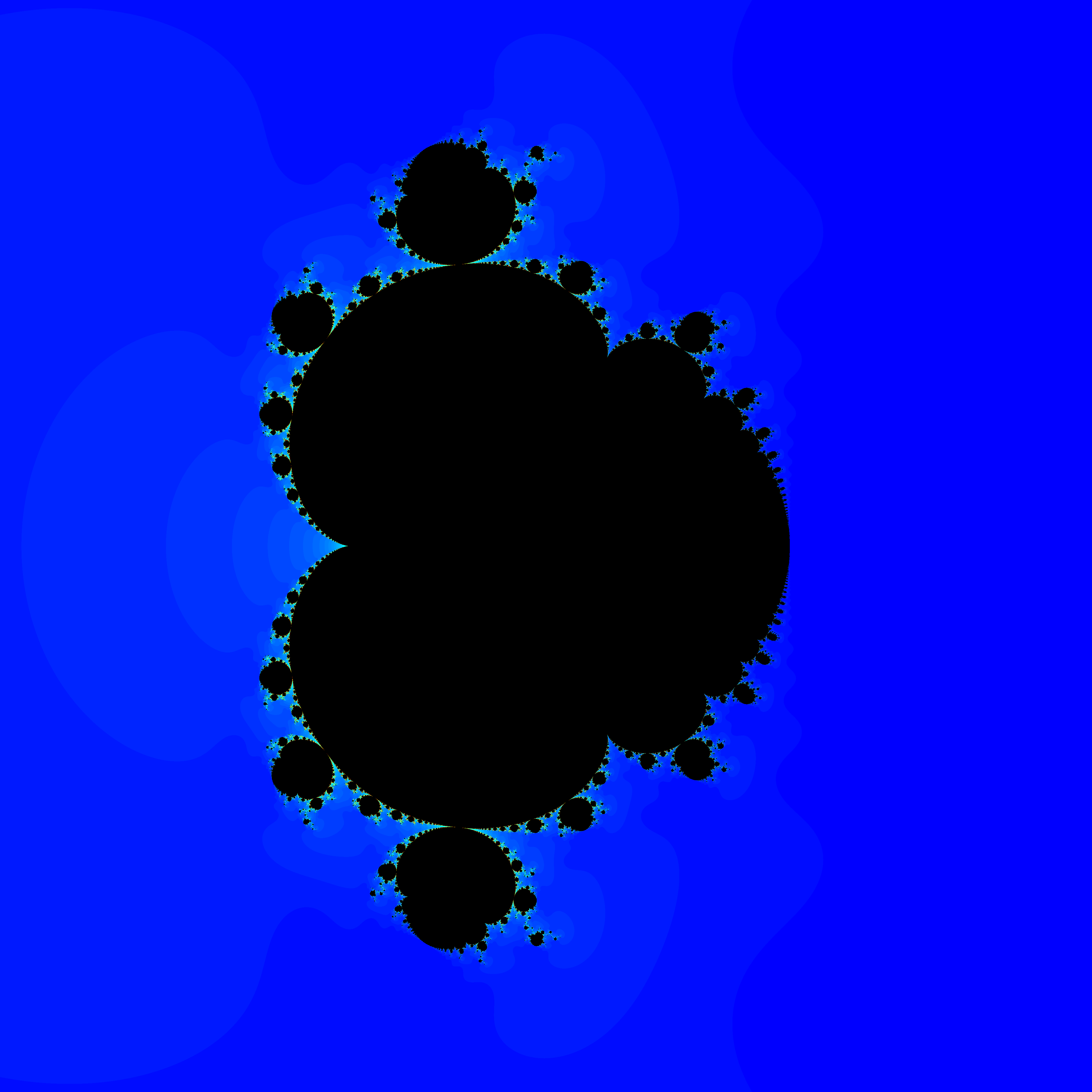}
		\setlength{\unitlength}{0.5\textwidth}
		\put(-0.275,0.5){\color{white}{\circle*{0.02}}}
		\put(-0.26,0.49){\color{white}{\scriptsize $0$}}
		\put(-0.68,0.5){\color{white}{\circle*{0.02}}}
		\put(-0.66,0.49){\color{white}{\scriptsize$-0.019$}}
		\caption{\label{fig:Tandelbrot} \small In black, the Tandelbrot set. } 
	\end{center} 
\end{figure}
\begin{equation}\label{def:M}
	\tcal:=\{\alpha \in \C \setminus \{1\}: \text{ $T_\alpha^n \left(\frac{1}{\alpha}\right)\not\rightarrow 0$ as $n\to\infty$}\}.   
\end{equation}

\begin{thmB}[Properties of $\tcal$]
	The  Tandelbrot set $\tcal$ is a full and connected compact set, contained in $\D$. Its boundary is the bifurcation locus of the family $\{T_\alpha\}_{\alpha \in \D}$. Moreover, for all $\alpha \in \D$,
	$\alpha \in \tcal$ if and only if $J(T_\alpha)$ is connected.
\end{thmB}

The Tandelbrot set $\tcal$  can be viewed as an infinite-degree analogue of the Mandelbrot set. It has been studied from a combinatorial perspective in \cite{GaKo,GaKo2,chen2022accessible} and  \cite{chen2023slices}. 
The assertion that $\tcal$ is connected and full was already established in \cite{chen2023slices}; here, we give a different and shorter proof of these properties (see Theorem \ref{th:msconnected}).

An important feature of $\TT$, and a key result of this paper, is that copies thereof appear naturally in families of meromorphic maps. 
Recall that a holomorphic family of meromorphic maps $\{f_\la\}_{\la \in\Lambda}$, with $\Lambda$ a finite-dimensional complex connected manifold, is called \emph{natural} if there exist holomorphic motions  $\phi_\la: \rs \to \rs$ and $\psi_\la:  \C \to \C$ such that $f_\la = \phi_\la \circ f_{\la_0} \circ \psi_\la^{-1}$, for some $\la_0 \in \Lambda$ (see Section \ref{sec:natfam}). 
 For example, $T_\alpha$ is a natural family of meromorphic maps, with base map e.g.  $T_0(z)=1-e^{-z/8}$, and holomorphic motions $\psi_\alpha(z)=z/(1-\alpha)$ and $\phi_\alpha(z)=\frac{z}{\alpha z+1-\alpha}$.  The {\em activity locus} $\acal(v_{\la_0})$ of a marked asymptotic value $\la \mapsto v_\la:=\phi_\la(v_{\la_0})$ is 
by definition the non-normality locus of $\{ \la \mapsto f_\la^n(v_\la) , n \in \N \}$. 
We also refer to Section \ref{sec:natfam}  for the link between activity loci and 
$J$-stability, and to Section \ref{sec:tracts} for the definition of a {\em tract over an asymptotic value}.

The following theorem states that  copies of the Tandelbrot set can be found  arbitrarily close to any parameter $\la_*$ in the activity locus  of an asymptotic value of a natural family (satisfying certain conditions). Observe that the theorem is local, since
for every $\la_*\in\Lambda$ and any one-dimensional topological disc $D \subset \Lambda$ containing $\la_*$, the natural family can be reparametrized in $D$, so that it is defined for $\la\in \D$, with basepoint $f_{\la_*}=f_0$. 

\begin{thmC}[Universality of the Tandelbrot set]
		Let $\{f_\lam\}_{\la\in \D}$ 
	be a natural family of meromorphic maps, and assume that the following conditions hold:
	\begin{enumerate}
		\item\label{it:vc} $f_0$ has an active isolated asymptotic value $v_0$, whose forward orbit does not persistently contain a critical point;
		\item\label{it:tractqd} $f_0$ has a tract above $v_0$ which is a quasidisk;
		\item\label{it:pole} and $f_0$ has at least one simple pole which is not a singular value.
	\end{enumerate}
	{Then  there exists a topological embedding $\chi: \partial \tcal  \to \acal(v_0)$.}
\end{thmC}

One of the two main ingredients in the proof of Theorem C is a way to construct a one-parameter family of tangent-like maps from suitable restrictions of some iterates of the maps $f_\la$. In \cite{mcmullen2000mandelbrot},  the mechanism leading to the construction of polynomial-like restrictions is a careful analysis of perturbations of maps with a Misiurewicz relation, i.e. maps which have a critical point whose forward orbit eventually lands on a repelling cycle.
In our case, Misiurewicz relations are replaced with so-called \emph{virtual cycles}, i.e. asymptotic values whose forward orbit eventually lands on an essential singularity. See Figure \ref{fig:renormalization}. Virtual cycles are linked to bifurcations of families of meromorphic maps, as explained in Section \ref{sec:prelim} with a a short review of the results from \cite{ABF, ABFAhlfors}.
Assumption \eqref{it:tractqd}  in Theorem C is necessary to ensure the existence of these tangent-like restrictions. 
Additionally, the existence of non-constant families of tangent-like restrictions  imply the existence of poles, which is why we need assumption
\eqref{it:pole}. This explains why topological copies of $\partial \tcal$ are not observed in  pictures of parameter spaces of entire maps (such as the exponential family $f_\la(z):=\la e^z$ for instance).

As a consequence of Theorem C, and of  the results proved by McMullen in \cite{mcmullen2000mandelbrot},  copies of either Mandelbrot sets or Tandelbrot sets are dense in bifurcation loci of natural families of finite type meromorphic maps,  that is, meromorphic maps with finitely many singular values. Recall that, analogously to $\mcal$,  for any $d\geq2 $ the degree $d$ Mandelbrot sets $\mcal_d$ are defined as the values of $\la$ for which the Julia set of $z\mapsto  z^d+\la$ is connected.

\begin{corD}[Density of copies of  $\mcal_d$ or $\tcal$] \label{coro:C}
	Let $\{f_\la\}_{\la \in \Lambda}$ be  a natural family of finite type meromorphic maps. Assume that the following conditions hold:
	\begin{enumerate}
		\item No critical point is persistently a Picard exceptional value.
		\item There exists $\la_0 \in \Lambda$ such that all tracts above quasidisks of the map $f_{\la_0}$  are quasidisks.
		\item For all $\la \in \Lambda$, there exists a simple pole $p_\la$ which is not a singular value.
	\end{enumerate}	
	Then, near any $\la_1$ in the bifurcation locus, there exists a topological embedding of  $\partial \mcal_d$ or $\partial \tcal$  which is contained in the bifurcation locus of the family.
\end{corD}

A special case of Corollary D is the following self-similarity property of the Tandelbrot set, similarly to the Mandelbrot set. See Figure \ref{fig:copies_tan} (right).

\begin{corDp}[Self-similarity of $\tcal$]
	For any open set $W\subset \C$ intersecting $\partial \tcal$, there is a topological embedding $\chi: \partial \tcal \to W \cap \partial \tcal$.
\end{corDp}

\begin{figure}[htb!]
\centering
\includegraphics[width=0.35\textwidth]{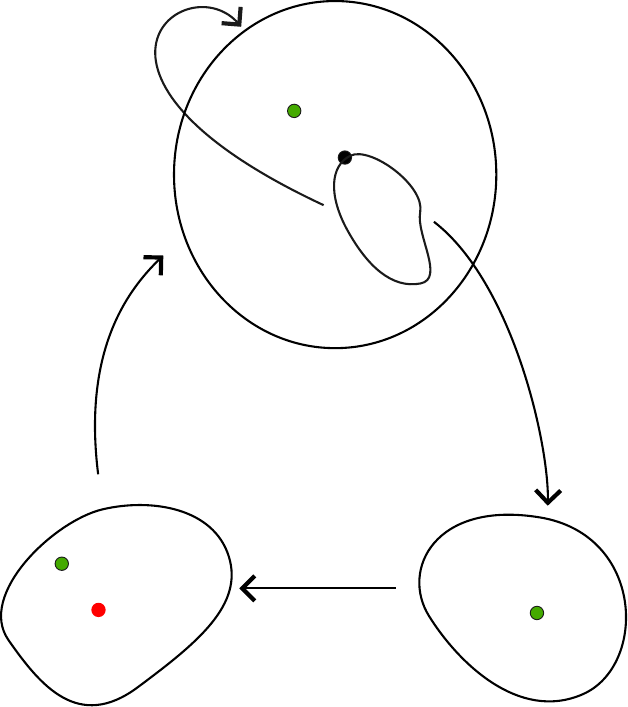}
\setlength{\unitlength}{0.35\textwidth}
\put(-0.28,1.05){$V$}
\put(-0.37,0.87){\small $U_\la$}
\put(-0.14,0.55){\small $f_\la$}
\put(-0.92,0.55){\small $f_\la$}
\put(-0.53,0.23){\small $f_\la^{n+1}$}
\put(-0.04,0.26){\small $D_\la$}
\put(-0.88,0.1){\scriptsize $p_\la$}
\put(-0.18,0.1){\scriptsize $v_\la$}
\put(-0.88,0.22){\scriptsize $f_\la^{n+1}(v_\la)$}
\put(-0.6,0.99){\scriptsize $f_\la^{n+2}(v_\la)$}
\put(-0.88,1.07){\small $f_\la^{n+3}$}
\put(-0.53,0.87){\scriptsize $\infty$}
\caption{\small \label{fig:renormalization} Sketch of the tangent-like renormalization scheme. Suppose $\la_0$ is a virtual cycle parameter for which $f_{\la_0}^n(v_{\la_0})=\infty$. Then, for nearby parameters $\la$ the domain $U_\la$ of the tangent-like restriction is a tract above a disk $D_\la$ containing the asymptotic value $v_\la$. The iterate $f_\la^{n+1}$ maps $D_\la$ to another disk containing the pole $p_\la$, which is then mapped to a fixed disk $V$ centered at infinity. Hence, the map $f_\la^{n+3}: U_\la \to V \setminus \{f_\la^{n+2}(v_\la)\}$ is tangent-like.}
\end{figure}

Another fundamental  ingredient in the proof of Theorem C is the following parameter version of the Straightening Theorem, which is of  independent interest. 
Informally, a tangent-like family is a holomorphic family of tangent-like maps $\{f_\la : U_\la \to V_\la\}_{\la \in \Lambda}$. As for families of quadratic-like maps, we say that such a family is \emph{proper} if $\Lambda$ is a Jordan domain, if $\{f_\la\}_{\la \in \Lambda}$ is a restriction of another tangent-like family defined over a slightly larger parameter space $\tilde \Lambda$ and domains $\tilde U_\la$, and if the asymptotic value $v_\la$ winds once around the essential singularity as $\la$ turns around $\partial \Lambda$. We say that it is \emph{equipped} if there is a holomorphic motion of the fundamental annulus $\overline{V_\la} \setminus U_\la$ commuting with $f_\la$. 
See Definitions \ref{def:tlfamily}, \ref{def:equipped} and  \ref{def:proper} and Fig. \ref{fig:copies_tan}.

\begin{thmE}[Parameter version of the Straightening Theorem]
	Let $\{f_\la\}_{\la \in \Lambda}$ be a proper equipped tangent-like family. Let
	$$\tcal(\mathbf f) = \{\la \in \Lambda : v_\la \in K(f_\la)\}.$$
	Then, there exists a continuous map $\chi: \Lambda \to \C \setminus \{1\}$ such that 
	\begin{enumerate}
		\item For all $\la \in \Lambda$, $f_\la$ is hybrid conjugate to $T_{\chi(\la)}$.
		\item $\chi: \tcal(\mathbf f) \to \tcal$ is a homeomorphism.
		\item $\chi$ is holomorphic on $\inter(\tcal(\mathbf f))$. 
	\end{enumerate}
\end{thmE}

\begin{figure}[h]
	\begin{center}
		\includegraphics[width=0.45\textwidth]{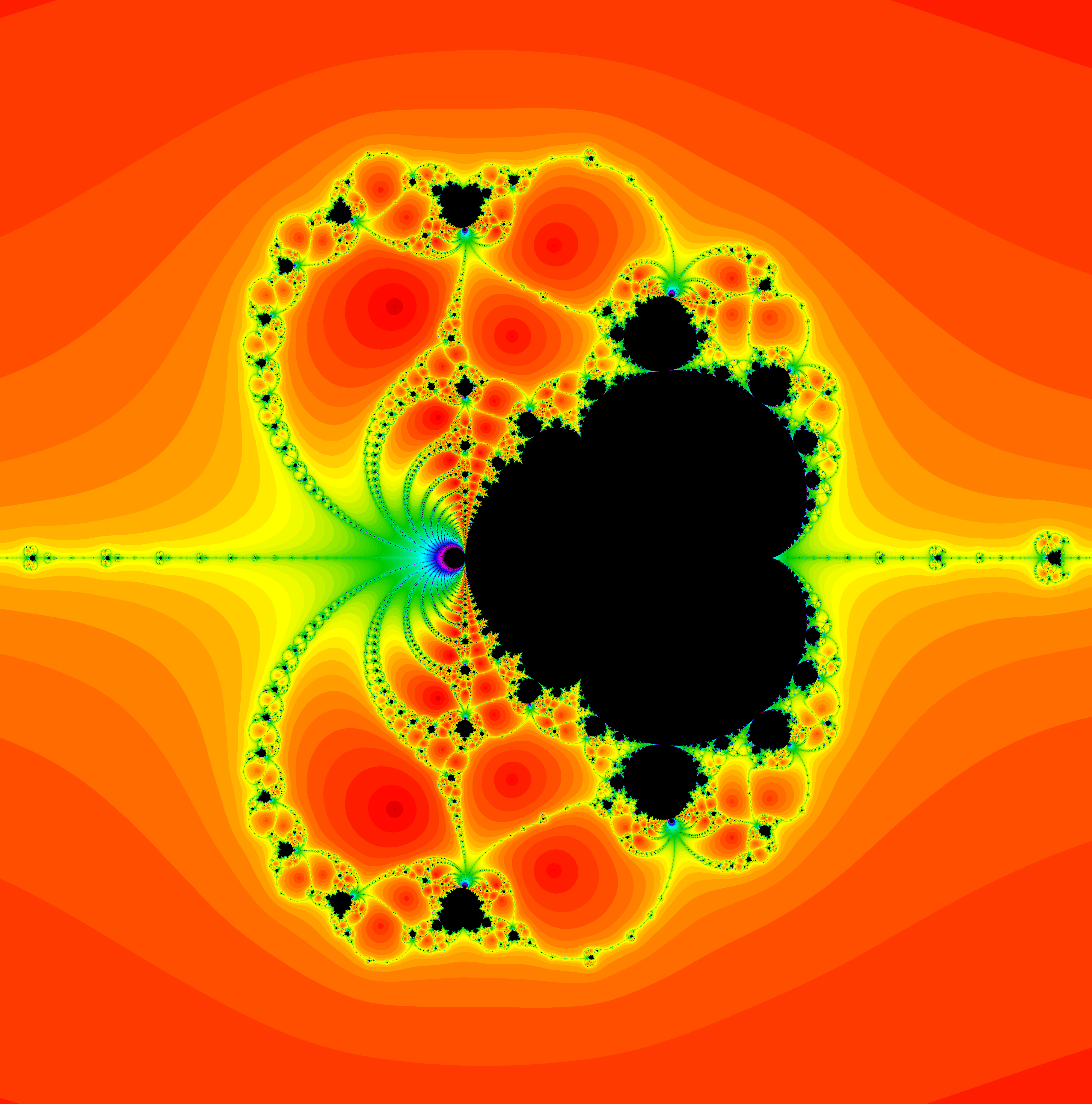}
		\hfil
		\includegraphics[width=0.45\textwidth]{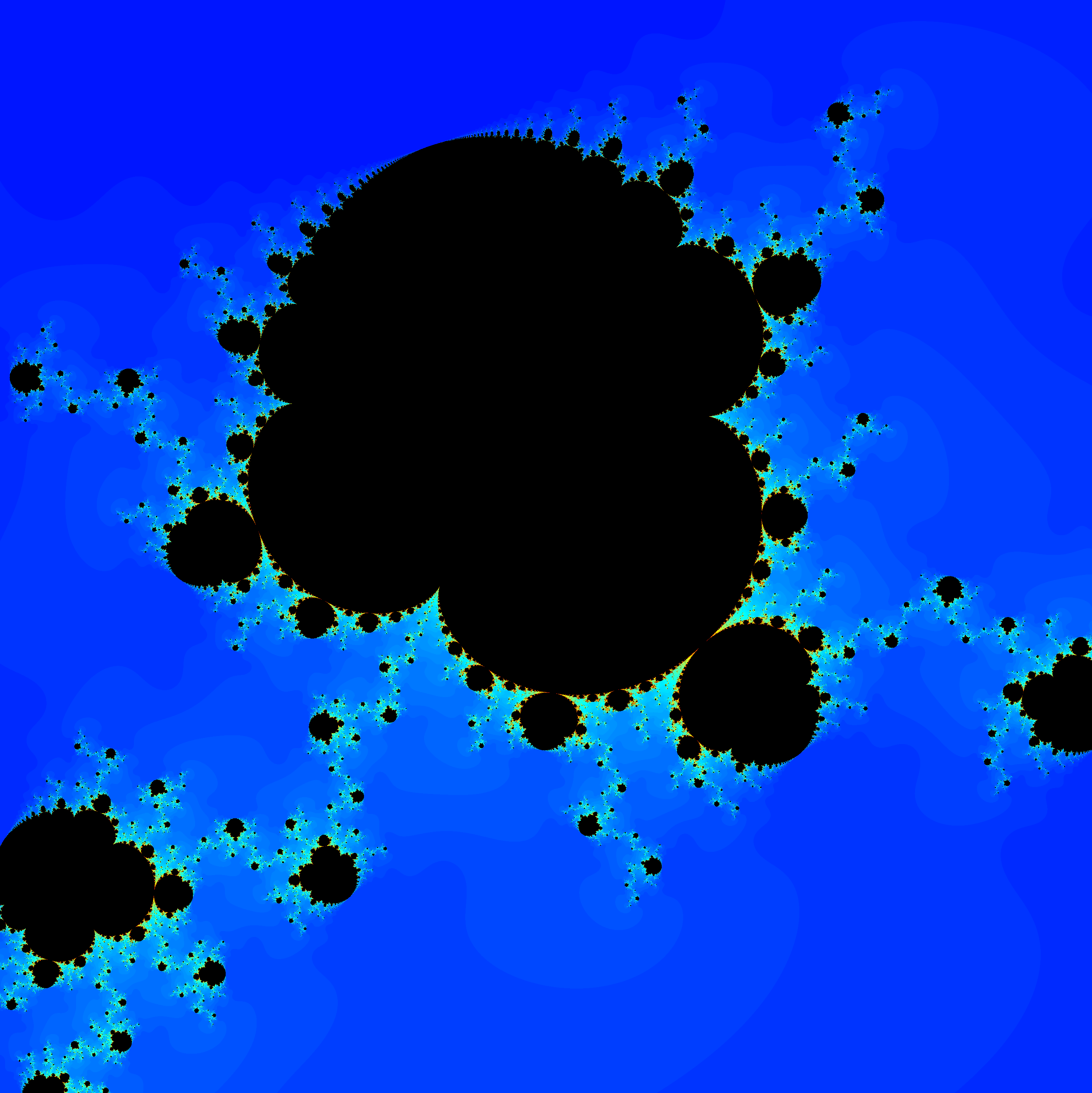}
		\caption{\label{fig:copies_tan} \small Left: A copy of $\tcal$ in the parameter space of the Newton's method of $f_a(z)=z+ae^z$, illustrating Theorem C. Right: A small copy of $\tcal$ inside $\tcal$, illustrating Theorem E.   } 
	\end{center} 
\end{figure}

\medskip

\pagebreak
\textbf{Organization of the paper}

In Section \ref{sec:prelim} we  collect results on holomorphic motions, on families of meromorphic maps, and on properties of tracts in such families.   In Section \ref{sec:TLmaps} we define tangent-like maps, and we prove some general properties of Julia sets. In Section \ref{sec:model} we introduce the model family $\{T_\alpha\}_{\alpha \in \C}$, we discuss affine, quasiconformal and hybrid equivalence classes  and investigate  the properties of the Tandelbrot set, proving Theorem B.  
Section \ref{sec:straight} is devoted to the proof of Theorem A. In Section \ref{sec:tlfam}, we define tangent-like families and we prove Theorem E,  the parameter version of the straightening theorem.
Finally, Section \ref{sec:existtlfam} is devoted to  finding families of tangent-like maps in the parameter spaces of meromorphic maps, proving  Theorem C and Corollary D.

\medskip
\textbf{AI disclosure}

No AI tools were used for the mathematical content or  the writing of this paper.

\textbf{Acknowledgements}

\thanks{We are grateful to Lasse Rempe for pointing out the lack of unicity in the Straightening Theorem in an earlier version of the paper.  In the final stages of writing this manuscript, it was kindly brought to our attention by L. Rempe that R. Oudkerk claimed some of these results in a talk of 2006, and conjectured Theorem E. The slides of the talk, kindly provided by P. Rippon, contained  no proofs and we have not been able to find any manuscript or other source. At the same time, we learned from G. Levin of the paper \cite{LevinSwiatek}, where there is a proof a straightening theorem for a similar class of maps (without unicity).}

\textbf{Notations} 

Troughout the paper, $\C$ denotes the complex plane, $\D$ the unit disk, $\rs$ the Riemann sphere, $\H$ the upper half plane, and $S^1$ the unit circle. We denote by  $\Lambda$ a finite-dimensional complex connected manifold, which in some special cases will reduce to be just a domain in $\C$.

\section{Preliminaries}\label{sec:prelim}

\subsection{Holomorphic motions, quasiconformal maps}\label{sec:holomotions}

\ 

We start by collecting some technical facts about holomorphic motions and quasiconformal homeomorphisms, that will be used throughougt the paper.

\begin{defi}[Holomorphic motion]\label{def:holomotion}
	A {\em holomorphic motion} of a set $X \subset \hat\C$ over a complex connected manifold $\Lambda$ with basepoint $\la_0\in \Lambda$ is a map  $H: \Lambda  \times X \rightarrow \chat$ given by 
	$(\la,x) \mapsto H_\la(x) $ such that
	\begin{enumerate}
		\item  for each $x \in X$ , $H_\la(x)$ is holomorphic in $\la$, 
		\item   for each  $\la \in \Lambda$,  $H_\la(x) $ is an injective function of  $x \in X$, and,
		\item  at $\la_0$, $H_ {\la_0} \equiv {\rm Id}$.
	\end{enumerate}
	
\end{defi}

 Given a set $X'\subset X$ and   a family of holomorphic maps  $\{f_\la\}_{\la\in \Lambda}$ depending holomorphically on $\la$ and defined on domains $U_\la\supset H_\la(X')$,   we say that the holomorphic motion $H$ \emph{respects the dynamics} (of $\{f_\la\}$) on $X'$ if for all $\la\in\Lambda$,  $z\in X'$, 
$$
   f_\la\circ  H_\la(z)= H_\la\circ f_{\la_0}(z).
$$

We will make repeated use of the following classical extension theorems on holomorphic motions.

\begin{prop}[Mañé-Sad-Sullivan's $\la$-lemma, \cite{mss}]
	Let $H: U \times X \to \rs$ be a holomorphic motion of a set $X \subset \rs$ over $U$. Then $H$ extends uniquely to a holomorphic motion $H: U \times \overline{X} \to \rs$ of the closure of $X$ over $U$.
\end{prop}
Notice that, by continuity, if $H$ respects the dynamics of a holomorphic family $f_\la$ in $X$, so does its extension in $\overline{X}$.

 The following much deeper result is due to Slodkowski   (\cite{slo};  compare \cite{LyubichBook}, Fourth $\lambda$-Lemma)  which states that holomorphic motions can be extended to the whole Riemann sphere (although the extension is unrelated to any dynamical property).

\begin{theo}[Slodkowski's extension]\label{thm:Slodkowski}
	Let $H: U \times X \to \rs$ be a holomorphic motion of a set $X \subset \rs$ over a domain $U \subset \rs$. Then $H$ extends to a holomorphic motion $H: U \times \rs \to \rs$. Moreover, for all $\la \in U$, $H_\la:=H(\la, \cdot): \rs \to \rs$ is a quasiconformal homeomorphism. 
\end{theo}

We shall use the properties of the boundary extension of maps defined in domains.

\begin{lem}[{Boundary extension of conformal maps of $\D$ \cite[Sect. 5.1]{pom}}]\label{lem:quasidisk}
	Let $D \subset \C$ be a quasidisk, and let $\omega: \D \to D$ be a Riemann map. Then $\omega$ extends continuously to $S^1$, and $\omega: S^1 \to \partial D$	is quasisymmetric.
\end{lem}

It follows that if $D_1, D_2$ are quasidisks and $\phi:D_1\ra D_2$ is conformal, then $\phi$ extends quasisymmetrically as a map from $\partial D_1$ to $\partial D_2$.

\begin{lem}[{Interpolation in quasi-annuli \cite[Theorem 3.4]{EFGP}}]\label{lem:interpolation}
	Let $U,V, U', V'$ be Jordan domains, with $U \Subset V$ and $U' \Subset V'$. Let $\psi_2: \partial U \to \partial U'$ and $\psi_1: \partial V \to \partial V'$ be quasisymmetric homeomorphisms.
	Then there exists a quasiconformal homeomorphism 
	$\psi: \overline{V} \setminus U \to \overline{V'} \setminus U'$ 
	such that $\psi_{|\partial U} = \psi_2$ and $\psi_{|\partial V} = \psi_1$. Moreover, the quasiconformality constant of $\psi$ depends only on the moduli of the annuli $\overline{V} \setminus U$ and  $\overline{V'} \setminus U'$ {and the quasisymetric constants of $\psi_1$ and $\psi_2$}.
\end{lem}

\subsection{Meromorphic maps, families and bifurcations}\label{sec:natfam} \ 

 Recall that a meromorphic map is a holomorphic map  $f:\C\to\chat$ for which $\infty$ is an essential singularity. In this section we collect properties of natural families of meromorphic maps and results about their bifurcations. A standard reference for iteration of meromorphic maps is \cite{bergweiler} and stability results for natural families can be found in \cite{ABF,ABFAhlfors}.

A point $v \in \rs$ is a critical value if it is the image of a critical point, i.e. if there exists $c \in \C$ such that $f'(c)=0$ and $f(c)=v$. 

An asymptotic value  for a meromorphic map is  a  point $v \in \rs$  such that there exists a curve $\gamma: [0,+\infty) \to \C$ such that $\lim_{t \to +\infty} \gamma(t)=\infty$ and $\lim_{t \to +\infty} f \circ \gamma(t)=v$. 
However, since in this paper we will consider not only meromorphic maps but also maps which are defined only on proper subsets of $\hat\C$, we will need this  more general definition.

 \begin{defn}[Asymptotic values]\label{def:asymptotic}
  Given a holomorphic map $f: U \to V$ between Riemann surfaces, a point $v \in V$ is an asymptotic value for $f$ if there exists a continuous curve $\gamma: [0,+\infty) \to U$ such that $\gamma(t)$ leaves every compact of $U$ as $t \to +\infty$, but $\lim_{t \to +\infty} f \circ \gamma(t)=v$.
 \end{defn}

The singular value set of $f$, $S(f)$,  is the closure of the set of critical or asymptotic values. 
It may also be characterized as the smallest closed set $S \subset \rs$ such that $f: \C \setminus f^{-1}(S) \to \rs \setminus S$ is a covering map (more generally, $f: U\setminus f^{-1}(S) \to V\setminus S$ is a covering map).  We say that $v$ is an {\em isolated critical or asymptotic value} if $v$ is an isolated element in $S(f)$. 
We say that $f$ is  of {\em finite type} if $S(f)$ is finite. In that case, $S(f)$ contains only critical or asymptotic values.

\begin{defi}[Natural families]
	Let $\Lambda$ a finite dimensional complex connected manifold. A holomorphic family $\{f_\la\}_{\la \in \Lambda}$ of meromorphic maps is called a \emph{natural family} if there exists $\la_0 \in \Lambda$ and  quasiconformal homeomorphisms $\phi_\la, \psi_\la : \rs \to \rs$ depending holomorphically on $\la \in \Lambda$ such that  for all $\la \in \Lambda$, and
	$$f_\la = \phi_\la \circ f_{\la_0} \circ \psi_{\la}^{-1}.$$
\end{defi}

The maps $\phi_\la(z),\psi_\la(z)$ can be thought of as holomorphic motions of the Riemann sphere over $\Lambda$, with base point $\la_0$.

We will often use the following observation: if $\{f_\la\}_{\la \in \Lambda}$ is a natural family and $f_\la= \phi_\la \circ f_{\la_0} \circ \psi_{\la}^{-1}$ for some $\la_0 \in \Lambda$, then we can choose any $\la_1 \in \Lambda$ as a new basepoint without changing $f_\la$, by replacing $\phi_\la$ with $\phi_\la \circ \phi_{\la_1}^{-1}$ and $\psi_\la$ with $\psi_\la \circ \psi_{\la_1}^{-1}$. By another similar change we can also assume that for any basepoint $\lambda_0$, $\phi_{\la_0}= \psi_{\la_0}=\id$.

We note the following basic but important properties of natural families: if $f_\la= \phi_\la \circ f_{\la_0} \circ \psi_{\la}^{-1}$, then
\begin{enumerate}
	\item  $\psi_\la$ fixes $\infty$, since this is  the only essential singularity of $f_\la$ for all $\la$.
	\item $\psi_\la$ maps critical points of $f_{\la_0}$ to critical points of $f_\la$, preserving multiplicites. In particular, all critical points move holomorphically over $\Lambda$ and have constant multiplicity.
	\item $\phi_\la$ maps critical values of $f_{\la_0}$ to critical values of $f_\la$, and asymptotic values of $f_{\la_0}$ to asymptotic values of $f_\la$. In particular, singular values move holomorphically over $\Lambda$.
	\item 	Let $\mathcal{V}_\la$ denote the set of Picard exceptional values of $f_{\la}$, that is, the values which do not have infinitely many preimages. Then  $\mathcal{V}_\la=\phi_\la(\mathcal{V}_{\la_0})$. Indeed,  $f_{\la_0}(z)=w$ iff  $f_\la\circ \psi_\la(z)=\phi_\la(w)$, hence since $\psi_\la,\phi_\la$ are homeomorphisms, $w $ has infinitely many preimages under $f_{\la_0}$ iff $\phi_\la(w)$ has infinitely many preimages under $f_\la$.
\end{enumerate}

The following is a criterion for proving that a family of finite type meromorphic maps is natural.

\begin{theo}[Characterization of natural families, {\cite[Theorem 2.6]{ABF}}] \label{thm:natural}
	Let $\{f_\lam\}_{\lam \in\Lambda}$ be a holomorphic family of finite type meromorphic maps, for which
	$S(f_\lam)$ and $f_\lam^{-1}(S(f_\lam))$ both move holomorphically. Then for every $\lam \in\Lambda$, there is a neighborhood $V$ of $\lam$ such that $\{f_\lam\}_{\lam \in V}$ is a natural family.
\end{theo}

Meromorphic maps with asymptotic values may have a special type of cycles which play a special role in the dynamics of the map. 

\begin{defi}[Virtual cycle]\label{def:virtualcycle}
	Let $f: \C \to \chat$ be a meromorphic map. 
	A {\em virtual cycle} of length $n$ is a finite, cyclically ordered set $\{z_0, z_1, \ldots, z_{n-1}\}$   such that for all $i$, either $z_i \in \C$ and $z_{i+1}=f(z_i)$ (with indices taken modulo $n$),
	or $z_i=\infty$ and $z_{i+1}$ is an asymptotic value for $f$, and at least one of the $z_i$ is equal to $\infty$.  If $z_i=\infty$ only for one value of $i$ then we say that the virtual cycle has {\em minimal length} $n$.
\end{defi}

 If a virtual cycle remains after perturbation within the family, then it is called {\em persistent}. 
If a virtual cycle for $f_{\la_0}$ is non-persistent, we will say that  $\la_0$ is  a {\em virtual cycle parameter}.

Our next  definition concerns the concept of activity or passivity of a singular value
(see also \cite{Levin}, \cite{mcmbook}).

\begin{defi}[Passive and active singular value]\label{defn:active singular values}
	Let $\{f_\lam\}_{\lam \in\Lambda}$ be a holomorphic family of 
	 meromorphic maps. Let $v(\lam)$ be a singular value (or a critical point) of $f_\lam$ 
	depending holomorphically on $\lam$ near some $\lam_0 \in\Lambda$.
	We say that $v(\lam)$ is \emph{passive} at $\lam_0$ if there exists a neighborhood $V$ of $\lam_0$ in $\Lambda$ such that:
	\begin{enumerate}
		\item	either $f_\lam^n(v(\lam))=\infty$ for all $\lam \in V$; or
		\item the family
		$\{\lam \mapsto f_\lam^n(v(\lam)) \}_{n \in \N}$ is well-defined and normal on $V$.
	\end{enumerate}
	We say that  $v(\la)$ is {\em active} at $\la_0$ if it is not passive. 
\end{defi}

 A parameter $\lambda_*$ is  $\JJ$-stable if there is a neighborhood of $\la_*$  over which the Julia set moves holomorphically with a holomorphic motion that respects the dynamics (see Definition~\ref{def:holomotion}). 

For natural families of meromorphic maps of finite type the following are equivalent definitions of $J-$stability.

\begin{theo}[Characterizations of $\JJ$-stability, {\cite[Theorem E]{ABF}}]\label{thm:J stability}
	Let $\{f_\lam\}_{\lam \in\Lambda}$ be a natural family of finite type meromorphic maps.
	Let $U \subset \Lambda$ be a simply connected domain of parameters. The following are equivalent:
	\begin{enumerate}[\rm (1)]
		\item The Julia set moves holomorphically over $U$ { with a holomorphic motion that respects the dynamics} 
		(i.e. $f_\la$ is $\JJ$-stable for all $\la\in U$)
		\item Every singular value is passive on $U$.
		\item  The maximal period of attracting cycles is bounded on $U$.
		\item  The number of attracting cycles is constant in $U$.
		\item  For all $\la\in U$, $f_\la$ has no non-persistent parabolic cycles.
	\end{enumerate}
\end{theo}

In \cite[Proposition 5.5]{ABF} it is shown that parameters with virtual cycles are dense. We will use this result several times.
\begin{prop}[Density of parameters with virtual cycles]\label{prop:density_of_virtual cycles}
	Let $\{f_\lam\}_{ \in\Lambda}$ be a natural family of finite type meromorphic maps, and assume that there exists $\lam_* \in\Lambda$ 
	such that $f_{\lam_*}$ has at least one non-omitted pole. Assume that an asymptotic value is active at some $\lam_0 \in\Lambda$.
	Then $\lam_0$ can be approximated by virtual cycle parameters of arbitrarily large order.
\end{prop}
In \cite[Proposition 2.23]{ABFAhlfors} is proven a version of  Proposition~\ref{prop:density_of_virtual cycles} which applies to  meromorphic maps (even if they are not of finite type). We will use this in Section~\ref{sec:tlfam}.

Misiurewicz parameters are also dense (\cite{ABF}, Corollary 5.6). Recall that for a given $\la$, a singular value $v_\la$ is Misiurewicz if there exists $k$ such that $f_\la^k$ is a  repelling periodic point,  non-persistently in $\la$. 
\begin{prop}[Density of Misiurewicz parameters]\label{prop: density Misiurewicz}
	Let $\{f_\lam\}_{ \in\Lambda}$ be a natural family of finite type meromorphic maps, and assume that there exists $\lam_* \in\Lambda$ 
	such that $f_{\lam_*}$ has at least one non-omitted pole.  Assume that a singular value $v_\la$ is active at some $\lam_0 \in\Lambda$. Then $\la_0$ can be approximated by  parameters for which $v_\la $ is Misiurewicz. 
 \end{prop}

Finally, we will make use several times of the following Shooting Lemma (\cite[Proposition 2.13]{ABF}).

\begin{prop}[Shooting Lemma]\label{shooting}
Let $\{f_\la\}_{\la\in\Lambda}$ be a natural family of meromorphic maps.\ 
Let $\la_0\in\Lambda$  be such that a singular value $v_\la$ satisfies $f_{\lam_0}^n(v_{\la_0})=\infty$,  but this relation is not satisfied for all $\la\in\Lambda$. 
 Let $\lambda \mapsto \gamma(\la)$ be a  holomorphic map  such that $\gamma(\la_0) \notin S(f_{\la_0})$. Then we can find $\la'$ arbitrarily close to $\la_0$  such that
$f_{\lam'}^{n+1}(v_{\la'}) = \gamma(\la')$.
\end{prop}

\subsection{Tracts over asymptotic values}\label{sec:tracts}

In our proofs it will be important to control the size of {\em tracts} and how they move with respect to parameters. 

\begin{defi}[Tracts]\label{defi:tracts}
	Let $f: \C \to \hat{\C}$ be a meromorphic map,  let $v \in \hat{\C}$ be an asymptotic value of $f$, and $D$ be a  punctured Jordan neighborhood of $v$. We say that a simply connected unbounded set $T$ is a \emph{(logarithmic) tract} above $D$  if $f:T\ra D $ is an  infinite degree unbranched covering.
\end{defi}
A tract above the asymptotic value $v$ is a  tract above any  such a domain $D$.  
Notice that by taking a nested decreasing sequence of Jordan domains  $D_i$ shrinking to $v$,  we obtain a nested sequence of logarithmic tracts $T_i$ shrinking to infinity. We say that two logarithmic tracts (or sequences thereof) are different, if they are disjoint. 
 The number of different logarithmic tracts lying over an asymptotic value $v$ is the {\em multiplicity} of $v$.

If $v$ is an isolated asymptotic value one can see that there is always a logarithmic tract above $v$ (see e.g. \cite{BE08}). In particular, if $f$ is of finite type all of its asymptotic values have logarithmic tracts lying over them. For this reason, in this paper we will call them simply {\em tracts}. 

We see next that more can be said about tracts if we allow no singular values in the closure of their images.
\begin{lem}[Jordan tracts]\label{lem:jordantracts}
Let $f: \C \to \hat{\C}$ be a meromorphic map,  let $v \in \C$ be an asymptotic value of $f$ and $V$ be a punctured Jordan neighborhood  of $v$ (which may be unbounded) such that $\overline{V}\setminus\{v\}$ (in $\chat$) has no asymptotic values. Let $T\subset f^{-1}(V)$ be a tract over $v$. Then $T$ is a Jordan domain.
\end{lem}
\begin{proof}
Since $\partial V$  is a simple closed curve and $T$ is simply connected, all boundary components (in $\C$) of $T$  must be unbounded.
Let $\riem_V:V\to \D^*$ be a conformal map. Then, by uniqueness of universal coverings, there exist a conformal map $\riem_T$ from $T$ to the left half plane such that, in $T$, 
\[
f=\riem_V^{-1}\circ \exp\circ \riem_T.
\]
Since $\partial T$ (in $\chat)$ is locally connected, $\riem_T^{-1}$ extends continuously to $i \R\cup\infty$. If $\partial T$ has more than one component, then there exists at least one $z\in i\R$ such that $\riem_T^{-1}(z)=\infty$. This implies that a continuous curve $\gamma(t)\subset \partial T$ such that $\gamma(t)\to \infty$, satisfies that $\exp\circ \riem_T(\gamma(t))\to e^z \in\partial \D$. Hence $f(\gamma(t)) \to  w\in\partial V$, where $w=\riem_V^{-1}(e^z)$, which means that $w$ is an asymptotic value, a contradiction. 
\end{proof}

The following properties of tracts for natural families will be used only in Section~\ref{sec:tlfam} for the proof of Theorem C.

\begin{lem}[Tracts move holomorphically]\label{lem:tractmoveholo}
	Let $\{f_\la\}_{\la \in \Lambda}$ be a natural family of meromorphic maps, where $\Lambda\subset \C$ is a simply connected domain. Let $v_\la$ be an asymptotic value for $f_\la$, and assume that there exists a holomorphic motion $w_\la: D_{\la_0} \to D_\la$ over $\Lambda$ of a Jordan domain $D_{\la_0}$ such that for all $\la \in \Lambda$, $D_\la \cap S(f_\la) = \{v_\la\}$ and $w_\la(v_{\la_0})=v_\la$. Let $T_{\la_0}$ denote a tract above $D_{\la_0}$ for $f_{\la_0}$. Then there is a holomorphic motion $m_\la: T_{\la_0} \to T_\la$
	over $\Lambda$ such that for all $\la \in \Lambda$, $T_\la$ is a tract above $D_\la$ for $f_\la$ and 
	$$f_\la \circ m_\la = w_\la \circ f_{\la_0}.$$
\end{lem}
\begin{proof}
	(See also the proof of  Lemma \ref{lem:liftholmot}.)
	We may write $f_\la = \phi_\la \circ f_{\la_0} \circ \psi_\la^{-1}$, 
	where $\phi_\la, \psi_\la$ are holomorphic motions of $\rs$, 
	with basepoint $\la_0$. We will define $m_\la(z_0)$ for any point  $z_0 \in  T_{\la_0}$ and show that the resulting map is holomorphic in $\la$ for any $z_0\in T_{\la_0}$ and injective in $z$ for any $z\in T_{\la_0}$. 
Consider the 	continuous map $ H_{z_0}:\Lambda\ra\C$ defined as $ H_{z_0}(\la )=\phi_\la^{-1} \circ w_\la \circ f_{\la_0}(z_0).$
	
	By the lifting property for covering maps, there exists a unique lift $\gamma_{z_0}(\la)$  of $ H_{z_0}$ by $f_{\la_0}$ 
	such that $\gamma_{z_0}(\la_0)=z_0$.

	\[
\begin{tikzcd}[row sep=10ex, column sep=10ex]
&  \C \setminus f_{\la_0}^{-1} (S(f_{\la_0})) 
 \arrow[d,"f_{\la_0}"] \\
\Lambda \arrow[r, " H_{z_0}"]  \arrow[ur, "\gamma_{z_0}" ] 
 & \C\setminus S(f_{\la_0}).
\end{tikzcd}
\]

	Indeed, since $\Lambda$ is simply connected, we only need to check that $v_{\la_0}\notin  H_{z_0}(\Lambda)$, so that $f_{\la_0}$ is a covering over $ H_{z_0}(\Lambda)$.	This is true because by assumption $\phi_\la^{-1}\circ w_\la (v_{\la_0})= v_{\la_0}$, $\phi_\la^{-1}\circ w_\la$ is injective, and $f_{\la_0}(z_0)\neq v_{\la_0}$. 
By definition of lifting, $f_{\la_0}\circ\gamma_{z_0}(\lambda)= H_{z_0}(\la)=\phi_\la^{-1}\circ w_\la\circ f_{\la_0}(z_0)$.

	We then let $m_\la(z_0):=\psi_\la\circ \gamma_{z_0}(\la)  $. By construction, we have {
	\begin{equation}\label{eq:commute}
		f_\la \circ m_\la(z_0) 
		= f_\la\circ \psi_\la \circ \gamma_{z_0}(\la) 
		=\phi_\la \circ f_{\la_0} \circ \gamma_{z_0}(\la)
		 =\phi_\la \circ  H_{z_0} 
		= w_\la \circ f_{\la_0}(z_0)
	\end{equation}
	and  $m_{\la_0}(z_0)=z_0$ since $\psi_{\la_0}={\rm id}$}.
	By construction, $\gamma_{z_0}(\la) \notin f_{\la_0}^{-1}(S(f_{\la_0}))$, for all $\la \in \Lambda$. In particular, it is never a critical point of $f_{\la_0}$.
	Since $\psi_\la$ maps critical points of $f_{\la_0}$ to critical points of $f_\la$, we therefore have $f_\la'(m_\la(z_0)) \neq 0$ for all $\la \in \Lambda$.
	By the Implicit Function Theorem, we have that  $\la \mapsto m_\la(z_0)$ is holomorphic in $\Lambda$.
	By the uniqueness of the lifting property, the map $z \mapsto m_\la(z)$ is injective for all $\la \in \Lambda$.
	Thus, $m_\la: \Lambda \times T_{\la_0} \to \rs$ is indeed a holomorphic motion, and it follows from \eqref{eq:commute} that $T_\la:=m_\la(T_{\la_0})$ is a tract for $f_\la$ above $D_\la$.	
\end{proof}

\begin{lem}[Tracts above small disks are small]\label{lem:smalltracts}
	Let $\{f_\la\}_{\la\in \Lambda}$ be a natural family of  meromorphic maps and let $v_\la$ be an isolated asymptotic value.
	For any $\delta>0$, we denote by $T_\la(\delta)$ a nested family of tracts above $\D(v_\la,\delta)$, i.e. $T_\la(\delta_1) \subset T_\la(\delta_2)$ if $0<\delta_1 \leq \delta_2$.
	 Let $\la_0 \in \Lambda$. For every $r>0$, there exists $\delta>0$ and $\eta>0$ such that for any $\la \in \B(\la_0,\eta)$,  any tract $T_\la(\delta)$ above $\D(v_\la,\delta)$ for $f_\la$, is compactly contained in $\D(\infty,r)$.
\end{lem}

\begin{proof}
	We first prove the statement for   $f=f_{\la_0}$ and $v=v_{\la_0}$.  
	Suppose  for a contradiction that there exists $r>0$,  a decreasing sequence $\delta_n \to 0$,  a sequence of tracts $T_n:=T_{\la_0}(\delta_n)$ over $\D(v,\delta_n)$   and points $z_n \in T_n$ such that for all $n$,  the spherical distance satisfies $d(z_n,\infty) \geq r$. Up to extracting a subsequence, we may assume that $z_n \to z \in \C$. 
	Moreover, since the tracts are nested, we have $z \in T_n\subset T_0$ for all $n \in \N$, hence $z\in \ov{T_0}\cap \C$. Since  	$f(z_n) \in \D(v,\delta_n)$, by continuity  $f(z)=v$, which 
	 gives a contradiction because $f(\ov{ T_0}\cap\C)= \ov{\D(v,\delta_0)}\setminus \{v\}$.
	 
	The statement  for $\lambda$ in a neighborhood of $\lambda_0$ then follows from the fact that tracts move holomorphically with the parameters (Lemma~\ref{lem:tractmoveholo} with $w_\la(z)=z+v_\la-v$).
\end{proof}

\section{Tangent-like maps} \label{sec:TLmaps}

In this section we introduce tangent-like maps, filled Julia sets,   and  the notion of hybrid conjugacy.  We then show that the filled Julia set is connected if and only if it contains the singular value.

The definition of tangent-like maps (TL for short)  first appeared in \cite{GaKo}.

\begin{defi}[Tangent-like map]\label{def:tlmap}
	A tangent-like map is a tuple $(f,U,V,u,v)$ such that 
	\begin{enumerate}
		\item $U,V \subset \rs$ are Jordan domains, and $\overline{U} \subset V$
		\item $u \in \partial U$, $v \in V$
		\item $f: \overline{U}\setminus \{u\} \to \overline{V} \setminus \{v\}$ is a covering map
		\item $f: U \to V$ is holomorphic
		\item $\partial U$ is a quasicircle and $\partial V$ is real-analytic. 
	\end{enumerate}
\end{defi}

See Figure \ref{fig:TLmap}. We refer to the closed annulus $\ov{V}\setminus U$ as the \emph{fundamental annulus}. 
Since $U$ is simply connected and $V\setminus \{v\} \cong \D^*$, it follows from  the classification of coverings of  $\D^*$ that $f:U\ra V\setminus\{v\}$ is equivalent to the exponential map, so $u$ is an essential singularity (in the sense that $f$ cannot be continuously extended at $u$), $v$ is an asymptotic value, and   $f$ has infinite degree.  
Note that, up to restricting  $V$ slightly, we can always assume that $\partial V$ is real-analytic; this will imply that $\partial U$ is analytic except possibly at $u$.

\begin{rem}[Maps whose tracts are quasidisks]Let $f:=R_2\circ g\circ P_1$ with $P_1$ a polynomial and $R_2$ a rational map, and such that, for $g$,  tracts over quasidisks containing isolated singular values  are quasidisks. 
Then, for $f$, tracts over quasidisks  containing  isolated singular values are quasidisks. 
\end{rem}

\begin{defi}[Julia set and filled Julia set] \label{def:filledJset}
	Let  $(f,U,V,u,v)$ be a tangent-like map. The filled Julia set is 
	$$K(f):=\bigcap_{n \geq 0} \overline{ f^{-n}(U) }.$$
	The Julia set is $J(f):=\partial K(f)$.
\end{defi}

\begin{rem}
	By definition, we always have $u \in J(f)$. Indeed, $f$ is equivalent to an exponential map where $u$  is an essential singularity, hence  for any $z \in V$,  $z\neq v$, $u \in \overline{f^{-1}(\{z\})}$, and $u \in \partial U$ so $u \in \partial K(f)$.
\end{rem}

The fact that $f:U\ra V$ is neither surjective nor proper, and does not extend continuously at $u$, calls for a little caution in the proof of the next proposition.

\begin{prop}[Complete invariance of $K(f)$]\label{lem:K}
	The set  $K(f)$ is completely invariant, in the sense that $f^{-1}(K(f)) = K(f) \setminus \{u\}$. Moreover, we have
	$$K(f) = \left(\bigcup_{n \geq 0} f^{-n}(\{u\}) \right) \bigsqcup \left( \bigcap_{n \geq 0} f^{-n}(U) \right).$$
\end{prop}

Of course, since the map $f$ is not defined at $u$, $u$ cannot belong to $f^{-1}(K(f))$.

\begin{proof}
	Let us first prove that $f^{-1}(K(f)) \subset  K(f) \setminus \{u\}$. Let $z \in f^{-1}(K(f)) \subset \overline{U} \setminus \{u\}$, so that $f(z) \in K(f)$. Then by definition, for every $n \in \N$, there exists a sequence 
	$y_{k,n}$ such that $\lim_{k \to +\infty} y_{k,n}=f(z)$, and $f^n(y_{k,n}) \in U$.
	 Since $f: U \to V \setminus \{v\}$ is a covering map and $f(z)\neq v$, there is  an inverse branch $g$ of $f$ mapping $f(z)$ to $z$.
	For $k$ large enough, $y_{k,n}$ is in the domain of $g$, and we let $x_{k,n}:=g(y_{k,n})$. Then 
	$\lim_{k \to +\infty} x_{k,n}=z$, and $f^{n+1}(x_{k,n}) = f^{n}(y_{k,n}) \in U$. Therefore $z \in \bigcap_{n \geq 0} \overline{f^{-n}(U)} = K(f)$.

	Conversely, let $z \in K(f) \setminus \{u\}$. By definition, for every $n  \geq 1$ there exists a sequence $x_{k,n}$ 
	such that $\lim_{k \to +\infty} x_{k,n}=z$, and $f^{n}(x_{k,n}) \in U$. Let $y_{k,n}:=f(x_{k,n})$: then 
	$\lim_{k \to +\infty} y_{k,n}=f(z)$, and $f^{n-1}(y_{k,n}) \in  U$. This proves that $f(z) \in K(f)$, and therefore that $f^{-1}(K(f)) \supset  K(f) \setminus \{u\}$	as required.

	We now turn to the second assertion in the Lemma. Let $z \in K(f)$, such that $z \notin \bigcap_{n \geq 0} f^{-n}(U)$.
	Then again, by definition of $K(f)$, for every $n  \geq 0$ there exists a sequence $x_{k,n}$ 
	such that $\lim_{k \to +\infty} x_{k,n}=z$, and $f^{n}(x_{k,n}) \in U$. Moreover, there exists $n_0 \in \N$ such that 
	$f^{n_0}(z) \notin U$. Let us prove that $f^{n_0}(z) = u$. 
	First observe that $f^{n_0}(z) \in \partial U$. Indeed, otherwise we would have $f^{n_0}(z) \in V \setminus \overline{U}$ which is open, so for $k$ large enough we would have $f^{n_0}(x_{k,n_0}) \in V \setminus \overline{U}$, a contradiction. Next, if $f^{n_0}(z) \in \partial U \setminus \{u\}$, then we would have $f^{n_0+1}(z) \in \partial V$, and so by the same reasoning, $f^{n_0+1}(x_{k,{n_0+1}}) \notin U$ for $k$ large enough, again a contradiction.	Therefore $f^{n_0}(z) = u$. This proves that $K(f) \subset  \left(\bigcup_{n \geq 0} f^{-n}(\{u\}) \right) \sqcup \left( \bigcap_{n \geq 0} f^{-n}(U) \right).$ Finally, the inclusion 
	$ \left(\bigcup_{n \geq 0} f^{-n}(\{u\}) \right) \sqcup \left( \bigcap_{n \geq 0} f^{-n}(U) \right) \subset K(f)$ follows from the relation $f^{-1}(K(f)) =  K(f) \setminus \{u\}$ and the fact that $u \in K(f)$.
\end{proof}

\begin{rem}
	In particular, there is  exactly a countable set of points in $K(f)$ with finite orbit (which act as prepoles), unless $u=v$, in which case 
	there are none.
\end{rem}

The concept of hybrid conjugacy can be defined in this context analogously as for polynomial-like maps.

\begin{defi}[hybrid conjugate tangent-like maps]
	Two tangent-like maps $(f_1, U_1,V_1, u_1, v_1)$ and $(f_2, U_2, V_2, u_2, v_2)$ are hybrid conjugate if there exists a quasiconformal homeomorphism $\conj: V_1 \to V_2$ such that for all $z \in U_1$, 
	$$
	\conj \circ f_1(z) = f_2 \circ \conj(z)
	$$
	and $\dbar \conj=0$ a.e. on $K(f_1)$.
\end{defi}

As in the quadratic-like case, the topology of $K(f)$ is determined by the dynamics of the singular value. 

\begin{prop}\label{prop:Kconnected}
	Let $(f,U,V,u,v)$ be a tangent-like map. Then $K(f)$ is connected if and only if $v \in K(f)$.
\end{prop}

\begin{rem*}
In fact, if $v\notin K(f)$, then $K(f)$ is totally disconnected. This will follow from the Straightening theorem and the analogue result for the model family proven in \cite[Lemma 3.3]{GaKo2}, although it could probably be proven directly.  
\end{rem*}

\begin{proof}
	Let us first assume that $v \in K(f)$  and let us prove that $K(f)$ is connected.  Either 
for all $n \in \N$, $f^n(v) \neq u$, or there exists $n_0 \in \N$ such that $f^{n_0}(v)=u$.
In the first case, by Lemma \ref{lem:K}, $v \in f^{-n}(U)$ for all $n \in \N$. We prove by induction on $n$ that   $f^{-n}(U)$ is a simply connected domain. For $n=0$, $U$ is simply connected by definition; now let $n \in \N$ and assume that $f^{-n}(U)$ is a simply connected domain.  Then $f^{-(n+1)}(U)$ is a logarithmic tract above $f^{-n}(U)$ (since $v \in f^{-n}(U)$). Therefore, $f^{-(n+1)}(U)$ is again simply connected.
	The closure of a connected set is connected, and the decreasing intersection of connected compact sets is connected.
	Therefore $K(f) = \bigcap_{n \geq 0} \overline{f^{-n}(U)}$ is connected.
	
	\begin{figure}[htb!]
	\centering
	\includegraphics[height=0.4\textwidth]{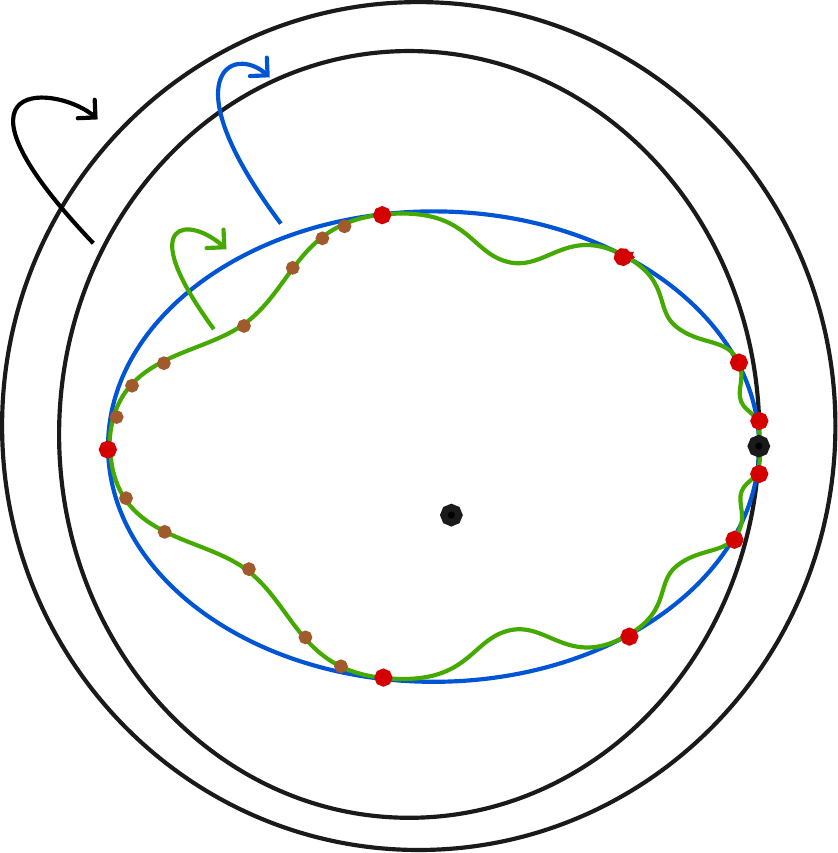} 
	\hspace*{0.05\textwidth}
	\includegraphics[height=0.4\textwidth]{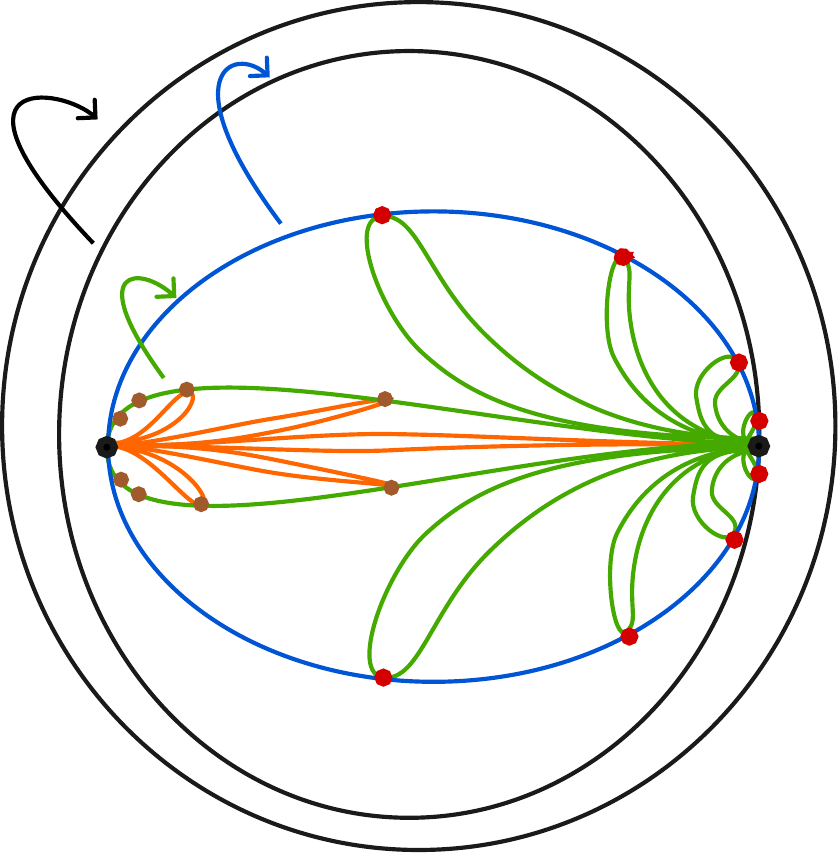}
	\setlength{\unitlength}{0.85\textwidth}
	\put(-0.028,0.22){$u$}
	\put(-0.57,0.22){$u$}
	\put(-0.425,0.22){$v$}
	\put(-0.74,0.18){$v$}
	\put(-0.05,0.39){\large $V$}
	\put(-0.07,0.35){\large $U$}
	\put(-0.592,0.39){\large $V$}
	\put(-0.612,0.35){\large $U$}
	\put(-0.475,0.405){$f$}
	\put(-1.017,0.405){$f$}
	\caption{\label{connected case} \small Sketch of the proof of Proposition \ref{prop:Kconnected}, in the case when $v\in K(f)$. In blue, $f^{-1}(\partial U)$, an unbounded  Jordan curve which contains all poles (preimages of $u$) of $f$ (red dots). In green, $f^{-2}(\partial U)$, which contains all prepoles of order one or two of $f$ (red and brown dots).  Left: $f^n(u)\neq v$ for all $n\geq 0$, hence $f^{-n}(\partial U)$ is a single (unbounded) Jordan curve.  Right: $f(v)=u$ ($n_0=1$). Hence $f^{-2} (\partial U)$ consists of infinitely many unbounded curves mapping one-to-one to $\partial U$. In orange, $f^{-3}(\partial U)$, illustrating the pattern that repeats inside every preimage of $U$.}
	\end{figure}
	
	Let us now assume that there exists $n_0 \in \N$ such that $f^{n_0}(v)=u$. By the same argument as above, we have that for all $n \leq n_0$,
	$f^{-n}(U)$ is simply connected. But now $f^{-n_0}(U)$ is simply connected, and $v$ belongs to the boundary of $ f^{-n_0}(U)$
	; so $f^{-n_0-1}(U)$ is a countable union of disjoint simply connected sets, which also do not contain $v$. Moreover, the closure of each connected component of $f^{-n_0-1}(U)$ contains $u$.
	By a quick induction on $n$, we have that for all $n > n_0$, $f^{-n}(U)$ is a disjoint union of countably many simply connected domains, each having $u$ in their closure. In particular, even though $f^{-n}(U)$ is disconnected, $\overline{f^{-n}(U)}$ is connected. Since the decreasing intersection of a sequence of connected sets is connected, $K(f)$ is again connected.
	
	Conversely, assume that $v \notin K(f)$. Then  there exists $n_0 \in \N$ (minimal) such that $v \notin  \overline{f^{-n_0}(U)}$.  By the same reasoning as above, $f^{-n_0}(U)$ is simply connected, and every connected component of  $f^{-(n_0+1)}(U)$ maps one to one to $f^{-n_0}(U)$; however in this case, $v $ is not in the boundary of $f^{-n_0}(U)$, so the aforementioned connected components do not contain $u$ in their boundary. 
	In particular,  $\ov{f^{-(n_0+1)}(U)}$ is disconnected. Since each component of $f^{-(n_0+1)}(U)$ (open) contains some Julia set (closed) (by backward invariance), and since $J \subset f^{-(n_0+1)}(U)$, the Julia set is disconnected.
\end{proof}
\begin{figure}[htb!]
\centering
\includegraphics[width=0.4\textwidth]{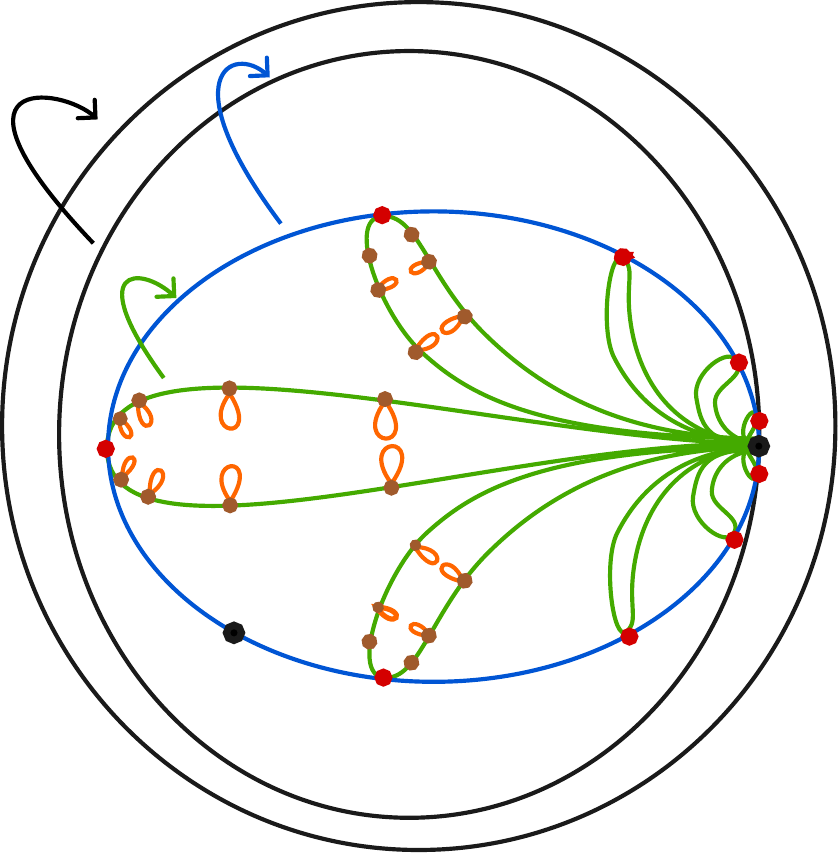} 
\hspace*{0.05\textwidth}
\includegraphics[width=0.4\textwidth]{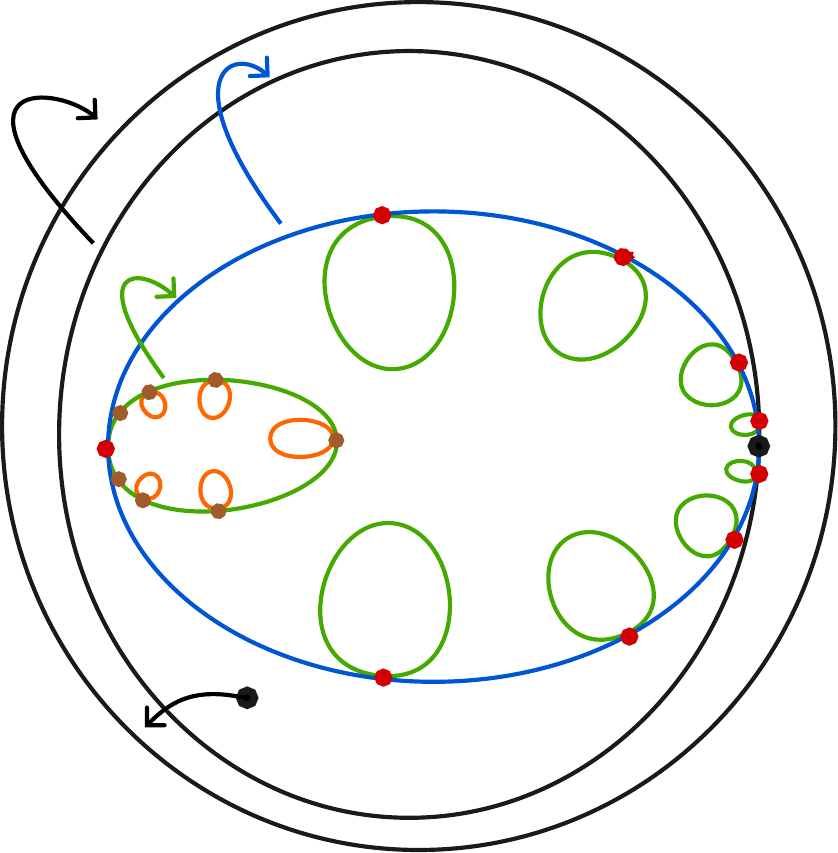} 
	\setlength{\unitlength}{0.85\textwidth}
	\put(-0.028,0.22){$u$}
	\put(-0.58,0.22){$u$}
	\put(-0.9,0.1){$v$}
	\put(-0.32,0.08){$v$}
	\put(-0.05,0.39){\large $V$}
	\put(-0.07,0.35){\large $U$}
	\put(-0.592,0.39){\large $V$}
	\put(-0.616,0.35){\large $U$}
	\put(-0.485,0.405){$f$}
	\put(-1.037,0.405){$f$}
\caption{\label{disconnected case} \small Sketch of the proof of Proposition \ref{prop:Kconnected}, in the case  $n_0=2$, that is when $f^2(v)\notin \overline{U} $. Left: $f(v) \in \partial{U}$. 
Right: $f(v)\notin \overline{U}$. }
\end{figure}

\section{The model family}\label{sec:model}
Quadratic polynomials can be  viewed  as the slice of $\operatorname{Rat}_2$ (degree two rational maps) for which one of the two critical points is superattracting. In analogy to this,  we consider  the family $\mathbb{F}_2$ of  meromorphic maps with exactly two omitted asymptotic values: see the discussion in \cite[Appendix]{FK} and also \cite{DeKePS,KK,EreGab}) and the one-dimensional slice  therein of maps  with an attracting fixed point of constant multiplier. 
  For simplicity purposes we fix the value of the multiplier to be $1/8$.

Maps in this slice are of the form 
	\begin{equation} \label{model}
	T_\alpha(z):= \frac{e^{(\alpha-1)z/8}-1}{\alpha e^{(\alpha-1)z/8}-1},
	\end{equation}
where $\alpha \in \C \setminus \{1\}$.

In a way,  $\{T_\alpha\}_{\alpha\in\C}$ is the simplest family whose members are naturally tangent-like maps, in the same way that quadratic polynomials are the simplest functions which are naturally quadratic-like maps.

In this section, we shall study the general properties of the maps $T_\alpha$, both in dynamical and in parameter plane. In the latter, we shall introduce the {\em Tandelbrot set}, an infinite analogue of the Mandelbrot set.

\begin{rem*} The slice(s) that we consider in this paper, up to taking a different normalization, have been extensively studied in \cite{chen2022accessible, chen2023slices, GaKo, GaKo2}, where a detailed combinatorial description of both dynamical and parameter plane is given. There is little overlap with the results in the present paper; we indicate explicitely where such overlap occurs.

\end{rem*}

\subsection{Basic properties} \label{sec:firstproperties} \

Let $T_\alpha$ be defined as in (\ref{model}).  For $\alpha \in \C \setminus \{1\}$, the map $T_\alpha$ has two asymptotic values, $1$ and $\frac{1}{\alpha}$.
It also fixes $0$ with multiplier equal to $\frac{1}{8}$, hence at least one of the two asymptotic values must always belong to the basin of $0$. In particular, their activity loci are disjoint. Finally, observe that for $\alpha=0$, 
$T_0(z)=1-e^{-z/8}$ is an exponential map of disjoint type  (i.e. the Fatou set of $T_0$ is connected and consists of a completely invariant basin of attraction of a fixed point). This is the only parameter value for which $T_\alpha$ is an entire map.

We shall see that when $\alpha \in \overline{\D} \setminus \{1\}$, the asymptotic value $1$ is always contained in the basin of attraction of $0$, so that $1/\alpha$ acts as a {\em free} asymptotic value.

\begin{lem}[{$1/\alpha$ is free if $\alpha$ is small}]\label{lem:disk mapping inside}
	For all $\alpha \in \overline{\D}\setminus\{1\}$, we have $T_\alpha(\D(0,2)) \subset \D \Subset \D(0,2)$. Therefore, $\overline{\D}(0,2)$ (and in particular the asymptotic value $1$) lie in the basin of attraction of $z=0$. 
\end{lem}

\begin{proof}
	Let $\alpha \in \overline{\D}\setminus\{1\}$. Then $T_\alpha$ has no singular values in $\D$, since its singular values are
	$1$ and $\frac{1}{\alpha}$. Therefore, there exists a well-defined univalent inverse branch $g$ of $T_\alpha^{-1}$ fixing $0$ and defined on $\D$. We have $g'(0)=8$; thus, by Koebe's distortion theorem, 
	$g(\D) \supset \D(0,2)$. In particular, this implies that $T_\alpha( \D(0,2) ) \subset \D$ and therefore that 
	$T_\alpha(\overline{\D}(0,2)) \subset \overline{\D} \Subset \D(0,2)$.
\end{proof}

\begin{rem}
The lemma above makes it is convenient to work with modulus of the multiplier of $0$ less than $1/8$, since the disc $\D(0,2)$ maps compactly into itself {\em independently of $\alpha$}. It is plausible that all subsequent arguments would work with larger multipliers, considering different Jordan domains. 
\end{rem}

Next, we prove that our family is characterized by its dynamical features: up to an affine conjugacy, any meromorphic map with exactly two singular values and a fixed point of multiplier $\frac{1}{8}$ belongs to the model family.
Additionally, the parameter space is invariant under the involution $\alpha \mapsto \frac{1}{\alpha}$.

\begin{lem}[Affine conjugacy classes] \label{lem:rigidity}
	Any meromorphic transcendental map with  exactly two singular values 	and a fixed point with multiplier equal to $\frac{1}{8}$ is affinely conjugate to a map $T_\alpha$. Moreover, if $\alpha_1, \alpha_2\in \C \setminus \{0,1\}$, then $T_{\alpha_1}$ is affinely conjugate to $T_{\alpha_2}$  if and only if $\alpha_2=\frac{1}{\alpha_1}$ or $\alpha_2=\alpha_1$.
\end{lem}

\begin{proof}
Let $\mu,\la\in \C$ denote the two distinct singular values and $M(z):=\frac{z-\mu}{z-\la}$. Then, the map $M\circ f: \C\setminus \{f^{-1}(\mu),f^{-1}(\la)\} \longrightarrow \C\setminus\{0,\infty\}$ is a covering.
Coverings of $\C^*$ are either equivalent to $z^d$ and then they have finite degree, or to the exponential map in which case the domain is simply connected.

In this case, since the degree is infinite, the domain is simply connected, i.e. $0$ and $\infty$ are omitted by $M\circ f$, and $M\circ f$ is a universal covering of $\C\setminus \{0,\infty\}$. By uniqueness of universal coverings, $f=M^{-1}(\exp (L(z))$, where $L$ is a conformal map of $\C$ and hence afine.

Thus $f(z)=\frac{\la e^{az+b}-\mu}{e^{az+b}-1}$, for some $a\in\C^*$, $b\in\C$.  Let $A$ denote an affine map sending $\mu$ to 1 and the fixed point of $f$ with multiplier $1/8$ to 0. It is then straightforward to check that $ A\circ f\circ A^{-1}=T_\alpha$, where $\alpha=1/\la$.

Finally, if $\la=\infty$ or $\mu=\infty$, the same procedure results in $T_0(z)$.

Let us now prove that for all $\alpha \in \C \setminus \{0,1\}$, $T_\alpha$ is linearly conjugate to $T_{1/\alpha}$ by the linear map $L(z):=\alpha z$. Indeed,
\[
	L \circ T_\alpha \circ L^{-1}(z)=  \alpha \frac{e^{z(\alpha-1)/(8\alpha)}-1}{\alpha e^{z(\alpha-1)/(8\alpha)}-1 }
	=\frac{e^{(1-1/\alpha)z/8} -1 }{e^{(1-1/\alpha)z/8} - \frac{1}{\alpha}} 
	=\frac{e^{(1/\alpha-1)z/8}-1}{\frac{1}{\alpha}e^{(1/\alpha-1)z/8} -1}
	=T_{1/\alpha}(z).
\]

Finally, let us prove that if  $T_{\alpha_1}$ is affinely conjugate to $T_{\alpha_2}$ , then $\alpha_2= \frac{1}{\alpha_1}$ or $\alpha_1=\alpha_2$. Let $A$ denote the affine conjugacy such that $A \circ T_{\alpha_1} \circ A^{-1}=T_{\alpha_2}$ and set $A(z):=az+b$, $a \in \C^*$, $b\in\C$. Then $A$ must map $S(T_{\alpha_1})=\{ 1, \frac{1}{\alpha_1}\}$ to $S(T_{\alpha_2})=\{1, \frac{1}{\alpha_2}  \}$.

Suppose first that $A(0)=0$ and hence $b=0$. If $A(1)=1$, then $A$ fixes two points and hence it is the identity. If $A(1)=1/\alpha_2$, then $a=1/\alpha_2$. But then, $A(1/\alpha_1)=1$ and hence $a=\alpha_1$. It follows then that $\alpha_1=1/\alpha_2$. 

Nevertheless, it could also happen that $A(0)=z_2\neq 0$, where $z_2$ is a fixed point of $T_{\alpha_2}$ with multiplier $1/8$. (This situation could only happen for a countable number of parameters, for which the Fatou set consists of two completely invariant basins which are Jordan domains, see Remark \ref{rem:switch}.) More precisely, for $\alpha\in\C\setminus\{0,1\}$, one can check that the fixed points of $T_\alpha$ must satisfy
\begin{equation}\label{fp}
e^\frac{(\alpha-1)z}{8}=\frac{z-1}{\alpha z -1},
\end{equation}
while points of derivative $1/8$ satisfy
\begin{equation}\label{der18}
e^\frac{(\alpha-1)z}{8}\in \{1, 1/\alpha^2\}. 
\end{equation}
From the two equations together, it follows that there is only one possible nonzero fixed point with multiplier $1/8$, and that is  $z_\alpha=1+\frac{1}{\alpha}$. Thus, an affine conjugacy switching the basin of $0$ and the basin of $z_{\alpha_2}$ must satisfy $A(0)=1+\frac{1}{\alpha_2}$ and $A(1+\frac{1}{\alpha_1})=0$ and hence has the form 
\[
A(z)=-\frac{\alpha_1(1+\alpha_2)}{\alpha_2(1+\alpha_1)}z + 1+\frac{1}{\alpha_2}.
\]
On the other hand, $A$ must map $S(T_{\alpha_1})=\{ 1, \frac{1}{\alpha_1}\}$ to $S(T_{\alpha_2})=\{1, \frac{1}{\alpha_2}  \}$. From the expression of $A$, one can check that if $A(1)=1$ then $\alpha_1=\alpha_2$ while, if $A(1)=1/\alpha_2$, then $\alpha_2=1/\alpha_1$, as we wanted to show.
\end{proof}

\begin{rem}\label{rem:switch} (Maps with a symmetry switching the basins) If $\alpha$ is a parameter satisfying 
\beq \label{eq:symmetries}
e^{\frac{\alpha^2-1}{8\alpha}}=1/\alpha^2,
\eeq
then $z_\alpha=1+1/\alpha$ is a fixed point of multiplier $1/8$. The affine conjugacy $A(z)= -z + 1 + 1/\alpha$ conjugates $T_\alpha$ to itself switching the two attracting basins. Composing with $z\mapsto\alpha z$, we obtain an affine conjugacy to $T_{1/\alpha}$ switching the two basins and mapping $1$ to itself. Examples of filled Julia sets for these parameter values can be seen in Figure \ref{switch}.
\end{rem}
\begin{figure}[hbt!]
\centering
 \includegraphics[width=0.35\textwidth]{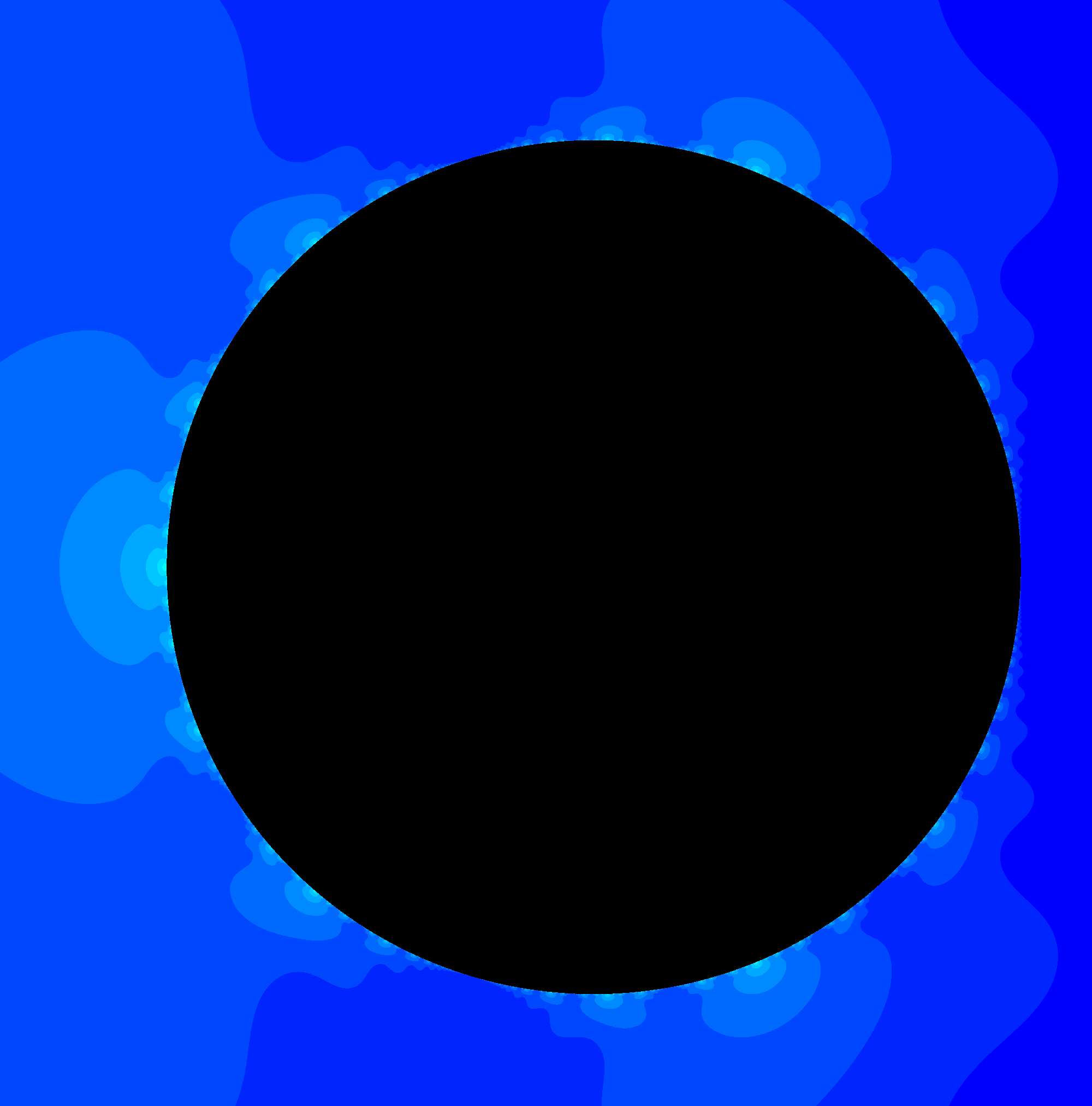}\hspace*{0.1\textwidth}
  \includegraphics[width=0.35\textwidth]{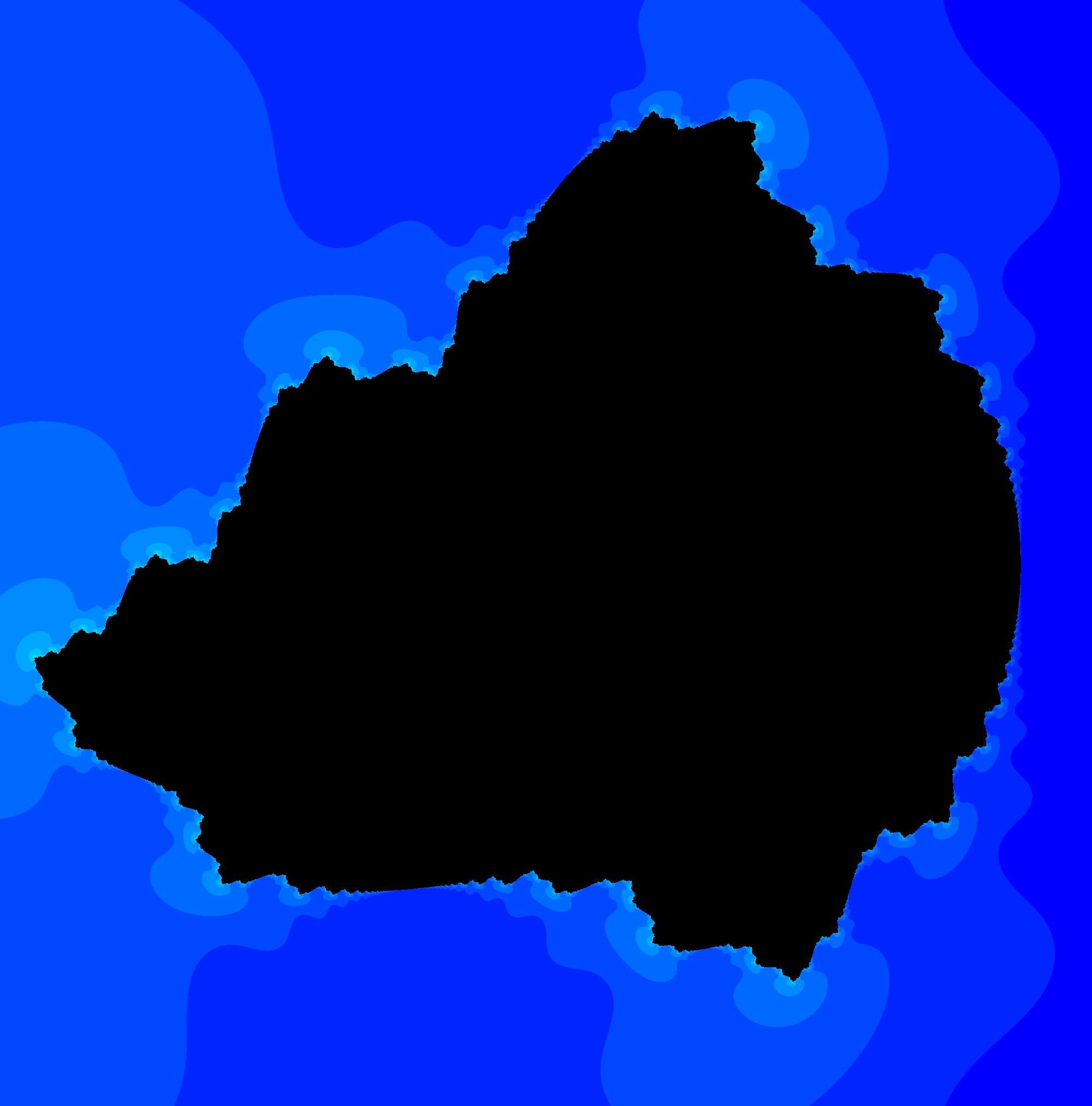}
  \setlength{\unitlength}{0.8\textwidth}
  	\put(-0.602,0.2){\color{white}{\circle*{0.01}}}
 	\put(-0.64,0.195){\color{white}{\small$\infty$}}
	\put(-0.03,0.2){\color{white}{\circle*{0.01}}}
	\put(-0.07,0.195){\color{white}{\scriptsize$\infty$}}	
 \caption{\small \label{switch} Filled Julia sets of $T_\alpha$ for $\alpha_{\sharp}=-0.01484108...$ (left,  see Remark \ref{rem:specialmap}) and $\alpha=-0.00801734...+0.00675639... i$ (right).}
 \end{figure}

\begin{rem}\label{rem:specialmap} (A special map)
The example shown in Figure \ref{switch} (left) is actually quite special. Indeed, consider  $g(z):=C \tan(\pi z)$, for the unique  $C>0$ such that $g$ has a fixed point $p\in i\R$ with multiplier $1/8$ (see Section \ref{sec:confconj} for more details). By symmetry, $\overline{p}$ is also fixed under $g$ with the same multiplier,  $J(g)=\R$ and both half planes are completely invariant. The two asymptotic values are $\pm C i$. Now, by Lemma  \ref{lem:rigidity}, $g$ is affine conjugate to a unique map $T_{\alpha_{\sharp}}$, with $\alpha_{\sharp}\in\D$, by a conjugacy $\omega_{\alpha_\sharp}$ sending $p$ and $C i$ to $0$ and $1$ respectively. Consequently, $\omega(-Ci)=1/\alpha$ and an easy calculation shows that $\alpha\in\R^-$. We denote this special value of the parameter by $\alpha_{\sharp}$, which corresponds to the only real solution of (\ref{eq:symmetries}). Observe that  $J(T_{\alpha_{\sharp}})$ is the vertical  straight line at equal distance from $0$ and $z_{\alpha_{\sharp}}=1+\frac{1}{\alpha_{\sharp}}$, i.e. $J(T_{\alpha_{\sharp}})=\{\re(z)=\frac{\alpha_{\sharp}-1}{2\alpha_{\sharp}}\}$.  
\end{rem}

Next we show that for $\alpha \in \D$, $T_\alpha$ is a tangent-like map when seen from the complement of the basin of $0$ (see Figure \ref{fig:modelTL}).

\begin{lem}[The maps $T_\alpha$ are tangent-like] \label{lem:tlmodel}
	For all $\frac{1}{2}\leq r<1$, for all $\alpha \in \D(0,r)$, the map 
	$T_\alpha : \rs \setminus T_\alpha^{-1}(\D(0,\frac{1}{r})) \to \rs \setminus \D(0,\frac{1}{r})$ is TL, with $u=\infty$ and $v=\frac{1}{\alpha}$. 
\end{lem}

\begin{proof}

	Let  $\frac{1}{2}\leq r<1$ and $\alpha \in \D(0,r) \subset \D$.
	Let  $V:=  \rs \setminus \overline{\D}(0,\frac{1}{r})$ and $U:=T_\alpha^{-1}(V)$, and let $u:=\infty$ and $v:=\frac{1}{\alpha}$. 
	Since $|\alpha|<r<1$, $V$ contains $v=\frac{1}{\alpha}$ but not $1$. Moreover, $V$ is a Jordan domain, whose boundary is real-analytic (in particular, $\partial V$ is a quasicircle). The map $T_\alpha: U \to V \setminus \{v\}$ is a universal cover, hence $U$ is simply connected. From the expression of $T_\alpha$, it is also clear that $\partial U$ is a real-analytic Jordan curve, in particular a quasicircle.
	Finally, by Lemma \ref{lem:disk mapping inside} and the fact that $r>\frac{1}{2}$, we have $T_\alpha(\D(0,\frac{1}{r}) )\subset \D  \Subset \D(0,\frac{1}{r})$, from which it follows that   $U \Subset V$.%
\begin{figure}[hbt!]
	\centering
\includegraphics[width=0.4\textwidth]{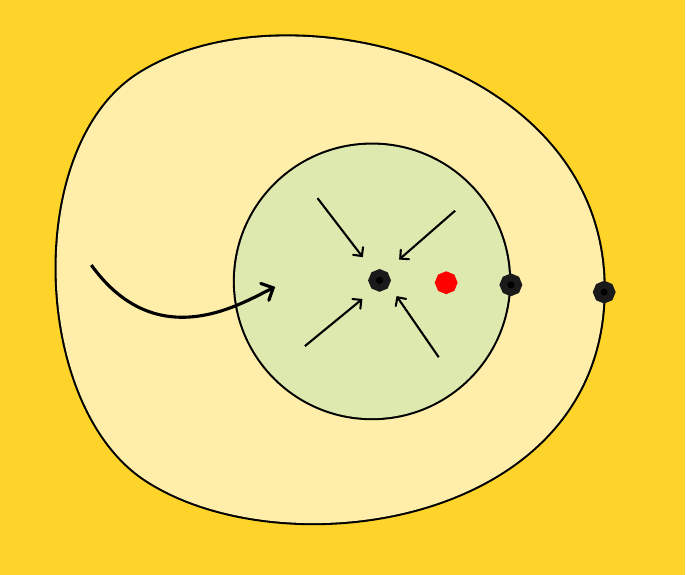}
\hspace*{0.05\textwidth}
\includegraphics[width=0.4\textwidth]{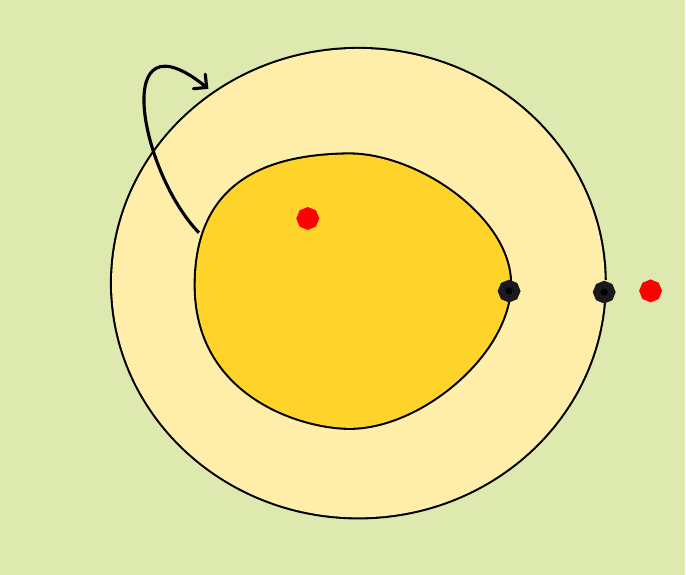}
\setlength{\unitlength}{0.85\textwidth}
\put(-0.028,0.165){\scriptsize$1$}
\put(-0.08,0.185){\scriptsize$2$}
\put(-0.155,0.187){\scriptsize$\infty$}
\put(-0.25,0.245){\scriptsize$1/\alpha$}
\put(-0.41,0.34){$T_\alpha$}
\put(-0.59,0.187){\scriptsize$\infty$}
\put(-0.655,0.19){\scriptsize$2$}
\put(-0.7,0.19){\scriptsize$1$}
\put(-0.785,0.195){\scriptsize$0$}
\put(-0.92,0.15){$T_\alpha$}
\caption{\label{fig:modelTL} The Tangent-like restriction of the model family, seen from the complement of the basin of attraction of 0. }
\end{figure}
\end{proof}

As it is the case for polynomials and polynomial-like maps, the choice of $U$ and $V$ is not really relevant and produce the same filled Julia set, which in fact is the whole Julia set of the model map, in this case $T_\alpha$. 

\begin{prop}[Choices of $U$ and $V$] \label{prop:choicesUV}
	For any $\alpha \neq 1$,  for any TL restriction  $T_\alpha: U \to V$ with $u=\infty$ and $v=\frac{1}{\alpha}$:
	\begin{enumerate}
		\item the filled Julia set $K(T_\alpha)$ is $\rs \setminus \Omega_\alpha$, where $\Omega_\alpha$ is the attracting basin of $0$.
		\item the Julia set of $T_\alpha$ is also the Julia set of its TL restriction.
	\end{enumerate}
	In particular, Julia sets and filled Julia sets of the TL restriction do not depend on the choice of $U$ and $V$.
\end{prop}

\begin{proof}
	Let us first prove that  $K(T_\alpha)= \rs \setminus \Omega_\alpha$. Indeed, observe 
	that since $U \Subset V$, we have $\rs \setminus \overline{V} \Subset \rs \setminus \overline{U}$; and 
	$T_\alpha$ maps $\rs \setminus \overline{U}$ into $\rs \setminus \overline{V}$. Therefore, the map $f$ has a unique attracting fixed point $\rs \setminus \overline{V}$, which must be the attracting fixed point $0$; and $\rs \setminus \overline{U} \subset \Omega_\alpha$. But by definition, if $z \notin K(T_\alpha)$, then there exists $n \in \N$ such that $T_\alpha^n(z) \in \rs \setminus \overline{U}$, so $z \in \Omega_\alpha$. This proves $\rs \setminus K(T_\alpha) \subset \Omega_\alpha$. Moreover, $\rs \setminus K(T_\alpha)$ is backward invariant, therefore, 
	 $\rs \setminus K(T_\alpha) = \Omega_\alpha$.

	 The second claim follows from the fact that the Julia set of $T_\alpha$ is $\partial \Omega_\alpha$.
\end{proof}

In particular, the general properties in Section \ref{sec:TLmaps} apply to maps $T_\alpha$ in the model family, for $\alpha\in \D$ and, by switching $\alpha$ and $1/\alpha$, for all $\alpha\in \C\setminus\overline{\D}$.

\subsection{Conformal conjugacy on the attracting basin of $0$} \label{sec:confconj}

In this section we show that every  map $T_\alpha$ with connected filled Julia set  is conformally conjugate on its  attracting basin $\Omega_\alpha$ to the model map $g(z)=C\tan(\pi z):\H\ra\H\setminus\{C i\}$, where $C$ is a constant independent of $\alpha$. In particular, all such tangent maps $T_\alpha$ are conformally conjugate to each other on the attracting basins of zero. This  is an analogue of B\"ottcher's Theorem, with the model map $z^d$ replaced by the map $g$. This  conjugacy to $g$ allows us to  introduce a canonical  labeling of the  poles for the maps $T_\alpha$ and to identify a distinguished repelling fixed point.  In Section~\ref{sec:qc and hybrid classes} we will introduce another  model map  with a different normalization  which will be used to study conjugacy classes.

\begin{lem}[Conjugacy on the basin of $0$]\label{lem:confconjH}
Let $\alpha\in \D(0,1/2)$ such that   $K(T_\alpha)$ is connected, and let  $\Omega_\alpha$ be the  attracting basin of $0$. Then $T_\alpha$  is  conformally conjugate on  $\Omega_\alpha$ to the map  $g:\H\ra\H\setminus \{Ci\}$ defined as 
$$
g(z)=C\tan(\pi z), 
$$
 where $C>0$ is a  constant independent of $\alpha$. The conformal conjugacy is given by the Riemann map $\riem_\alpha:\H\ra\Omega_\alpha$,  normalized so that  $\riem_\alpha^{-1}(1)=Ci$ and $q:=\riem^{-1}_\alpha(0) \in (0,Ci)$. 
 Consequently, any two tangent-like maps satisfying the hypothesis of this proposition are conformally conjugate on their basins of attraction of $0$. 
\end{lem}

The uniqueness was essentially proved  in \cite[Lemma 5.2]{ERS}, which is equivalent to the following (compare with  \cite{DG1987}).

\begin{lem}[Uniqueness of hyperbolic inner functions] \label{lem:innerunique}
Let $\rho\in\D\setminus \{0\}$. Then, up to conjugacy by a Möbius
automorphism of $\H$, there exists a unique inner function of infinite degree from $\H$ to itself, having  an attracting fixed point of multiplier $\rho$,  and only one singular value in $\H$.
\end{lem}

\begin{proof}[Proof of Lemma \ref{lem:confconjH}]
Imposing that $g$ has a fixed point $i y_0\in i \R$ of multiplier $1/8$ we obtain the equations
\[
C \tanh(y_0\pi)=y_0; \quad C\pi /\cosh^2(\pi y_0)=1/8.
\]
From the second equation, $y_0\simeq 0.670099...$ is the unique solution of $2 y_0 \pi/\sinh(2\pi y_0)=1/8$, and from the first one, $C=\frac{1}{8\pi}\cosh^2(\pi y_0) \simeq 0.690289...$, hence both are uniquely determined. Notice that $g$ has a unique essential singularity at $\infty$. 

We now show that $\riem_\alpha: \H\to \Omega_\alpha $ conjugates $T_\alpha$ to $g$. Recall that   $\riem_\alpha(C i)=1$ and $q:=\riem^{-1}_\alpha(0) \in (0,C i)$.  Then $\tilde g:= \riem_\alpha^{-1} \circ T_\alpha\circ \riem_\alpha$ is a holomorphic selfmap of $\H$ with a fixed point of multiplier $1/8$ and only one singular value  in $\H$.

By Lemma \ref{lem:innerunique}, $g=M^{-1}\circ \tilde g \circ M$, where $M$ is an automorphism of $\H$ which fixes $C i$ and maps $i y_0$ to $q$. Since $iy_0$ and $q$ are both in the segment $(0,Ci)$, it follows that $M$ is the identity (by uniqueness of the Riemann map), and the Lemma is proved . 
\end{proof}

The conjugacy $\omega_\alpha$ normalized as above allows us to mark some special points in $J(T_\alpha)$. In particular, there is a unique labeling for the poles independent of $\alpha$, as long as $\alpha\neq 0$. See Figure \ref{fig:Riemann}.

\begin{lem}[Labelling of poles and fixed point] \label{lem:labelling}
Let $\alpha \in\D(0,1/2)\setminus\{ 0\}$ and $\riem_\alpha, g$ be defined as in Lemma \ref{lem:confconjH}. Then 
\begin{equation}\label{eq:radial}
\lim_{t\to\infty} \riem_\alpha(it) = \infty.
\end{equation}
Consequently,  infinity and all poles of $T_\alpha$ are accessible from $\Omega_\alpha$ and can be consistently labelled as 
\begin{equation} \label{eq:labelling}
p_k:=p_k(\alpha):= \lim_{t\to 0} \riem_\alpha \left(\frac12+k + i \,t \right), k\in\Z,
\end{equation}
 which defines a \emph{marked pole} $p:=p_0=p_0(\alpha)$. Likewise,  we can consistently define a special fixed point, which we will call the {\em $\beta_0$-fixed  point}, as
\[
\beta_0:=\beta_0(\alpha):= \lim_{t\to 0} \riem_\alpha(i \, t).
\] 
\end{lem}

\begin{figure}[hbt!]
\centering
\includegraphics[width=0.9\textwidth]{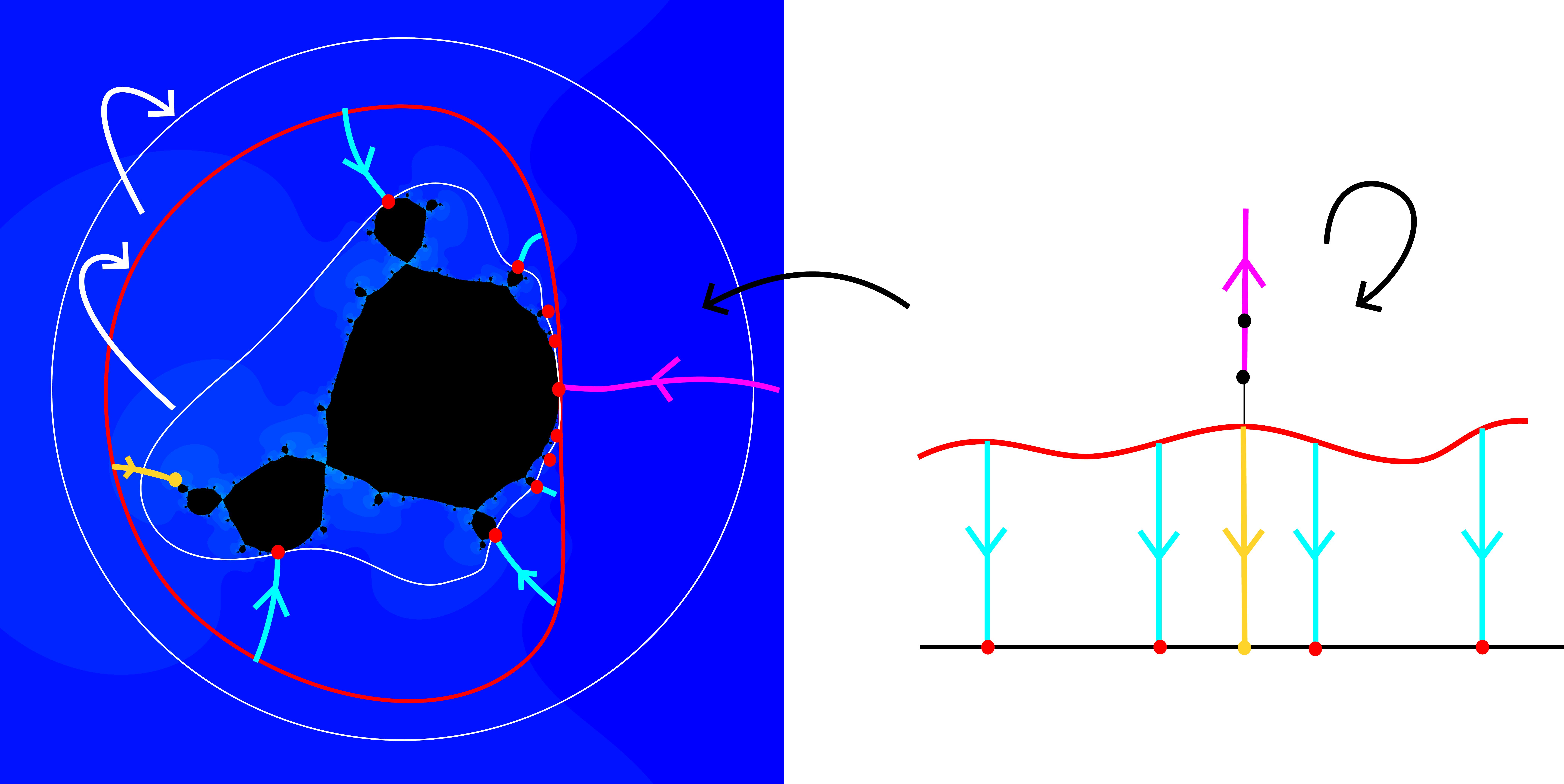}
\setlength{\unitlength}{0.9\textwidth}
\put(-0.3,0.35){\large$\H$}
\put(-0.1,0.39){$g$}
\put(-0.195,0.26){\scriptsize $p$}
\put(-0.195,0.29){\scriptsize $C i$}
\put(-0.21,0.06){\scriptsize $0$}
\put(-0.165,0.06){\scriptsize $\frac12$}
\put(-0.06,0.06){\scriptsize $\frac32$}
\put(-0.28,0.06){\scriptsize $-\frac12$}
\put(-0.39,0.06){\scriptsize $-\frac32$}
\put(-0.1,0.25){\small $\riem_\alpha^{-1}(\partial U_\alpha)$}
\put(-0.48,0.35){$\riem_\alpha$}
\put(-0.64,0.26){\color{white} \scriptsize $\infty$}
\put(-0.82,0.13){\color{white} \scriptsize $p_0$}
\put(-0.89,0.21){\color{white} \scriptsize $\beta_0$}
\put(-0.67,0.15){\color{white} \scriptsize $p_1$}
\put(-0.75,0.395){\color{white} \scriptsize $p_{-1}$}
\put(-0.97,0.43){\color{white}  $T_\alpha$}
\put(-0.66,0.4){\color{white}  $U_\alpha$}
\put(-0.6,0.43){\color{white}  $V$}

\caption{\small The conformal conjugacy on the basin of 0 (blue) and the labelling $p_k$ of the poles (red dots)  ($1/2$ corresponds to $k=0$). The origin corresponds to the fixed point $\beta_0$.}
\label{fig:Riemann}
\end{figure}

 \begin{proof}
 Since $\partial U_\alpha$ is a Jordan curve and $V\setminus \overline{U_\alpha} \subset \Omega_\alpha$, it follows that $\infty$ is accessible from $\Omega_\alpha$ for all $\alpha$. Let $\gamma\in\Omega_\alpha$ be an arc landing at $\infty$.  Then, by Lindelöf's Theorem (see e.g. \cite{pom}) $\riem_\alpha^{-1}(\gamma)$ is a curve in $\H$ landing at some point $w\in\partial \H$. Since $\infty$ is an essential singularity of $T_\alpha$ (and for example, all preimages of a given point accumulate at $\infty$), it follows that $w$ is the only singularity of $g$ in $\partial\H$, which is $\infty$. Again by Lindelöf's Theorem, the radial limit of $\riem_\alpha$ at $\infty$ exists and equals $\infty$, hence (\ref{eq:radial}) holds.
 
 Let us now consider $T_\alpha^{-1}(\partial U_\alpha)$, which is a Jordan curve going through $\infty$ and containing all the poles of $T_\alpha$. Since $T_\alpha$ is a covering from the pinched annulus $U_\alpha\setminus \overline{T_\alpha^{-1}(U_\alpha)}$ to $V \setminus \overline{U_\alpha}$, the curve $\gamma$ has infinitely many disjoint preimages which are Jordan arcs joining $\partial U_\alpha$ and $\partial T_\alpha^{-1}( U_\alpha)$, ending at each pole of $T_\alpha$. Again by Lindelöf's Theorem, the image of these segments under $\riem_\alpha^{-1}$ must land at the distinct poles of $g$ (since $\riem_\alpha$ is a conjugacy), at which the radial limits exist and hence (\ref{eq:labelling}) holds.
 
To define $\beta_0$ observe that $w=0$ is a repelling fixed point of $g$. By \cite[Thm. A]{bfjk_accesses}, the radial limit of $\riem_\alpha$ at $w=0$ exists and equals either $\infty$ or a weakly repelling fixed point of $T_\alpha$. To rule out infinity, suppose that $\lim_{t\to 0} \riem_\alpha(i t)=\infty$. Then, since $g(it)=C i \tanh(\pi \, t) \to 0$ as $t\to 0$, and $\riem_\alpha$ is a conjugacy, it follows that $T_\alpha(\riem_\alpha(it))\to \infty$ as $t\to 0$, which implies that $\infty$ is an asymptotic value for $T_\alpha$, a contradiction, since that occurs only for $\alpha=0$. Hence $\beta_0$ is well defined as the landing point of the invariant arc  $\riem_\alpha(it)$ as $t\to 0$.

  \end{proof}

\subsection{Rigidity of conjugacies near the Julia set}\ 

The main result in this Section is  that every quasiconformal conjugacy between $T_\alpha$ and itself defined near the filled Julia set, which is close to the identity near  the marked pole, extends to $K(f)$ as the identity.  This  result is crucial and will  be used in Section \ref{sec:hybrid classes} and in Section \ref{sec:ff}.

\begin{prop}[Extension of conjugacies to $K(T_\alpha)$] \label{prop:extend}
	Let $\alpha \in \D(0,1/2)$ such that $K(T_\alpha)$ is connected. Let $U,V$ be such that $T_\alpha: U \to V$ is TL, and let $p_0=p_0(\alpha)$ be  the  marked pole as defined in (\ref{eq:labelling}) (Lemma \ref{lem:labelling}). Let $h: V \setminus K(T_\alpha) \to V \setminus K(T_\alpha)$ 
	be a quasiconformal homeomorphism such that
	\begin{equation} \label{eq:conmute}
	h \circ T_\alpha = T_\alpha \circ h \text{\ \ \ \ \  on $U \setminus K(T_\alpha)$}.
	\end{equation}
	Then, there exists $\epsilon>0$ such that:  if there exists a domain $W' \subset U$ with $p_0 \in \partial W'$ and 
	$$\sup_{z \in W'} |h(z)-z|<\epsilon,$$ then
	$h$ extends to a continuous homeomorphism $h: V \to V$ such that $h=\id$ on $K(T_\alpha)$.
\end{prop}

We shall use notation and definitions from Section \ref{sec:confconj}.  We start by proving that topological conjugacies of $g$ to itself near the real line extend to the real line as the identity, as long as they are close to the identity near one pole.

\begin{lem}[Real maps commuting with the tangent] \label{lem:noaut}
	 Suppose that $h: \R\cup \{\infty\}\to \R\cup \{\infty\}$ is an increasing  homeomorphism such that $h \circ g= g \circ h$ on $\R$ and  $|h(\frac12) -\frac12 |<\frac{1}{2}$.  Then $h= \id$.
\end{lem}

\begin{proof}	
	First, observe that since $g(x+1)=g(x)$ for all $x \in \R$, we have $g \circ h(x+1) = g \circ h(x)$ for all $x \in \R$. But we also have that $g(x_1) = g(x_2)$ if and only if $x_1-x_2 \in \Z$. Therefore, for every $x \in \R$, there exists $k_x \in \Z$ such that $h(x+1) - h(x) = k_x$.  Since $k_x$ is continuous and takes values in  a discrete set, it is constant, hence  there exists $k \in \Z$ such that for all $x \in \R$, 
	\begin{equation}\label{eq:periodicity}
	h(x+1)=h(x) + k.
	\end{equation}

	On the other hand, $h(\infty)=\infty$, and $h$ must map poles of $g$ to other poles; so $h : \frac{1}{2}+\Z \to \frac{1}{2} + \Z$ is a monotone bijection. This implies that $|k|=1$,  and since $h$ is increasing, $k=1$. In particular, there exists $n_0 \in \Z$ such that for all $n \in \Z$, we have 
	$h(\frac{1}{2}+n) = \frac{1}{2}+n_0 + n$. But $|h(\frac12) -\frac12 |<\frac{1}{2}$ by hypothesis and hence $n_0=0$, i.e., all poles are fixed by $h$, and every interval $I_n:=[-1/2+n,1/2+n]$ is mapped bijectively to itself. 
		
	Now for any $j \in \Z$, let  $k_j: \R\to I_j$ be the corresponding inverse branch of $g^{-1}$ (explicitly, $k_j(z)=\frac{1}{\pi} \arctan(\frac{y}{C}) +j$).
	By the above, $h(I_j)=I_j$.
	Therefore, for any $j \in \Z$ and for any $x \in I_j$, 
	$$
	h(x) = k_j \circ h \circ g(x).
	$$
	Let $x_{j}:=\frac{1}{2}+j$ and $x_{j_2,j_1} := k_{j_2}(x_{j_1}) \in I_{j_2}$.
	Then 
	$$h(x_{j_2,j_1}) = k_{j_2} \circ h(x_{j_1}) = k_{j_2}(x_{j_1}) =  x_{j_2,j_1}.$$
	Proceeding inductively, it follows that if we let $x_{j_n, \ldots, j_1}:=k_{j_n} \circ \ldots \circ k_{j_2}(x_{j_1})$, then $h( x_{j_n, \ldots, j_1}  ) = x_{j_n, \ldots, j_1}$ for all $n \in \N^*$ and all 
	$(j_1, \ldots, j_n) \in \Z^n$. 
	
	Since these points are dense in $\R$, we deduce $h=\id$.
\end{proof}

In particular, we deduce the following.

\begin{coro}\label{coro:qc(g)}
	Assume that there exists a homeomorphism $H: \R \to \R$ commuting with $g$ such that $H(\frac{1}{2})=-\frac{1}{2}$. 
	Then $H$ generates the group of homeomorphisms of $\mathbb R$ which commute with $g$.
\end{coro}

\begin{proof}
	Let $G$ be a homeomorphism of $\mathbb R$ commuting with $g$. Since $G$ must map poles to poles, there exists $k \in \Z$ such that $G(\frac{1}{2})= \frac{1}{2}+k$. Then $H^{k} \circ G$ is a homeomorphism commuting with $g$, and which fixes $\frac{1}{2}$. Therefore, by Lemma~\ref{lem:noaut}, it is the identity.
\end{proof}

We will see later that, indeed,  such a homeomorphism exists.

\begin{lem}[Maps commuting in a neighborhood of $\R$]\label{lem:strip}
	Let $W$ be some  strip in $\H$ with one boundary component being $\R$,
	 such that $g^{-1}(W) \subset W$.
	Let $h: W \to W$ be a quasiconformal homeomorphism such that 
	$g \circ h= h \circ g$ on $g^{-1}(W)$. Let $U$ be an open set in the upper half-plane $\mathbb H$ such that 
	$\frac12 \in \partial U$, and
	suppose that $ |h(z)- z| < 1/4$ for all $z \in U$.
	Then $h$ extends continuously to $\R$ as the identity.
\end{lem}
Notice that the condition of $g^{-1}(W) \subset W$ ensures that the conjugacy makes sense. 
\begin{proof}
Since $h$ is a  quasiconformal homeomorphism, it  extends as a homeomorphism to the boundary of $W$  (see e.g.  Proposition 2.17 in \cite{bf}). We still denote by $h$ the map $h:\overline{W}\ra\ov{W}$ obtained via this extension. 
By continuity, we have  that $h \circ g = g \circ h$ on $\R$ and, by hypothesis,   $|h(z)- z| <1/4$ for all $z \in U$. So $h$ restricted to $\R$ satisfies the hypothesis of		Lemma \ref{lem:noaut} and hence   $h=\id$ on $\R$.
\end{proof}

We are now ready to transport these theorems to the dynamical plane of $T_\alpha$ via the conformal conjugacies $\riem_\alpha$. We will make use of the following theorem by Krzyz \cite{Krzyz} (see \cite[Theorem 2.19]{Solynin}, for the special case of the unit disk).

\begin{prop}\label{prop:krzyz}
Let $h:\D\ra\D$ be a $K$-quasiconformal map which extends as the identity to $\partial \D$. Then for any $z\in\D$,
\begin{equation}
\dist_{\D}(h (z),z)\leq C,
\end{equation} 
Where $C$ is a constant depending only on $K$.
\end{prop}

We are now ready to prove the main result in this section. 

\begin{proof}[Proof of Proposition~\ref{prop:extend}]
The homeomorphism $h$ extends continuously to $\partial V$ and hence it commutes with $T_\alpha$ also on $\partial U_\alpha \setminus \{\infty\}$, where it is well defined. Hence, by setting $h(\infty)=\infty$,  $h$ is a homeomorphism of the closed annulus $A:=\overline{V}\setminus U$, to itself, which commutes with $T_\alpha$. Let $\gamma(t)=\riem_\alpha(i \, t)$ for $t>t_0$ large, where $\riem_\alpha$ is the conformal conjugacy from Lemma \ref{lem:confconjH},  conjugating $T_\alpha$ to $g(z)=C \tan(\pi z)$ on the basin of $0$. Then, by Lemma \ref{lem:labelling}, $\gamma(t)\to \infty$ as $t\to \infty$. Since $h$ is continuous on the closed annulus, $h(\gamma(t))\to \infty$ as $t\to +\infty$. Let $\sigma_0(t)$ be the component of $T_\alpha^{-1}(\gamma(t))$ landing at $p_0$, which equals $\riem_\alpha(\frac12 + it)$ for $t>0$ small (see Lemma \ref{lem:labelling}).  Since $h$ commutes with $T_\alpha$, we have that $h(\sigma_0(t))$ must be a curve  whose image under $T_\alpha$ is $h(\gamma(t))$, hence it is landing at a pole.

Let  $\epsilon>0$ be such that the disc $D:=\D(p_0,2\epsilon)$ contains no other pole of  $T_\alpha$. Then $h(\D(p_0,\epsilon) \cap W')\subset D$ and hence $h(\sigma_0(t))$ lands at $p_0$.
  
Let  $W:=\riem_\alpha^{-1}(V\setminus K(T_\alpha))$ which is a horizontal strip in $\H$ having $\R$ as its lower boundary component. Since $\riem_\alpha$ conjugates $g$ to $T_\alpha$, it follows that $g^{-1}(W)=\riem_\alpha^{-1}(U\setminus K_\alpha) \subset W$ and that  the quasiconformal homeomorphism $\tilde{h}=\riem_\alpha^{-1}\circ h \circ \riem_\alpha: W\to W$  commutes with $g$ on $g^{-1}(W)$. 

Since $\tilde{h}$ is quasiconformal, it extends continuously to $\R$, and since it is close to the identity, the extension is increasing. We claim that $\tilde{h}(\frac12)=\frac12$ and hence, by Lemma  \ref{lem:strip} $\tilde h$ extends continuously to $\R \cup \{\infty\}$ by the identity.
Indeed, both curves $\sigma_0(t)$ and $h(\sigma_0(t))$ are homotopic in $\Omega_\alpha$ (because they belong to the pinched annulus $U\setminus \overline{T_\alpha^{-1}(U)}$ which is disjoint from the Julia set) and hence belong to the same access to $p_0$. Hence $\riem_\alpha^{-1}$ sends them both to curves landing at the same point of $\R\cup \{\infty\}$ which needs to be $1/2$ by definition of $\sigma_0$. Since $\tilde{h}$ commutes with $g$, the claim follows.

\medskip

We now need to show that $h$ extends to $J(T_\alpha)$ as the identity. 
Let us  first show that there exists a constant $L>0$ such that 
\begin{equation}\label{eq: h hyperbolically small}
\dist_{W}(\tilde h (z),z)\leq L \text{\ \ \ for all $z\in W$ with $\Im z$ sufficiently small}.
\end{equation} 

Since $\tilde h$ is the identity on $\R\cup \{\infty\}$,  by mapping $\H$ conformally to $\D$ we obtain a map $h_\D$ which is defined on an annulus whose outer boundary is $\partial \D$, which is quasiconformal and which  extends as  the identity   to $\partial \D$. Up to restricting $V$ if necessary we can assume that $\partial V$ is real analytic and hence $\partial W$ is a quasicircle. Hence  we can extend $h_\D$ as a quasiconformal map to all of $\D$, to which Proposition~\ref{prop:krzyz} applies.  This gives that  $\dist_{\H}(\tilde h (z),z)\leq C $ for $z\in W$. Since the hyperbolic density in $\H$ and in $W$ are comparable for $z$ with small imaginary part, (\ref{eq: h hyperbolically small}) follows. 

Now, since $\riem_\alpha: W\ra V\setminus K(T_\alpha)$ is an isometry for the hyperbolic metric, \eqref{eq: h hyperbolically small} implies that  
\begin{equation}
\dist_{V\setminus K(T_\alpha)}(h (z),z)\leq L \text{\ \ \ for all $z\in V\setminus K(T_\alpha)$ close to $K(T_\alpha)$}
\end{equation} 
which implies that $|h(z)-z|\ra0$ as $z\ra J(T_\alpha)$. Thus $h\ra\Id$ as $z\ra J(T_\alpha)$.
\end{proof}

\subsection{Parameter plane: The Tandelbrot set. Proof of Theorem B}\

As we saw in Lemma~\ref{lem:tlmodel}, for $\alpha \in\D(0,\frac{1}{2})$ the asymptotic value $1$ belongs $\Omega_\alpha$,  the basin of attraction of $z=0$, and hence $1/\alpha$ acts as a free asymptotic value, whose orbit may or may not be captured by the basin of 0. In analogy to the Mandelbrot set, we define the following subset of the $\alpha-$parameter plane (see Figure \ref{fig:Tandelbrot}). 

\begin{defi}[The Tandelbrot set]
	The Tandelbrot set  is 
\[
	\tcal:=\{\alpha \in \C \setminus \{1\}: \text{ $T_\alpha^n \left(\frac{1}{\alpha}\right)\not\rightarrow 0$ as $n\to\infty$.}\}.
\]
\end{defi}

As expected, the Tandelbrot set lies in a compact subset of the plane.

\begin{lem}[$\tcal$ is bounded]\label{lem:tbdd}
	$\tcal \subset \D(0,\frac{1}{2})$.
\end{lem}
\begin{proof}
 We need to show that if $|\alpha|\geq 1/2$ then $1/\alpha\in\Omega_\alpha$. 
On the one hand, if $1/2\leq |\alpha|\leq 1$, $\alpha\neq 1$, then $1/\alpha \in \overline{\D}(0,2)$ which, by Lemma \ref{lem:disk mapping inside}, belongs to $\Omega_\alpha$. 
On the other hand, if  $\alpha>1$, consider the disk $D:=\overline{\D}(0,|1/\alpha|)$ and argue as in the proof of Lemma \ref{lem:disk mapping inside} to show that $T_\alpha(D)\Subset D$, and hence $D\subset \Omega_\alpha$. 
\end{proof}

We shall see next that the boundary of $\tcal$ coincides with the activity locus of the asymptotic value $1/\alpha$. Note first that  $\{T_\alpha\}_{\alpha\in \C \setminus \{1\}} $ is a natural family of meromorphic maps, with base map e.g.  $T_0(z)=1-e^{-z/8}$, and holomorphic motions $\psi_\alpha(z)=z/(1-\alpha)$ and $\phi_\alpha(z)=\frac{z}{\alpha z+1-\alpha}$. We recall the definitions of activity and passivity of singular values in Section \ref{sec:prelim} and define the {\em activity locus} of $1/\alpha$ as the set
\[
\acal=\{\alpha\in\C\setminus\{1\}\mid \text{$1/\alpha$ is active at the parameter $\alpha$}\}.
\]

\begin{prop}[Activity locus of $1/\alpha$]\label{prop:activity locus}
The boundary of $\tcal$ is the set of parameter values for which $1/\alpha$ is active. In other words
\[
 \partial \tcal = \acal.
 \]
\end{prop}
\begin{proof}
Clearly $\partial \tcal \subset \acal$. To see the reverse inclusion suppose  that  $\alpha_0\in\acal$ 
and let us show that there are parameters $\alpha$ arbitrarily close to $\alpha_0$ such that $T^n_\alpha(1/\alpha) \to 0$ as $n\to\infty$. Since parameters in the activity locus cannot have a neighborhood on which  $T^n_\alpha(1/\alpha) \to 0$,  it will follow that $\alpha_0$ is indeed in the boundary of $\tcal$. 

Since $\{T_\alpha\}_{\alpha\in \C \setminus \{1\}}$ is a natural family of meromorphic maps of finite type, by Proposition~\ref{prop:density_of_virtual cycles}
 every active parameter can be approximated by 
  parameters $\alpha$ such that $T_\alpha^n(1/\alpha) =\infty$, for some $n\geq 1$.

At the same time,  by the shooting lemma \ref{shooting}, 
 every such virtual cycle parameter of order $n$ can be approximated by parameters $\alpha$ such that $T_\alpha^{n+1}(1/\alpha)=1/2$ (or any other constant function different from 1). 
Since $1/2$ is always in the basin of $0$ for $\alpha$ in a neighborhood of $\tcal$, 
we get the result.
\end{proof}

\begin{rem}
By Lemma \ref{lem:rigidity}, the parameter space of the family $\{T_{\alpha}\}_{\alpha \in \C \setminus \{1\}}$ is invariant under the involution $\alpha \mapsto \frac{1}{\alpha}$ (except for $\alpha=0$). 
Thus, the set $1/{\mathcal{T}}$ coincides with the set of $\alpha \in \C \setminus \{1\}$ such that 
the asymptotic value $1$ is not captured by $0$ for $T_\alpha$. The full bifurcation locus is the union 
$\mathcal{T} \sqcup \partial (1/{\mathcal T})$. For this reason, we will only consider $\alpha \in \D$.
\end{rem}

As for the Mandelbrot set, the interior of $\tcal$ consists of hyperbolic and {(perhaps)} non-hyperbolic components; in the latter case, the Julia set supports an invariant line field.

\begin{theo}[Interior components of $\tcal$]\label{th:hyperbolicornon-hyperbolic}
	Let $\SS$ be a connected component of $\inter(\tcal)$. Then, one of the two following statements hold. 
	\begin{enumerate}[{\rm (1)}]
		\item For all $\alpha \in \SS$, $T_\alpha$  has an attracting cycle different from $z=0$, or
		\item for all $\alpha \in \SS$, $T_\alpha$ has an invariant line field.
	\end{enumerate}
	In the first case, we say that $\SS$ is a hyperbolic component; in that case, the multiplier map $\rho: U \to \D^*$ is a universal covering map. In the second case, we say that $\SS$ is a non-hyperbolic component.
\end{theo}

The statement about hyperbolic components will follow  from the results in \cite{FK} (see also \cite[Theorem 4.4]{chen2023slices}). For the second case we shall need the following lemma. 
 Since $\mathcal{A}=\partial\TT $ by Proposition~\ref{prop:activity locus}, by Theorem~\ref{thm:J stability} the Julia set moves holomorphically over any connected component of $\inter(\tcal)$.  

\begin{lem}[Holomorphic motion of $J$] \label{lem:zakeri}
	Let $\SS$ be a connected component of $\inter(\tcal)$. Let $\alpha_* \in \SS$. Then the holomorphic motion of $J(T_{\alpha_*})$  over $\SS$ extends to a holomorphic motion $h_\alpha : \chat \to \chat$ which is holomorphic on $\Omega_{\alpha_*}$.
\end{lem}

To prove Lemma \ref{lem:zakeri}, we make use of the following theorem.
\begin{thm}[{\cite[Main Theorem]{zak2}}]\label{thm:Zakeri}
Let $\{(U_t,c_t)\}_{t\in \D}$  be a family of  disks  in $\hat{\C}$ with  marked points $c_t$ depending holomorphically on $t$. Suppose that the boundaries of these disks undergo a holomorphic motion $h_t: \partial U_0\ra\partial U_t$ over $\D$. Let $\operatorname{rad}(U_t)$ denote the conformal radius of $U_t$ with respect to $c_t$. Then the following conditions are equivalent:
\begin{itemize}
\item[(i)] The map $t\mapsto\log \operatorname{rad}(U_t)$ is harmonic in $\D$.
\item[(ii)]$h_t: \partial U_0\ra\partial U_t$  extends to a
holomorphic motion $h_t: \hat\C\ra\hat\C$ such that $h_t(c_0)=c_t$ for all $t\in\D$ and   such that the restriction $h_t: U_0\ra U_t$  is a biholomorphism for each t.
\end{itemize}
\end{thm}

\begin{proof}[Proof of Lemma~\ref{lem:zakeri}]
Let $C,p, g$ be as in Lemma \ref{lem:confconjH} and recall that they  do not depend on $\alpha$. 
 Since for any $\alpha$ we have that  $J(T_\alpha)=\partial\Omega_\alpha$, 	by  $J$-stability the boundary of the family of punctured  disks $(\Omega_\alpha,0)$   moves holomorphically over $\SS$. 
	By Theorem~\ref{thm:Zakeri}	it is enough to prove that the logarithm of the  conformal radius of $\Omega_\alpha$  with respect to $0$ is harmonic on $\SS$, i.e. that there exists a family of Riemann uniformizing maps $\riem_\alpha: \D \to \Omega_\alpha$ such that   $\riem_\alpha(0)=0$ and  $\alpha \mapsto \riem_\alpha'(0)$ is holomorphic on $\SS$, or equivalently, uniformizing maps $\tilde \riem_\alpha: \H\ra\Omega_\alpha$ such that   $\tilde\riem_\alpha(p)=0$ and that $\alpha \mapsto \tilde \riem_\alpha'(p)$ is holomorphic on $\SS$. 
	
	Let $\ell:=\limn 8^n g^n$ and $\ell_\alpha:=\limn 8^n T_\alpha^n$ denote the respective linearizing coordinates of $g$ and $T_\alpha$.  Notice   that  $\ell'(p)=\ell_\alpha'(0)=1$.  Note also that  
	\[
	 \ell_\alpha: \Omega_\alpha \longrightarrow  \C\setminus \{8^n a_\alpha\}_{n\geq 1}
	\] 
	is a universal covering, where $a_\alpha:=\ell_\alpha (1)$ depends holomorphically on $\alpha$, since $\ell_\alpha$ does. Likewise
	\[
	 \ell: \H\longrightarrow  \C\setminus \{8^n a\}_{n\geq 1}
	\] 
	is also a universal covering, where $a:=\ell(Ci)$ is independent of $\alpha$. Finally define $L_\alpha(z)= \frac{a_\alpha}{a} z$ a linear map which sends the sequence $8^n a$ to $8^n a_\alpha$ for all $n$, and conjugates $z\mapsto z/8$ to itself. Let $\tilde \riem_\alpha:\H\to\Omega_\alpha$ be the lift of $L_\alpha$ under the universal  coverings $\ell$ and $\ell_\alpha$, such that $\tilde \riem_\alpha(p)=0$,
	 that is  
	\[
	\ell_\alpha \circ \tilde \riem_\alpha =L_\alpha \circ \ell.
	\] 
	Since $L_\alpha$ is a homeomorphism, so is $\tilde \riem_\alpha$. Using the chain rule (and recalling $\tilde \riem_\alpha(p)=0$) we compute 
	$\tilde \riem_\alpha'(p)=\frac{L_\alpha'(\ell(0))\ell'(p)}{\ell_\alpha'(0)}=\frac{a_\alpha}{a}.$ It follows that $\tilde\riem_\alpha$ are  Riemann maps sending $p$ to $0$ and such that  $\tilde\riem_\alpha'(p)=a_\alpha/a$, which depends holomorphically on $\alpha$ as required. 
\end{proof}

\begin{proof}[Proof of  Theorem~\ref{th:hyperbolicornon-hyperbolic}]
 Let $\alpha_0\in \SS$. 
By \cite[Corollary G]{ABF}, if  the orbit of $1/\alpha_0$ converges to  an attracting cycle, the same holds for every parameter $\alpha\in \SS$.
 In that case, the period of the attracting cycle is constant throughout $\SS$ and, by \cite[Corollary 6.5]{FK}, the multiplier map $\rho:\SS\to \D^*$ is a holomorphic covering. Moreover, $\SS$ is either simply connected and $\rho$ has therefore infinite degree, or  $\SS$ is conformally equivalent to $\D^*$ and $\lim_{|w|\to 0} \rho^{-1}(w) =\alpha_*$, where $\alpha_*=0$, the only parameter singularity in $ \D(0,1/2)$.  We claim that the latter case is not possible, which finishes the proof of (a). Indeed, elementary calculations show that if $\alpha\in (0,\infty)$, then $T_\alpha$ is continuous, strictly increasing and concave in $\R^+$,  a half line which is invariant under $T_\alpha$. This  implies that $T_\alpha(x)<x$ if $\alpha, x>0$, and therefore  for $\alpha\in [0,\infty)$ we have that $1/\alpha\in\Omega_\alpha$, hence that  $\alpha\notin\TT$. It follows that a neighborhood of $0$ cannot be contained in the Tandelbrot set and this case can be excluded.

Assume now that  the orbit of $1/\alpha_0$ does not  converge to  an attracting cycle. The maps $T_\alpha$ are of finite type hence have no wandering domains.
By the classification of periodic Fatou components and the fact that there are no  Siegel disks or parabolic domains
and no Herman rings (since the maps $T_\alpha$ have at most one free singular value),
we must have $ \C\setminus \Omega_\alpha = J(T_\alpha)$ for all $\alpha \in \SS$.
 By  Theorem~\ref{thm:J stability}, since all singular values are passive in $\SS$,   there is a holomorphic motion $h_\alpha$ of the Julia set of $T_{\alpha_0}$, respecting the dynamics.
By Lemma \ref{lem:zakeri}, this holomorphic motion extends to a holomorphic motion $h_\alpha$ of the Riemann sphere which is holomorphic on $\Omega_{\alpha_*}$.
Let $\mu_\alpha$ denote the Beltrami coefficient of $h_\alpha$.
Then $\mu_\alpha$ is supported in the Julia set. It cannot vanish a.e., for otherwise we would have that all maps in $\SS$ would be linearly conjugate, 
contradicting Lemma \ref{lem:rigidity}. Therefore, $\mu_\alpha/|\mu_\alpha|$ is  an invariant line field. 
\end{proof}

Hence, analogously to the Mandelbrot set, the interior of the Tandelbrot is organized in hyperbolic components  of constant period (also known as {\em shell components}), and (conjecturaly non-existent) non-hyperbolic components.  The component containing $\alpha_{\sharp}$ (the unique real value of $\alpha$ for which $T_{\alpha_{\sharp}}$ is affine conjugate to $g(z)=C\tan(\pi z)$, see Remark \ref{rem:specialmap}), is a component of period one, and we shall refer to it as $\SS_1$ or the {\em main (shell) component}, since it is the largest visible one in Figure \ref{fig:Tandelbrot}.
It is proven in \cite[Thm. 4.4]{GaKo2} that there are no other hyperbolic components of period one. Since every hyperbolic component of period $p$ must have a virtual cycle parameter  of order $p-1$ on its boundary \cite[Cor. 6.4]{FK}, the parameter  $\alpha=0$ must belong to the boundary of $\SS_1$.

 The remaining   missing ingredient  for the proof of  
 Theorem B is the following theorem, which is also the main result in \cite{chen2023slices}. 
\begin{theo}\label{th:msconnected}
	The set $\tcal$ is connected and full.
\end{theo}
We deduce Theorem B assuming Theorem~\ref{th:msconnected} and then dedicate the rest of the section to its proof. 

\begin{proof}[Proof of Theorem B]
By Lemma \ref{lem:tbdd}, $\tcal \subset \D(0,\frac{1}{2})\subset \D$.  Theorem \ref{th:msconnected} shows  that $\tcal$ is connected and full. Finally, by Proposition \ref{prop:Kconnected},  for any tangent-like map $(f,U,V,u,v)$, $J(f)$ is connected if and only if 
 	$v \in K(f)$;  by Lemma \ref{lem:tlmodel}, for any $\alpha \in \D$, there exists domains $U_\alpha, V_\alpha$ such that $(T_\alpha, U_\alpha, V_\alpha, \infty, \frac{1}{\alpha})$ is tangent-like and the Julia set of $T_\alpha$ as a meromorphic map coincides with the Julia set of $T_\alpha$ as a tangent-like map (by Proposition \ref{prop:choicesUV}).
 	It then follows from the definition of $\tcal$ that $\alpha \in \tcal$ if and only if $J(T_\alpha)$ is connected.
\end{proof}

We now prove  Theorem~\ref{th:msconnected}. The approach in \cite{chen2023slices} is to use quasiconformal surgery to prove that the complement of $\tcal \sqcup 1/\tcal$ is a topological annulus; the proof is relatively involved. We will give here a different proof, making use of an approximation of the maps $T_\alpha$ by rational maps and using the existing theory about unicritical polynomial-like maps.

 As usual, let $V=\D(0,2)$. Our goal is to show that the sets
\[
A_n:=\{\alpha \in \D(0,1/2)\mid T_\alpha^n(1/\alpha)\in V\}
\]
are connected and full. Then $\overline{A_n}$ is a decreasing sequence of compact, full, connected  sets. Since
\[
\tcal = \bigcap_{n\geq 0} \overline{A_n},
\]
it follows that $\tcal$ is compact, connected and full. 

We will study the properties of the sets $A_n$ by approximating them with their rational analogues. 

\begin{defi}[A natural family of rational maps]
	For $\alpha \in \C \setminus \{1\}$ and $k \in \N^*$, let $R_{\alpha,k}:=M_\alpha \circ P_k \circ N_\alpha$, where 
	$M_\alpha(z):=\frac{z-1}{\alpha z-1}$, $N_\alpha(z):=\frac{(\alpha-1)z}{8}$, and $P_k(z):=(1+\frac{z}{k})^k$.
	Let $V:=\{ z \in \rs:  |z|>2\}$.
\end{defi}

The following properties are straightforward to check:

\begin{lem}[Properties of $R_{\alpha,k}$] 
Let $R_{\alpha,k}$ be defined as above. Then,
	\begin{enumerate}
		\item $R_{\alpha,k}$ has two critical points of order $k-1$: $\infty$ and $c_{\alpha,k}:=-\frac{8k}{\alpha-1}$. They are mapped to $\frac{1}{\alpha}$ and $1$ respectively.
		\item $0$ is a fixed point for $R_{\alpha,k}$, with multiplier $\frac{1}{8}$.
		\item $R_{\alpha,k}(z) \to T_\alpha(z)$ as $k \to +\infty$, 
		uniformly on every compact of $(\C \setminus \{1\}) \times \C$.
	\end{enumerate}
\end{lem}

We now show that for small values of $\alpha$, maps in this rational family are polynomial-like maps with good properties.  For this we need a definition (see \cite[p.~545]{LyubichBook}).

\begin{defi}[Proper equipped family of polynomial-like maps]\label{def:plfamily}
	Let $\Lambda\subset\C$ be a domain. A holomorphic family of unicritical polynomial-like maps of degree $d$ is a  family of polynomial-like maps
	$f_\la : U_\la \to V_\la$, $\la \in \Lambda$, such that:
	\begin{enumerate}
		\item for all $\la \in \Lambda$, $f_\la$ is unicritical of degree $d$
		\item $\mathcal{U}:=\{(\la,z) : z \in U_\la\}$ and $\mathcal{V}:=\{(\la,z) : z \in V_\la\}$ are domains in $\C^2$
		\item $\mathbf{f}: \mathcal{U} \to \rs$ defined by $\mathbf{f}(\la,z)=f_\la(z)$ is holomorphic.	
	\end{enumerate}
	Such a family is called \emph{equipped}  if there exists $\la_0 \in \Lambda$ and a holomorphic motion $h_\la: \overline{V_{\la_0}} \setminus U_{\la_0} \to \overline{V_\la} \setminus U_\la$ such that for all $\la \in \Lambda$, for all $z \in \partial U_{\la_0}$, 
	\begin{equation}\label{eq:equivarant poly}
	h_\la \circ f_\la(z) =  f_\la \circ h_\la(z).
	\end{equation}
	
	It is called \emph{proper} if
	\begin{enumerate}
		\item $\Lambda$ is a Jordan domain, and $\{f_\la\}_{\la \in \Lambda}$ is the restriction  of a slightly larger  holomorphic family of unicritical polynomial-like maps  $\{f_\la\}_{\la \in \tilde \Lambda}$, where $\Lambda \Subset \tilde \Lambda$. 
		\item For all $\la \in \partial \Lambda$,  $v_\la \in \partial V_\la$, where $v_\la$ is the unique critical value of $f_\la$.
		\item The winding number of
		  $v_\la - c_\la$ around $0$ 
		  is 1, as $\la$ turns once along $\partial \Lambda$ and where $c_\la$ is the unique critical point of $f_\la$.
	\end{enumerate}
\end{defi}

\begin{lem}[Maps $R_{\alpha,k}$ are polynomial-like] 
	There exists $k_0 \in \N$ such that for every $k \geq k_0$, 
	the family  $R_{\alpha,k}: R_{\alpha,k}^{-1}(V) \to V$, $\alpha \in \D(0,\frac{1}{2})$, is a proper equipped  holomorphic family of unicritical polynomial-like maps of degree $k$. 	The  unique critical point is $\infty$, and the unique critical value is 
	$\frac{1}{\alpha}$. 
\end{lem}

\begin{proof}
	Let us prove that there exists $k_0 \in \N$ such that  for all $k \geq k_0$ and for all $\alpha \in \D(0,\frac{1}{2})$, 
	$R_{\alpha,k}: R_{\alpha,k}^{-1}(V) \to V$, $\alpha \in \D(0,\frac{1}{2})$ 
	is a PL map with just $\infty$ as a critical point.
	
	We first claim that there exists $k_0 \in \N$ such that for all $k \geq k_0$ and for all $\alpha \in \D(0,\frac{1}{2})$, $R_{\alpha,k}(\D(0,2)) \Subset \D(0,r)$, for some $0<r<2$. Indeed, this is true for $T_\alpha$, and we know that $R_{\alpha,k}$ converges \emph{uniformly} in $(\alpha,z)$ to $T_\alpha(z)$
	for $(\alpha,z) \in \D(0,\frac{1}{2}) \times \D(0,2)$. 
	This proves  $R_{\alpha,k}^{-1}(V) \Subset V$.
	Moreover, since $R_{\alpha,k}(c_{\alpha,k}) = 1$ and $1 \notin V$, and since $\frac{1}{\alpha} \in V$, $V$ contains exactly one critical value (which is $\frac{1}{\alpha}$) and $ R_{\alpha,k}^{-1}(V) $ contains exactly one critical point (which is $\infty$). 
	Finally, $R_{\alpha,k} :  R_{\alpha,k}^{-1}(V)  \to V$ is a proper map, 
	with exactly one critical value, hence $ R_{\alpha,k}^{-1}(V) $ is simply connected, and is a Jordan domain.
	All this proves that for all $\alpha \in \D(0,\frac{1}{2})$, $R_{\alpha,k} : R_{\alpha,k}^{-1}(V) \to V$, $\alpha \in \D(0,\frac{1}{2})$ 
	is indeed a unicritical PL map, whose critical point is $\infty$ and critical value is $\frac{1}{\alpha}$.

	 Next, let us prove that it is proper. By definition of $V$, the winding property is trivially satisfied.
	 By Lemma \ref{lem:tlmodel}, for any $r \in (\frac{1}{2}, 1)$ we can replace the parameter space $\D(0,\frac{1}{2})$ with $\D(0,r)$ and $V=\rs \setminus \D(0,2)$ with $\rs \setminus \D(0,\frac{1}{r})$.
	 Finally, it is clear from the definition of $V$ that $\alpha \mapsto \frac{1}{\alpha}$ winds once around $\infty$ as $\alpha$ turns once along $\partial \D(0,\frac{1}{2})$.

	 It remains to prove that this family of unicritical polynomial-like maps is equipped. 
	 Let $k \geq k_0$, and let $U_{\alpha,k}:=R_{\alpha,k}^{-1}(V)$. 
	 Let $z_0 \in  \partial U_{0,k}$. We make the following  observation: for all $\alpha \in \D(0,\frac{1}{2})$, for all $z \in\partial V$, $M_\alpha^{-1}(z) \neq \infty$.
	 
	 In particular, the continuous  map
	 $$\D\left(0,\frac{1}{2}\right) \ni \alpha \mapsto M_\alpha^{-1} \circ  R_{0,k}(z_0)$$
	 avoids $\infty$, which is the only critical value of $R_{0,k}$.
	 Therefore, by the lifting property for covering maps, there exists a unique lift $\gamma_{z_0}$  by $R_{0,k}$ of 	 
	 $$ \alpha \mapsto M_\alpha^{-1}  \circ R_{0,k}(z_0)$$
	 such that $\gamma_{z_0}(0)=z_0$.

	 We then let $h_\alpha(z_0):=N_\alpha(\gamma_{z_0}(\alpha)  )$. By construction, we have
	 $$R_{\alpha,k} \circ h_\alpha(z_0) = R_{0,k}(z_0)$$
	 and $h_0(z_0)=z_0$.
	 
	 By the Implicit Function Theorem, $\alpha \mapsto h_\alpha(z_0)$ is holomorphic on $\D\left(0,\frac{1}{2} \right)$.
	 By the uniqueness of the lifting property, the map $z \mapsto h_\alpha(z)$ is injective for all $\alpha \in \D\left(0,\frac{1}{2} \right)$.
	 Thus, $h: \D\left(0,\frac{1}{2} \right) \times \partial U_{0,k} \to \rs$ is  a holomorphic motion, and $h_\alpha(\partial U_{0,k}) = \partial U_{\alpha,k}$.
	 
	 Finally, since $\partial U_{\alpha,k}$ and $\partial V$ are disjoint for every $\alpha \in \D\left(0,\frac{1}{2} \right)$, we may extend $h_\alpha$ to $\partial V$ by setting
	 $$
	 \tilde h_\alpha=\left\{
	 \begin{array}{ll}
	 	h_\alpha(z) & \text{ if $z \in \partial U_{0,k}$ } \\
	 	z &  \text{ if $z \in \partial V$. }
 	\end{array}
	 \right.
	 $$
	 Applying Slodkowski's  extension Theorem~\ref{thm:Slodkowski}  
	  to $\tilde h_\alpha$, we obtain a holomorphic motion of $\rs$ over $\D(0,\frac{1}{2})$, which we still denote by $\tilde h_\alpha$, and which maps $\overline{V} \setminus U_{0,k}$ to $\overline{V} \setminus U_{\alpha,k}$. Moreover, it satisfies \eqref{eq:equivarant poly}  by construction.
\end{proof}

The rational analogues of the sets $A_n$ are the following.

\begin{defi}[Rational approximants of $A_n$]
For $n\geq 0$ and $k\geq 2$, we define
\[ A_{n,k}:=\{\alpha \in \D(0,\frac{1}{2}): R_{\alpha,k}^n(\frac{1}{\alpha}) \in V\}.
\]
\end{defi}

We note the following lemma, from which it follows that for every $k \geq k_0$ and every $n \in \N$, 
$A_{n,k}$ is a Jordan domain. 

\begin{lem}[{\cite[p.~549]{LyubichBook}}]\label{lem:lyubichtransversality}
	Let $\{f_\la\}_{\la \in \Lambda}$ be a proper equipped  holomorphic family of unicritical polynomial-like maps. Then for any $n$ the set $\{ \la \in \Lambda : f_\la^n(c_\la) \in U_\la  \}$ is a Jordan domain.
\end{lem}

\begin{defi}[Subsets of $A_n$]
	For $\delta>0$, we define 
$$
A_n(\delta):=\{ \alpha \in A_n : |T_\alpha^j(\frac{1}{\alpha})|<\frac{1}{\delta}, \quad \forall \ 0 \leq j \leq n \}.
$$
\end{defi}

Observe that $A_n(\delta)\subset A_n$ and clearly, we have $A_n=\bigcup_{\delta>0} A_n(\delta)$.

For any set $X \subset \C$, we denote by $X_\epsilon$ the set $\{z \in \C : d(z,X)\leq \epsilon\}$.

\begin{lem}[$A_{n,k}$ nearly converges to $A_n$]\label{lem:nearcv}
	Let $n \in \N$.  Then,
	\begin{enumerate}
		\item For all $\epsilon>0$ there exists $k_0 \in \N$ such that for all $k \geq k_0$,
		$$\overline{A_{n,k}} \subset \left(\,\overline{A_n}\,\right)_\epsilon$$
		\item For all $\delta>0$, for all  $\epsilon>0$ there exists $k_0 \in \N$ such that for all $k \geq k_0$,
		$$\overline{A_{n}(\delta)} \subset \left(\,\overline{A_{n,k}}\,\right)_\epsilon.$$
	\end{enumerate}
\end{lem}

\begin{proof}
	Let us first prove the first item. Let $\epsilon>0$. 
	We will prove that there exists $k_0 \in \N$ such that for all $k \geq k_0$, the complement of  $\left(\,\overline{A_n}\,\right)_\epsilon$  is contained in the complement of $\overline{A_{n,k}}$. 
	
	Observe that all $\alpha \in \C\setminus \left(\,\overline{A_n}\,\right)_\epsilon$
	are a definite distance away from any virtual cycle parameter.
	By continuity of $(\alpha,z) \mapsto T_\alpha(z)$ away from $z=\infty$, there exists $\eta>0$ such that for all 
	$0 \leq j \leq n$, for all $\alpha \in \C\setminus \left(\,\overline{A_n}\,\right)_\epsilon$:
	\[
	\left|T_\alpha^j \left(\frac{1}{\alpha}\right)\right| \leq \frac{1}{\eta} \quad \text{ and } \quad \left|T_\alpha^n \left(\frac{1}{\alpha}\right)\right| \leq 2-\eta.
	\]

	By uniform convergence of $(\alpha,z) \mapsto R_{\alpha,k}(z)$ to $T_\alpha(z)$ on compact subsets, there exists $k_0 \in \N$ such that for all $k \geq k_0$, for all $\alpha \in  \C\setminus \left(\,\overline{A_n}\,\right)_\epsilon$: 
	\[
	\left|R_{\alpha,k}^n \left(\frac{1}{\alpha} \right)\right| \leq 2.
	\]
	This exactly proves the first item.

	\medskip
	
	Let us now prove the second item. Let $\epsilon>0$ and $\delta>0$.

	Let $\eta>0$ be small enough that if $|R_{\alpha,k}^n(\frac{1}{\alpha})| > 2 - \eta$, then $d(\alpha,\overline{A_{k,n}}) \leq \epsilon$.
	Again, by definition of $A_n(\delta)$, the sequence of maps $\alpha \mapsto R_{\alpha,k}^n(\frac{1}{\alpha})$ converge uniformly on $\overline{A_n(\delta)}$ to $T_\alpha^n(\frac{1}{\alpha})$.
	Therefore, there exists $k_0 \in \N$ such that for all $k \geq k_0$, for all $\alpha \in \overline{A_n(\delta)}$: 
	$$
	\left|R_{\alpha,k}^n \left(\frac{1}{\alpha}\right) \right|>2 - \eta,
	$$
	hence $\alpha \in \left(\,\overline{A_{k,n}} \,\right)_\epsilon$, 
	and we are done.			
\end{proof}

We are finally ready to prove the main goal and hence conclude the proof of Theorem \ref{th:msconnected}.

\begin{prop}
	For any $n \in \N$, the set $\overline{A_n}$ is connected { and full}.
\end{prop}

\begin{proof}
	Assume for a contradiction that there is $n \in \N$ such that $\overline{A_n}$ is disconnected. Then there exists $F_1, F_2$ two disjoint compact sets such that $\overline{A_n} = F_1 \sqcup F_2$.
	Let $\delta>0$ be small enough that $A_n(\delta) \cap F_i \neq \emptyset$,
	$1 \leq i \leq 2$. { Such a $\delta$ exists since $A_n\supset A_n(\delta) \to A_n$ as $\delta\to 0$.} Let $\epsilon:=\frac{1}{4} d(F_1,F_2)$.

	By Lemma \ref{lem:nearcv}, there exists $k_0 \in \N$ such that for all $k \geq k_0$, $\overline{A_{n,k}} \subset \left[\overline{A_n}\right]_\epsilon$. Since $\overline{A_{n,k}}$ is connected, and by our choice of $\epsilon$, we have either 
	$\overline{A_{n,k}} \subset (F_1)_\epsilon$ or $\overline{A_{n,k}} \subset (F_2)_\epsilon$. Assume without loss of generality that 
	$\overline{A_{n,k}} \subset (F_1)_\epsilon$: then
	$d(F_2, \overline{A_{k,n}}) \geq \epsilon$.

	On the other hand, applying the second item of Lemma \ref{lem:nearcv}
	with $\frac{\epsilon}{2}$ : there exists $k_1 \in \N$ such that for all $k \geq k_1$,
	$$\overline{A_{n}(\delta)} \subset \left[\overline{A_{n,k}}\right]_{\epsilon/2}.$$
	Let $k \geq \max(k_0,k_1)$.
	Since $A_n(\delta) \cap F_2 \neq \emptyset$, there exists $\alpha \in F_2$
	such that $d(\alpha, \overline{A_{n,k}}) \leq \frac{\epsilon}{2}$, a contradiction.
	
		It remains to prove that $\overline{A_n}$ is full. Since $A_0=\D(0,1/2)$, the claim is true for $n=0$. Let $n$ be the smallest natural number for which $\overline{A_n}$ is not full, and let $W$ denote a bounded connected component of $\C\setminus \overline{A_n}$, which by assumption is contained in $A_{n-1}$. Then, $T_\alpha^{n-1}(1/\alpha) \in U_\alpha$ for all $\alpha\in W$, which implies that the map $g(\alpha):= T_\alpha^{n}(1/\alpha)$ is meromorphic in $W$ and different from $0$ in $W$, since $\overline{U_\alpha}$ does not contain the zeros of  $T_\alpha$.
		
		By definition of $A_n$, $|g(\alpha|\geq 2$ for all $\alpha\in\partial W\subset A_n$. Hence by the { Minimum Modulus Principle}, the same is true for all $\alpha\in W$, which implies that $W\subset A_n$. 
\end{proof}

\subsection{Hybrid classes in $\tcal$ and the wringing map}\label{sec:hybrid classes}

In this section, we introduce the wringing map $\wring$ and we prove that hybrid equivalence classes in $\mathcal T$ are exactly orbits of $\wring$.  We first prove that any hybrid conjugacy between two tangent maps which  preserves the marking of the poles must be the identity.

\begin{defi}[hybrid conjugate tangent maps]
	For $\alpha_i\in\D$, we say that $T_{\alpha_1}$ and $T_{\alpha_2}$ are hybrid conjugate if there exist tangent-like restrictions $(T_{\alpha_1}, U_1,V_1,\infty,  \frac{1}{\alpha_1})$ and $(T_{\alpha_2}, U_2,V_2,\infty,\frac{1}{\alpha_2})$ which are hybrid conjugate. That is, if  there exists  a quasiconformal homeomorphism $\conj: V_1 \to V_2$ such that for all $z \in U_1$, 
	$$
	\conj \circ f_1(z) = f_2 \circ \conj(z)
	$$
	and $\dbar \conj=0$ a.e. on $K(T_{\alpha_1})$.
\end{defi}
\noindent Notice that necessarily $\conj(v_1)=v_2$, and $\conj(\infty)=\infty$.
 
\begin{prop}\label{prop:crucial}
	Let $\alpha_1, \alpha_2 \in \mathcal T$, and assume that there exists a hybrid conjugacy $\conj$ between $T_{\alpha_1}$ and $T_{\alpha_2}$ mapping the marked pole $p(\alpha_1)$ to the marked pole $p(\alpha_2)$ (see Lemma \ref{lem:labelling}).
	Then $\alpha_1=\alpha_2$ and $\conj$ is the identity.
\end{prop}

\begin{proof}
	Let $\riem: \Omega(T_{\alpha_1}) \to \Omega(T_{\alpha_2})$, where $\riem = \riem_{\alpha_2} \circ \riem_{\alpha_1}^{-1}$ and 
	the $\riem_{\alpha_i}$ are the Riemann maps of $\Omega(T_{\alpha_i})$ used in Lemma \ref{lem:confconjH}, which conjugate $T_{\alpha_i}$ to $g(z)=C \tan(\pi z)$. Then   the map $\riem$ is a holomorphic conjugacy between $T_{\alpha_1}$ and $T_{\alpha_2}$ in the basin of 0.

	Let $\tilde h:=\riem_{\alpha_2}^{-1} \circ \conj \circ \riem_{\alpha_1}$. By diagram chasing, $\tilde h$ commutes with $g$ on some strip on the upper half plane  having $\R$ as part of its boundary. Since it is quasiconformal, it extends continuously to a homeomorphism of $\R$. 
	
 Since $\conj( p(\alpha_1)) = p(\alpha_2)$, we have $\tilde h(\frac{1}{2}) = \frac{1}{2}$. 	By Lemma \ref{lem:noaut}, $\tilde h=\id$ on $\R$. By Proposition \ref{prop:extend}, $h:=\riem^{-1}\circ \conj= \riem_{\alpha_1}\circ \riem_{\alpha_2}^{-1} \circ \conj$ extends continuously to $J(T_{\alpha_1})$ by the identity. 
	Hence the map defined by $\conj$ on $K(T_{\alpha_1})$ and $\riem_{\alpha_2}^{-1} \circ \riem_{\alpha_1}$ on $\Omega(T_{\alpha_1})$ is continuous. Then, by Rickman's lemma \cite{ric,douhub}, it is quasiconformal. Since its $\dbar$ vanishes a.e.,  Weyl's lemma implies that it is holomorphic, hence affine. Then,  by Lemma \ref{lem:rigidity}, $\conj$ is the identity and $\alpha_1=\alpha_2$.
\end{proof}
 Nevertheless, we will  see that (with the exception of $\alpha=0$) hybrid conjugacy classes have infinitely many elements in $\mathcal{T}$, all related by a map which we call the {\em wringing map} and we denote by $\wring$.

To that end, we will need to work with a larger family than $T_\alpha$. More precisely,  let us define the two-dimensional {\em generalized tangent family} of maps, consisting of maps with the same properties as $T_\alpha$ ( marked fixed point at $z=0$, asymptotic values at $1$ and $1/\alpha$ and no other singular values), but with the multiplier of the fixed point $0$ allowed to be any arbitrary number $\rho\in\D\setminus \{0\}$. We shall denote these maps by  $T_{\alpha,\rho}$ and they can be written as 
	\begin{equation} \label{eq:enlargedmodel}
	T_{\alpha,\rho}(z):= \frac{e^{(\alpha-1)\rho z}-1}{\alpha e^{(\alpha-1)\rho z}-1},
	\end{equation}
where $\alpha \in \C \setminus \{1\}$. This makes $T_\alpha=T_{\alpha,1/8}$. 

Analogously, we can then define the $\rho-$Tandelbrot set as 
\[
\mathcal T_\rho:=\{\alpha\in \D \mid T_{\alpha,\rho}(1/\alpha)\not\rightarrow 0\}.
\]

Most of the results in Section \ref{sec:firstproperties} hold independently of $\rho$, but we shall  use only Lemma  \ref{lem:rigidity}  for which this claim is obvious.
\subsubsection{Definition of the wringing map}\label{sec:wringing} \ 

 The idea is as follows. By defining and integrating a new invariant complex structure in the basin $\Omega_\alpha$,  in order to change the multiplier of $0$, we shall construct a continuous path of hybrid conjugate maps in the generalized family of the form $T_{A(\alpha,t), \rho_t}$, where $t\in \R$ and $\rho_t=\frac18 e^{2\pi i t}$. For every $t \in \Z$, the resulting map is in the Tandelbrot set. The wringing map is then the result of doing exactly one turn before landing again in $\mathcal T$, that is, $\wring(\alpha)=A(\alpha,1)$, see Figure \ref{fig:wringingpath}.   
 
 \begin{figure}[htb!]
\centering
\includegraphics[width=0.35\textwidth]{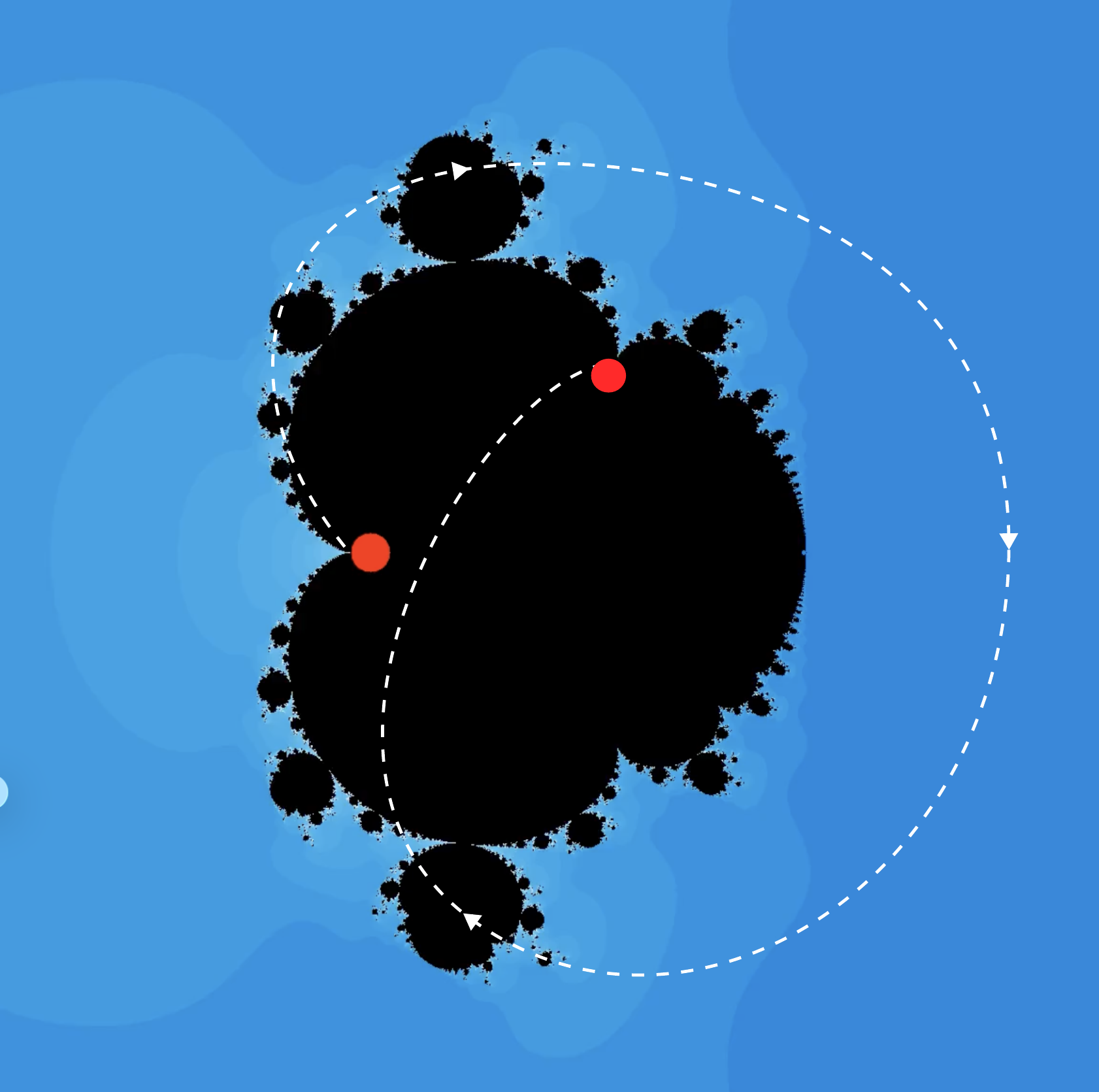}
\setlength{\unitlength}{0.35\textwidth}
\put(-0.68,0.53){\scriptsize \color{white}{$\alpha$}}
\put(-0.5,0.58){\scriptsize \color{white}{$\wring(\alpha)$}}
\put(-0.35,0.83){\scriptsize \color{white}{$A(\alpha,t)\in \mathcal T_{\rho_t}$}}

\caption{\small \label{fig:wringingpath} Sketch of the projection of the path $A(\alpha,t) \in \mathcal T_{\rho_t}$ down to the $\mathcal T_{\frac18}$ plane, for $t\in [0,1]$. The path joins $\alpha$ and $W(\alpha)$. }
\end{figure} 
 
 We make this idea precise in the following Proposition.

\begin{propdef}[The wringing map]\label{prop:wringing}
Let $\rho_t:=\frac18 e^{2\pi i t}$. There exists a map 
\[
\begin{array}{rccccc}
A:& \mathcal T & \times &\R &\longrightarrow & \C^2\\
&\alpha&,&t&\longmapsto & A(\alpha,t) \in \mathcal T_{\rho_t}
\end{array}
\]
satisfying:
\begin{enumerate}[\rm (a)]
\item $A$ is continuous in $t$;
\item $A(\alpha,0)=\alpha$;
\item There exists a hybrid conjugacy $h_t:=h_{\alpha,t}$ between $T_\alpha$ and $T_{A(\alpha,t),\rho_t}$;
\item  For every $k,n\in\Z$, $A(\alpha,k+n)=A(A(\alpha,k), n)$, and  $h_{\alpha,k+1}  =h_{A(\alpha,k),1} \circ h_{\alpha,k}$.
\end{enumerate}
We define the {\em wringing map} as 
\[
\begin{array}{rccc}
\wring:& \mathcal T  &\longrightarrow & \mathcal T\\
&\alpha&\longmapsto & A(\alpha,1) 
\end{array}
\]
which consequently is bijective of inverse $\alpha \mapsto A(\alpha,-1)$.  Moreover, for all $n\in\Z$,  $\wring^n(\alpha)=A(\alpha,n)$  and
$
h_{\alpha,n+1}=h_{\wring^{n}(\alpha),1}\circ h_{\alpha,1}.
$
\end{propdef}

The remainder of this subsection is dedicated to prove Proposition \ref{prop:wringing}. For more details about  the surgery construction, we refer the reader  to \cite[Sections 1.5 and 4.1]{bf}.

Fix $\alpha\in \mathcal T$, and  let $D=D_\alpha=\riem_\alpha^{-1}(\D)$, where $\riem_\alpha$ is a linearizing coordinate conjugating $T_\alpha$ in $D$ to $z\mapsto\frac{z}{8}$ in $\D$. Then $D$ is a Jordan domain containing $0$ and $T_\alpha$ maps $D$ compactly and univalently into itself.   

Let $\D^*=\D\setminus\{0\}$ and $\H_l$ denote the left half plane and observe that  $\exp:\H_l \to \D\setminus \{0\}$, (semi)conjugates $z\mapsto z+\nu_0$  to $z\mapsto \frac{z}{8}$, where $\nu_0=\ln\frac18$.

Now fix any $t\in\R$ and define $\nu_t:=\ln\frac18 + 2\pi i t \in\H_l$. Consider the unique $\R^2-$linear map $L_t$ that maps the rectangle $R_0=[\nu_0,0]\times[0,2\pi i]$ to the parallelogram $R_t=\{r \nu_t + s 2 \pi i; r,s\in[0,1)\}$. This map can be explicitely computed as
\begin{equation}\label{eq:Lt}
L_t= \frac{\nu_t+\overline{\nu_0}}{\nu_0+\overline{\nu_0}} z +\frac{\nu_t-\nu_0}{\nu_0+\overline{\nu_0}} \overline{z} =
\frac{t\pi i +\nu_0}{\nu_0} z + \frac{t \pi i}{\nu_0} \overline{z}.
\end{equation}
(see \cite[Sec. 1.5]{bf} for details), it is the identity on the imaginary axis and conjugates $z\mapsto z+\nu_0$ to $z\mapsto z+\nu_t$ in $\H_l$. See Figure \ref{fig:paral}.

Let $\xi_t=\exp\circ L_t \circ \Log$, where $\Log$ denotes the principal branch of the logarithm, and note that it is a quasiconformal map conjugating $z\mapsto \frac18 z$ to $z\mapsto \rho_t$ in $\D$. It is important to realize that, even when $t\in \Z$  (so that $\rho_t=\frac18$), this conjugacy is not the identity. For example, for $t=1$ it performs a {\em Dehn twist} in $\D$: it fixes the points in $\partial \D$ and does a full turn at every subsequent circle $\{|z|=1/8^n, n\in \N\}$ (see Fig. \ref{fig:paral2group}).

 We use this map to define a new complex structure in $\D^*$ namely $\widetilde \mu_t=\frac{\overline{\partial}\xi_t}{\partial \xi_t}$, which can also be explicitely computed, giving
\[
\widetilde\mu_t(z)=\frac{\nu_t-\nu_0}{\nu_t+\overline{\nu_0}}\, \frac{z}{\overline{z}}=\frac{\pi i t}{\pi i t +2\nu_0}\, \frac{z}{\overline{z}}.
\]
By construction, $\widetilde\mu_t$ is invariant under multiplication by $1/8$ and has bounded dilatation.  
Finally we use the linearizing coordinate $\riem_\alpha$ to pull back $\widetilde\mu_t$ to $D$, defining $\mu_t:=\mu_{\alpha,t}:=\riem_{\alpha}^*\widetilde\mu_t$, a $T_\alpha$ invariant complex structure with bounded dilation. 

\begin{figure}[htb!]
\centering
\includegraphics[width=0.65\textwidth]{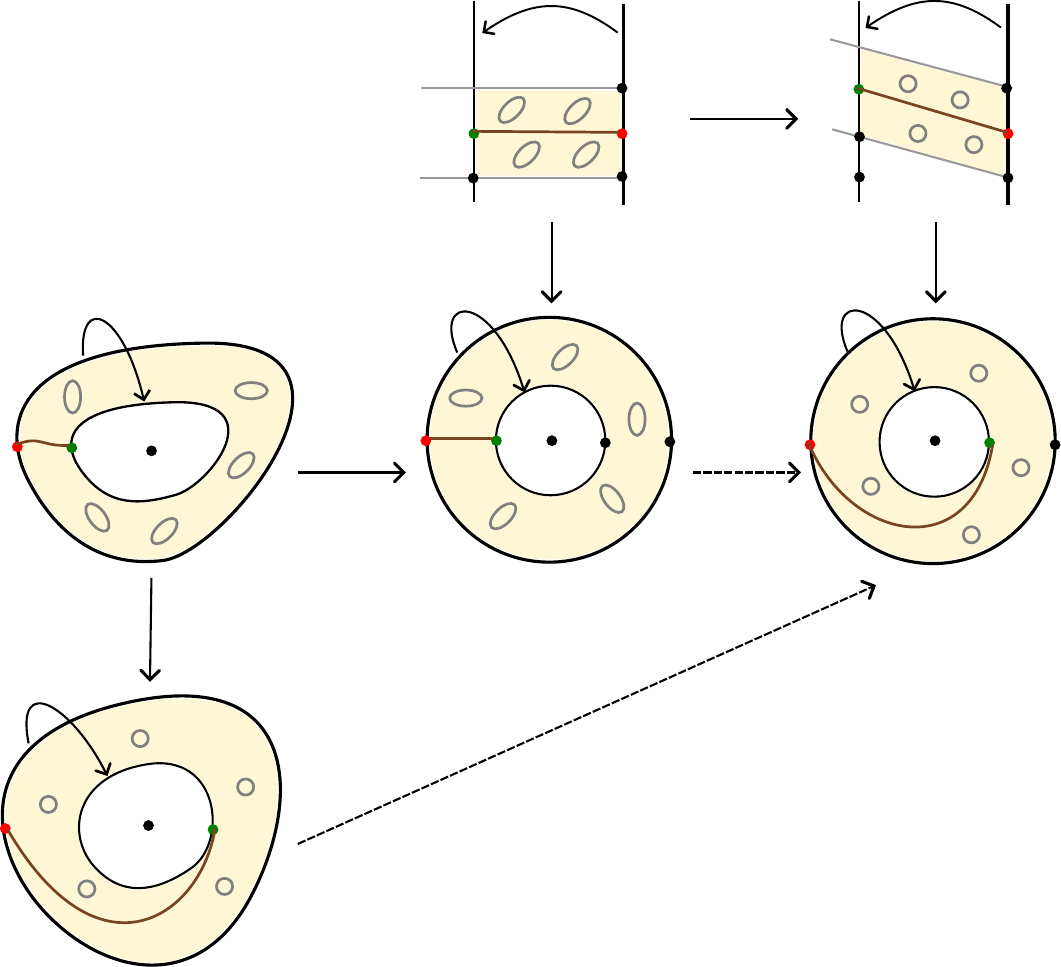}
\setlength{\unitlength}{0.65\textwidth}
\put(-0.48,0.2){\small ${\riem}_{A(\alpha,t)}$}
\put(-1.08,0.27){\small $T_{A(\alpha,t),\rho_t}$}
\put(-0.85,0.31){\small $h_{t}$}
\put(-0.95,0.62){\small $T_\alpha$}
\put(-0.85,0.6){\small $D_\alpha$}
\put(-0.85,0.48){\scriptsize $0$}
\put(-0.69,0.48){\small ${\riem}_\alpha$}
\put(-0.62,0.62){\small $\frac18 z$}
\put(-0.47,0.488){\scriptsize $0$}
\put(-0.11,0.488){\scriptsize $0$}
\put(-0.85,0.12){\scriptsize $0$}
\put(-0.402,0.735){\scriptsize $0$}
\put(-0.402,0.82){\scriptsize $2\pi i$}
\put(-0.04,0.735){\scriptsize $0$}
\put(-0.04,0.82){\scriptsize $2\pi i$}
\put(0.01,0.486){\scriptsize $1$}
\put(-0.26,0.62){\small $\rho_t z$}
\put(-0.315,0.81){\small $L_t$}
\put(-0.315,0.48){\small $\xi_t$}
\put(-0.475,0.66){\scriptsize $\exp$}
\put(-0.11,0.66){\scriptsize $\exp$}
\put(-0.23,0.735){\scriptsize $\nu_0$}
\put(-0.23,0.77){\scriptsize $\nu_t$}
\put(-0.59,0.735){\scriptsize $\nu_0$}
\put(-0.51,0.915){\scriptsize $z+\nu_0$}
\put(-0.15,0.915){\scriptsize $z+\nu_t$}
\put(-0.4,0.6){\small \textcolor{gray}{$\widetilde{\mu}_t$}}
\put(-0.74,0.6){\small \textcolor{gray}{$\mu_t$}}
\caption{\small \label{fig:paral} Sketch of the construction to change the multiplier of 0. }
\end{figure} 

Spread $\mu_t$ to the whole $\Omega_\alpha$, the  basin of attraction of $0$, using the dynamics of  $T_\alpha$, and define it to be zero on $K(T_\alpha)$. Denote the extended complex structure again by $\mu_t$ and observe that, by construction, $\mu_t$ has bounded dilation, is $T_\alpha$ invariant, and varies continuously with respect to $t$ (see also \cite[Sec. 4.1]{bf} for details). 

Let $h_t=h_{\alpha,t}$ be the integrating map of $\mu_t$ given by the Measurable Riemann Mapping Theorem, normalized so that it fixes $0,1$ and $\infty$. Note that $h_t$ is quasiconformal in $z$ and continuous in $t$. Then, the map  $F_t:=h_t \circ T_\alpha \circ h_t^{-1}$ is meromorphic, with asymptotic values $1$ and $h_t(1/\alpha)$, and a fixed point at $0$. 

Observe that the map $\exp \circ  L_t \circ \Log \circ \riem_\alpha  \circ h_t^{-1}$ is conformal and conjugates $ F_t$ in $h_t(D)$ to $z\mapsto \rho_t z$ in $\D$. It follows that this is the new linearizing coordinate and that the multiplier of $0$ is  $\rho_t$. We then deduce from Lemma \ref{lem:rigidity} (changing $1/8$ by $\rho_t$) that 
\begin{equation}\label{def:alpha}
F_t= h_t \circ T_\alpha \circ h_t^{-1} = T_{A(\alpha,t),\rho_t},
\end{equation}
for some $A(\alpha,t)\in \mathcal T_{\rho_t}$, which must satisfy  $A(\alpha,t):=1/h_t(1/\alpha)$ and is hence continuous in $t$, proving (a). 
Observe that since $\overline{\partial} h_t=0$ on $K(T_\alpha)$, the quasiconformal conjugacy $h_t$ between $T_\alpha$ and $T_{A(\alpha,t),\rho_t}$ is actually a hybrid conjugacy, proving (c). Clearly $A(\alpha,0)=\alpha$ since $L_0=\Id$, $\mu_0=0$ and hence $h_0=\Id$, proving (b). 

It remains to prove (d). We first show that $A(\alpha, k+1)=A(A(\alpha,k),1)$ for all $k\in\Z$. So suppose we have performed the surgery for $t=k$ obtaining $A(\alpha,k)$. As observed above, the new linearizing coordinate in the Jordan domain $h_{\alpha,k}(D_\alpha)$ is precisely 
\[
\riem_{A(\alpha,k)}:=\xi_k  \circ \riem_\alpha  \circ h_{\alpha,k}^{-1}.
\]

\begin{figure}[htb!]
\centering
\includegraphics[width=0.95\textwidth]{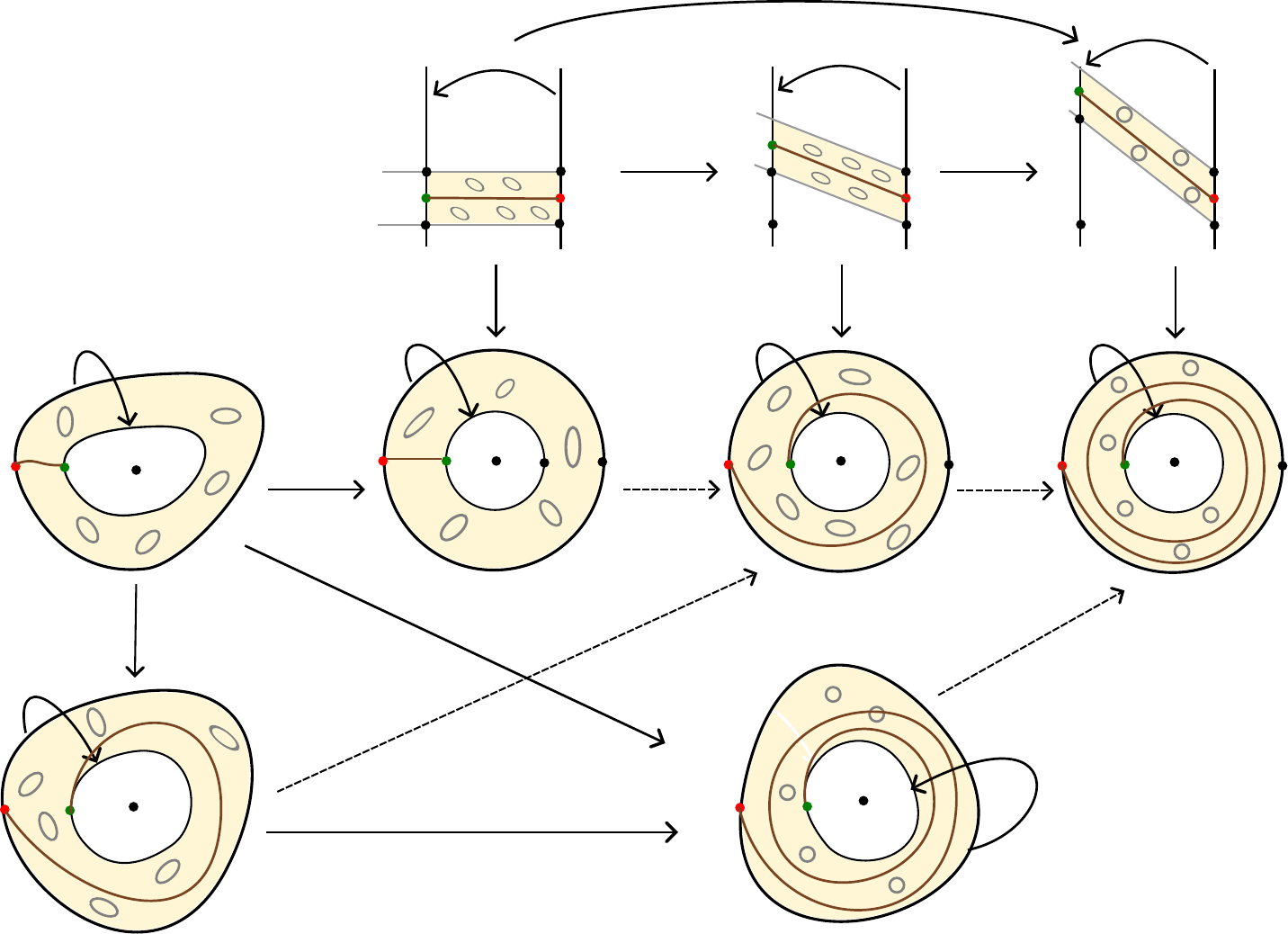}
\setlength{\unitlength}{0.95\textwidth}
\put(-1.01,0.2){\small $T_{A(\alpha,1)}$}
\put(-0.89,0.23){\small $h_{\alpha,1}$}
\put(-0.73,0.22){\small $h_{\alpha,2}$}
\put(-0.49,0.22){\small $\riem_{A(\alpha,1)}$}
\put(-0.65,0.05){\small $h_{A(\alpha,1),1}$}
\put(-0.175,0.22){\small $\riem_{A(\alpha,2)}$}
\put(-0.19,0.12){\scriptsize $T_{A(A(\alpha,1),1)}=T_{A(\alpha,2)}$}
\put(-0.955,0.46){\small $T_\alpha$}
\put(-0.7,0.47){\small $\frac18 z$}
\put(-0.43,0.47){\small $\frac18 z$}
\put(-0.18,0.47){\small $\frac18 z$}
\put(-0.85,0.445){\small $D_\alpha$}
\put(-0.76,0.36){\small $\psi_\alpha$}
\put(-0.49,0.36){\small $\xi_1$}
\put(-0.23,0.36){\small $\xi_1$}
\put(-0.5,0.6){\small $L_1$}
\put(-0.25,0.6){\small $L_1$}
\put(-0.4,0.735){\small $L_2$}
\put(-0.08,0.49){\scriptsize $\exp$}
\put(-0.34,0.49){\scriptsize $\exp$}
\put(-0.61,0.49){\scriptsize $\exp$}
\put(-0.595,0.36){\tiny $\frac18$}
\put(-0.525,0.36){\tiny $1$}
\put(-0.34,0.355){\scriptsize $0$}
\put(-0.08,0.355){\scriptsize $0$}
\put(-0.885,0.355){\scriptsize $0$}
\put(-0.61,0.355){\scriptsize $0$}
\put(0,0.355){\scriptsize $1$}
\put(-0.325,0.095){\scriptsize $0$}
\put(-0.89,0.09){\scriptsize $0$}
\put(-0.63,0.68){\scriptsize $z+\nu_0$}
\put(-0.37,0.68){\scriptsize $z+\nu_1$}
\put(-0.13,0.7){\scriptsize $z+\nu_2$}
\put(-0.56,0.545){\tiny $0$}
\put(-0.56,0.59){\tiny $2\pi i$}
\put(-0.29,0.545){\tiny $0$}
\put(-0.29,0.59){\tiny $2\pi i$}
\put(-0.05,0.545){\tiny $0$}
\put(-0.05,0.59){\tiny $2\pi i$}
\put(-0.7,0.54){\tiny $\nu_0$}
\put(-0.425,0.547){\tiny $\nu_0$}
\put(-0.425,0.58){\tiny $\nu_1$}
\put(-0.19,0.547){\tiny $\nu_0$}
\put(-0.185,0.625){\tiny $\nu_2$}
\put(-0.56,0.44){\small \textcolor{gray}{$\widetilde{\mu}_2$}}
\put(-0.28,0.44){\small \textcolor{gray}{$\widetilde{\mu}_1$}}
\put(-0.8,0.44){\small \textcolor{gray}{$\mu_{\alpha,2}$}}
\put(-0.8,0.18){\small \textcolor{gray}{$\mu_{A(\alpha,1),1}$}}
\caption{\small \label{fig:paral2group} Sketch of the construction to prove statement (d), in the case $k=1$. }
\end{figure} 

If we apply the surgery to $A(\alpha,k)$ for $t=1$ we will add to the diagram the map $\xi_1=\exp \circ L_1\circ \Log$ to build the new complex structure $\mu_{A(\alpha,k),1}$. Now, reading the diagram, cancelling $\Log$ and $\exp$ and deducing from \eqref{eq:Lt} that $L_1\circ L_k=L_{k+1}$, we obtain that 
\[
\mu_{A(\alpha,k),1}=  \riem_{A(\alpha,k)}^* \widetilde{\mu}_1 = \left(\exp \circ L_{k+1} \circ \Log \circ \riem_\alpha\circ h_{\alpha,k}^{-1}\right)^*\mu_0 = (h_{\alpha,k}^{-1})^* \mu_{\alpha,k+1}.
\]
Hence, by uniqueness of the integrating maps after normalization,
\[
 h_{A(\alpha,k),1}= h_{\alpha,k+1} \circ h_{\alpha,k}^{-1}
 \]
 from which it follows that $A(\alpha,k+1)=A(A(\alpha,k),1)$. The general case is proven from this equality by  induction in $n$, and we leave this detail to the reader. 

\subsubsection{Quasiconformal conjugacy classes and hybrid classes }\label{sec:qc and hybrid classes}\ 

In this section we describe quasiconformal conjugacy classes and hybrid classes. It will turn out that the latter are exactly orbits of the wringing map. 

\begin{prop}[Quasiconformal conjugacy classes] \label{prop:quasiconformalclasses}
Let $\alpha_0 \in \partial \mathcal T$ and $\alpha_1 \in \mathcal T$. Assume that there exists a quasiconformal homeomorphism $\conj: \rs \to \rs$ conjugating $T_{\alpha_0}$ to $T_{\alpha_1}$, and which is conformal on $\Omega(T_{\alpha_0})$. Then $\alpha_0=\alpha_1$. 
\end{prop}
\begin{proof}
 Let $\mu$ be the Beltrami form induced by $\conj$ on the dynamical plane of $T_{\alpha_0}$, and consider the holomorphic family of Beltrami forms $\{t \mu, t\in D:=\D\left(0, \|\mu\|_{\infty}^{-1} \right)\}$, which still equal 0 in the basin $\Omega(T_{\alpha_0})$. Let $\conj_t$ be the integrating maps of $t\mu$ fixing $0,1$ and $\infty$. Then the maps $\conj_t$ are conformal on $\Omega(T_{\alpha_0})$ and  depend holomorphically on $t$. Hence they conjugate $T_{\alpha_0}$  to a family of tangent maps $T_{\alpha(t)}$ (since the multiplier of $0$ is unchanged) such that $t\mapsto \alpha(t)$ is holomorphic, $\alpha(0)=\alpha_0$ and $\alpha(1)=\alpha_1$. 
 Since $\alpha_0$ is in the boundary of $\mathcal T$, the set $\alpha(D)$ cannot be open; and since the map $\alpha$ is holomorphic, it must be constant.
\end{proof}

The rest of this section is devoted to proving the following result. We shall use the notation of Proposition \ref{prop:wringing}.

\begin{theo}[Hybrid equivalence classes]\label{th:unicity}
	Hybrid equivalence classes in $\mathcal T$ are exactly orbits of the wringing map $\wring$.  Moreover, if $\alpha, \hat \alpha\in \mathcal T$, any  hybrid conjugacy between $T_\alpha$ and $T_{\hat\alpha}$  is a restriction of the map $h_{k}$ for the unique $k\in\Z$ such that $\hat\alpha=\wring^k(\alpha)$. 
\end{theo}

The main tools are  the following proposition and Propositon~\ref{prop:crucial}.   Recall the notation of  Lemma~\ref{lem:confconjH}:  for any parameter $\alpha\in\TT$,  $\riem_\alpha:\H\to \Omega_\alpha$ is the conformal conjugacy between $T_\alpha$  and $g(z)=C\tan(\pi z)|_\H$, and $p_j(\alpha)$ be the poles labeled as in \eqref{eq:labelling}. By Lemma~\ref{lem:labelling}, $\omega_\alpha$ has a radial limit at  $\infty$ and at the poles. Recall as well from Proposition~\ref{prop:wringing} that $h_1$ is the hybrid conjugacy between $\alpha$ and $\wring(\alpha)$.

\begin{prop} \label{prop:WinH}
	For $\alpha \in \mathcal T\setminus \{0\}$, consider the self map of the upper half plane $\H$ defined as 
\[
H:=\riem_{\wring(\alpha)}^{-1} \circ h_1 \circ \riem_\alpha,
\] 

Then \begin{itemize}
\item $H$  is a quasiconformal homeomorphism which commutes with $g$, and whose continuous extension to $\mathbb R$ satisfies   $H(\frac{2j+1}{2}) = \frac{2j-1}{2}$ for any $j\in\Z$
\item $H_{|_\R}$ generates the group of homeomorphisms which commute with $g$
\item $h_1(p_j(\alpha))=p_{j-1}(\wring(\alpha))$ for every $j\in\Z$.
\end{itemize} 
\end{prop}
The fact that 	 $H_{|\R}$ generates the group of homeomorphisms which commute with $g$ is a direct consequence of  Corollary \ref{coro:qc(g)}.

The proof of Proposition~\ref{prop:WinH} relies on several steps. The first one is to construct a model map on the attracting basin  also for maps in the family \eqref{eq:enlargedmodel}, whose multiplier is non-real.

In Section~\ref{sec:confconj}, since the multiplier was real, it was possible and convenient to choose that both the asymptotic value and  the fixed point for  the model map $g$ lie in the imaginary axis. Now, since the multiplier is complex, it is no longer possible to assume that both the asymptotic value and the fixed point of a model map  lie in the same geodesic ray connecting the repelling periodic point  $0$ to the essential singularity $\infty$. So we will look for a model map $\hat g_\rho:\H\ra\H$  with one asymptotic value, and with the normalization that $i$ is a fixed point  for  $\hat g_\rho$ of multiplier  $\rho$ and that infinity is an essential singularity.   This will enforce an explicit expression for  the singular value $\hat v_\rho$ of $\hat g_\rho$  (independent of $\alpha$).  We then show that we can conjugate $T_{\alpha,\rho}$ to $\hat g_\rho$ on its attracting basin of 0, or in other words, that for each $T_{\alpha,\rho}$ we can find a conformal map $\hat \riem_{\alpha, \rho}:\Omega_{\alpha,\rho}\ra\H$ that maps  the fixed point $0$ to $i$ and the singular value $1$ to  $\hat v_\rho$.

For $t\in [0,1)$ let 
$$
M_t(z)=\frac{\cos(\pi t) z+\sin(\pi t)}{-\sin(\pi t)z+\cos(\pi t)}.
$$

This is a hyperbolic rotation around $i$ of angle $2\pi t$ or in other words, the map $M_t$ is conjugate to the rotation $e^{2\pi it}z$ on the disk $\D$ via the Moebius map $L(z)=-\frac{z-i}{z+i}:\H\ra\D$, which gives immediately that $M_t(i)=i$ and that $M_t'(i)=e^{2\pi it}$. 

\begin{lem}\label{lem:model for enlarged family}
	Let $\rho \in \D^*$. Then  the map 
$$
\hat g_{\rho}:=M_t\circ\frac{\tan(cz)}{\tanh(c)}
$$
is a universal cover $\hat g_{\rho}:\H\ra\H\setminus\{\hat v_\rho\}$,  where $\hat v_\rho:=M_t(\frac{i}{\tanh(c)})$. Moreover, there is  a unique  $c>0$ (depending only on $|\rho|$)  such that  $\hat g_{\rho}$
	fixes $i$ with multiplier $\rho=|\rho|e^{2\pi i t}$.
\end{lem}

\begin{proof}
Since $M_t$ is a conformal map and $i$ is the unique singular value for $\tan(cz)$, the claim about the covering properties is immediate.

Fix $\rho=|\rho|e^{2\pi i t}$.  Consider the map $\tan(cz)$ with $c=c(\rho)>0$ to be determined. Then $\tan(ci)=i\tanh(c)$, hence $i$ is a fixed point for  the map $\frac{\tan(cz)}{\tanh(c)}$, unfortunately, with multiplier   $\frac{2c}{\sinh(2c)}$.  By imposing  $|\rho|=\frac{2c}{\sinh(2c)}$, which determine uniquely  $c>0$, we obtain that the multiplier is $|\rho|$; we still need to fix the argument by  postcomposing  with an automorphism $M$ of $\H$ that fixes $i$ and has derivative  $M'(i)=e^{2\pi i t}$. The map $M$ needs to be an automorphism of $\H$ so that  the  resulting map has exactly one singular value, so in fact $M$ has to be a  Riemann map, which is prescribed uniquely by its value and its derivative at the point  $i$.  So    $M=M_t$ for $t$ satisfying  $e^{2\pi i t}=\frac{\rho}{|\rho|}$, which determines uniquely  $t\in[0,1)$. 
\end{proof}

We claim that $\hat g_\rho$ is unique in the following sense. 

\begin{prop}\label{lem: uniqueness of grho}
The map $\hat g_\rho$ is the unique self-map of $\H$ of infinite degree, with only one singular value, for which $i$ is an attracting fixed point of multiplier $\rho$ and $\infty$ is the  essential singularity.
\end{prop}
\begin{proof}
By Lemma~\ref{lem:innerunique}, any other map $\tilde g$ must be obtained by $\hat g_\rho$ by conjugating with a Moebius automorphism of $\H$ which fixes $i$  and infinity, hence such an automorphism must be the identity   (a surjective conformal map from $\H\ra\H$ is a Riemann map hence uniquely characterized by the image of a point inside and a point in the boundary).
\end{proof}

 We now need to relate $\hat g_\rho$ with $T_{\alpha,\rho}$, with $\alpha\in \TT_\rho$. We want to use a Riemann map from $\H$ to $\Omega_{\alpha,\rho}$ (the basin of $0$ under $T_{\alpha,\rho}$), normalized so that it maps $i$ to $0$ and ``the essential singularity to $\infty$''. To make sense of this last condition we need to rephrase it in terms of accesses to $\infty$. Hence, we consider any Jordan domain $D\subset \Omega_{\alpha,\rho}$ containing $0$ and $1$ and such that $T_{\alpha,\rho}(\overline{D}) \subset D$. One can see that this set exists considering linearizing coordinates around 0 and pulling back a domain until it contains $1$. Then, $T:=T_{\alpha,\rho}^{-1}(D)$ is a logarithmic tract over $1$ containing $\overline D$, which  defines a unique  access to infinity from $\Omega_{\alpha,\rho}$, say $\mathcal A$ (the homotopy class of any asymptotic path in $T$). Under the Riemann map, accesses are in one-to-one correspondence with points in $\partial H$ for which the radial limit exists, so we shall require that the access $\mathcal A$ is send to $\infty$.

\begin{cor}[Conjugacy on the basin of $0$ for generalized tangent maps]
Let $T_{\alpha,\rho}$ be in the family \eqref{eq:enlargedmodel},  with $\alpha\in\TT_\rho$, and $\hat g_\rho$ be as in Lemma~\ref{lem:model for enlarged family}.  Let $\mathcal A$ be the access to $\infty$ from the attracting basin of $0$,  $\Omega_{\alpha,\rho}$, defined by a (logarithmic) tract over the asymptotic value $1$. Then, $T_{\alpha,\rho}$ is holomorphically conjugate to $\hat g_\rho$ on the attracting basin of $0$,  $\Omega_{\alpha,\rho}$, by the unique Riemann map $\hat\riem=\hat\riem_{\alpha,\rho}:\H\ra\Omega_{\alpha,\rho}$ that  maps $i$ to $0$ and sends the access $\mathcal A $ to $\infty$. 
\end{cor}

\begin{proof}
Let $\tilde g:=\hat \riem^{-1}\circ T_{\alpha,\rho}\circ\hat \riem:\H\ra\H\setminus \hat \riem^{-1}(1)$. It is an infinite degree  covering, with infinity as an essential singularity, and a fixed point at $i$ with multiplier $\rho$. Hence by Lemma~\ref{lem: uniqueness of grho}, $\tilde g=\hat g_\rho$.
\end{proof}

We will need to keep track of the poles of $\hat g_\rho$ as $\rho$ turns once around $0$.
Let  $c$ be as in Lemma~\ref{lem:model for enlarged family}, and let
\begin{equation}\label{eq: poles g enlarged model}
\hat{p_j}:=\frac{\pi}{2c}+j \frac{\pi}{c}, \quad j \in \Z
\end{equation}
 be the poles of $\hat{g}_{1/8}$.

In the following Lemma we shall use the notation of Proposition/Definition~\ref{prop:wringing}.

\begin{lem}[Shifting poles]\label{lem:shifting poles generalized model}
Let $\alpha\in\TT\setminus\{0\}$, $\wring(\alpha)$ be the  image  of $\alpha$ under the wringing map. Let $\hat\riem_{\alpha}:\H\ra\Omega_\alpha$, $\hat\riem_{\wring(\alpha)}: \H\ra\Omega_{\wring(\alpha)}$ be the Riemann maps conjugating $T_\alpha, T_{\wring(\alpha)}$ to $\hat g _{1/8}$ as given by Lemma~\ref{lem:model for enlarged family}. Let $h_1$ be the conjugacy between $T_\alpha$ and  $T_{\wring(\alpha)}$ used to define the wringing map.  Let $\hat h_1:=\hat \riem_{\wring(\alpha)}^{-1}\circ h_1  \circ\hat\riem_{\alpha}$. Then $\hat h_1(\hat p_j)=\hat p_{j-1}$, for all $j \in \Z$.
\end{lem}

Note that $\hat h_1$ extends to $\R\cup\{\infty\}$ because it is a quasiconformal map of the upper half plane, hence in particular it extends to $\hat p_j$.

\begin{proof}

For $t \in [0,1]$, let $\rho_t:=\frac{1}{8} e^{2i\pi t}$ and 	
let $h_t$ be the curve of hybrid conjugacies  constructed in Proposition \ref{prop:wringing} (c), 
such that $h_t \circ T_\alpha \circ h_t^{-1} =  T_{A(\alpha, t),\rho_t }$.	
From the explicit expression of $\hat{g}_{\rho_t}$, its poles are exactly the solutions of the equation 
\begin{equation*}
	\tan(cz) = \tanh(c) \cot(\pi t).
\end{equation*}

Let $j \in \Z$, let 
$\hat h_t:=\hat \riem_{A(\alpha,\rho_t),\rho_t}^{-1}\circ h_t  \circ\hat\riem_{\alpha,\frac{1}{8}}$
, and let $\gamma(t):=\hat h_t(\hat p_j)$. 
By construction, for all $t \in [0,1]$, $\gamma(t)$ is a pole of $\hat{g}_{\rho_t}$, so it satisfies 
\begin{equation}
	\tan(c \gamma(t)) = \tanh(c) \cot(\pi t).
\end{equation}
 Notice that $\tanh(c) \cot(\pi t)$ is a real decreasing function going from $\infty$ to $-\infty$. 
Since $z \mapsto \tan(c z )$ is a universal cover from   $ \rs $ to $\rs \setminus \{-i,i \}$ and since $t \mapsto  \tanh(c) \cot(\pi t), t \in [0,1]$ is a simple closed curve  in $\rs$ with $\ind(\gamma, i)=-1$, 
we have $\gamma(1) = \gamma(0) - \frac{\pi}{c}$, or in other words, $\hat h_1 (\hat{p}_j) = \hat{p}_{j-1}$.
\end{proof}

\begin{proof}[Proof of Proposition~\ref{prop:WinH}]
Let $g, \riem_\alpha, \riem_{\wring(\alpha)}$ be as in Lemma~\ref{lem:confconjH} and $\hat g _{1/8}, \hat \riem_\alpha, \hat \riem_{\wring(\alpha)}$  be as in  Lemma~\ref{lem:model for enlarged family}.
  Let $M:\H\ra\H$ be the unique Moebius automorphism of $\H$ fixing infinity  that sends $i$ to the attracting fixed point $p$ for $g$ (which was independent of $\alpha$ by Lemma \ref{lem:confconjH}). Since $p\in i\R$,  and Moebius maps send geodesics to geodesics, $M(i\R)=i\R$ hence $M(0)=0$.

Then $\riem_\alpha\circ M$  fixes infinity and maps $i$ to $0$, hence $\riem_\alpha\circ M=\hat\riem_\alpha$ by  uniqueness of the Riemann maps. Moreover  $M\circ \hat g_{1/8}=g\circ M$.  Since conjugacies  map  poles  to poles  bijectively preserving the order, the poles of $g$ are $\{\frac{2j+1}{2}\}_{j\in\Z}$, and  $M(0)=0$, the claim  that  $H(\frac{2j+1}{2}) = \frac{2j-1}{2}$ follows from Lemma  \ref{lem:shifting poles generalized model}. 

 The fact that $h_1(p_j(\alpha))=p_{j-1}(\wring(\alpha))$ for every $j\in\Z$ follows from the definition of the poles in \eqref{eq:labelling} and  the fact that conjugacies map poles to poles  bijectively preserving the order.
\end{proof}
	 Recall the numbering of poles and the  definition of marked pole from Equation \eqref{eq:labelling} in Section \ref{sec:confconj}.
	 
\begin{proof}[Proof of Theorem~\ref{th:unicity}]
 By construction   $T_\alpha$ and $T_{\wring^k(\alpha)}$ belong to the same hybrid class for any $k$. For the converse, suppose that there is a conjugacy  $\conj$ between two maps $T_\alpha,T_{\hat\alpha}$ with  $\alpha,\hat\alpha\in\TT$.  Since conjugacies map poles to poles there exists $k \in \Z$ such that $\conj(p_0(\alpha)) = p_{-k}(\hat\alpha)$. Let  $h_k$ be the  hybrid conjugacy between $T_{\alpha}$ and  $T_{\wring^k(\alpha)}$. By Proposition~\ref{prop:wringing}, 
	$$
h_k=h_{\alpha,k}=h_{\wring^{k-1}(\alpha),1}\circ\ldots\circ h_{\alpha,1}		
	$$
	hence by applying repeatedly Proposition~\ref{prop:WinH} we have that 	$h_k(p_0(\alpha)) = p_{-k}(\wring^k(\alpha))$.
	 Therefore, $T_{\hat\alpha}$ and $T_{\wring^k(\alpha)}$ are hybrid conjugate, via the map $ h_k \circ \conj^{-1}$, which maps  $p_{-k}(\hat\alpha)$ to $p_{-k}(\wring^k(\alpha))$ and hence preserves the index of all poles. In particular it  maps  the marked pole of $T_{\hat\alpha}$ to the marked pole of $T_{\wring^k(\alpha)}$, which implies, by Proposition \ref{prop:crucial}, that $h_k \circ \conj^{-1}=\id$ and $\hat\alpha=\wring^k(\alpha)$.
\end{proof}

\subsubsection{A technical lemma about extension of conjugacies}\

In this section we prove  a technical lemma, Lemma  \ref{lem:extendh}, which shows that a conjugacy between $T_\alpha$ and itself in any fundamental annulus extends uniquely to a conjugacy up to the filled Julia set.
 It will be used in Section \ref{sec:ff} (Proposition \ref{prop:keytool}) when dealing with the parameter version of the Straightening Theorem.

This result is  similar to a corresponding lemma for quadratic-like maps with connected Julia sets. However, our situation is more technical, due to the fact that if $T_\alpha:  U \to V$ is TL with connected Julia set, then pullbacks of the fundamental annulus $\overline{V} \setminus U$ do not remain topological annuli,  but annuli that are "pinched" at prepoles. See Figure \ref{connected case} (left).  {Nevertheless the closure of these sets are still connected and contractible to a circle.}

\begin{lem}[Pinched annuli]\label{lem:topo}
	Let $D, D' \subset \C$ be Jordan domains such that $D' \subset D$. Then  $X:=\overline{D} \setminus D'$ is connected, and $\pi_1(X) \simeq \Z$.
\end{lem}

Observe that we do not assume that $D' \Subset D$: there can be many intersections between the boundaries of $D$ and $D'$.

\begin{proof}
	We will prove that $X$ is a deformation retract of a closed annulus, which implies both its connectedness and the fact that $\pi_1(X) \simeq \Z$.
	Let $\riem: \D \to D'$ be a Riemann map. 
	Since by assumption $D'$ is a Jordan domain, $\riem$ extends to a homeomorphism from $\overline{\D}$ to $\overline{D'}$. 
	For $0<t \leq 1$, let $A_t:=\overline{D} \setminus \riem( \D(0,t) )$. For any $0<t<1$, $D_t:=\riem( \D(0,t) )$ is a Jordan domain compactly contained in $D$, therefore $A_t$ is a closed annulus; and of course $X \subset A_t$. 

		For $0\leq t \leq 1$, we define $F_t: A_{1/2} \to X$ by 
	$$F_t(x) = \left\{
	\begin{array}{ll}
		x & \text{ if $x \in X$} \\
		\riem\left((1-t)\riem^{-1}(x)+t\frac{\riem^{-1}(x)}{|\riem^{-1}(x)|}  \right)  & \text{otherwise}.
	\end{array}	
	\right. $$
	
	Then:
	\begin{enumerate}
		\item $F_t(x)=x$ for all $x \in X$, by definition
		\item For all $0 \leq t \leq 1$, $F_t(A_{1/2}) \subset A_{1/2}$
		\item For all $x \in A_{1/2}$, $F_0(x) =x$
		\item For all $x \in A_{1/2}$, $F_1(x) \in \partial D' \subset X$.
	\end{enumerate}
	Therefore $X$ is indeed a deformation retract of $A_{1/2}$.
\end{proof}

\begin{lem}[Extension of conjugacies on fundamental annuli]\label{lem:extendh}
	Let $\alpha \in \mathcal{T}$ be such that $T_\alpha$ has no virtual cycle. Let $(U_1,  V_1)$ and $(U_2, V_2)$ be two pairs of Jordan domains such that
	$(T_\alpha, U_i, V_i, \infty, \frac{1}{\alpha})$ are two TL restrictions of $T_\alpha$.
	Let $h: \overline{V_1} \setminus U_1 \to \overline{V_2} \setminus U_2$ be a homeomorphism such that
	for all $z \in \partial U_1 \setminus \{\infty\}$, 
	$$h \circ T_\alpha(z) = T_\alpha \circ h(z).$$
	Then $h$ extends uniquely to a  homeomorphism $\tilde h : \overline{V_1} \setminus K(T_\alpha) \to \overline{V_2} \setminus K(T_\alpha)$ 
	such that for all $z \in U_1 \setminus K(T_\alpha)$, 
	$$\tilde h \circ T_\alpha(z) = T_\alpha \circ \tilde h(z).$$
\end{lem}
\begin{proof}
	We start by introducing some notations: for $n \in \N$ and $i \in \{1,2\}$, let 
	$C_n^i:=\overline{T_\alpha^{-n}(\partial V_i)}$, $D_n^i:=T_\alpha^{-n}(V_i)$ and $X_n^i:=\overline{D_n^i} \setminus D_{n+1}^i$.
	We claim that the sets $D_n^i$ are Jordan domains  (in $\chat$). We will argue inductively on $n \in \N$. By definition of a TL map, $D_0^i$ is a Jordan domain.
	Let now $n \in \N$ be such that $D_n^i$ is a Jordan domain.
	Then, since $\frac{1}{\alpha} \in D_n^i$,  (here we are using that $T_\alpha$ has no virtual cycle), the set $D_{n+1}^i:=T_\alpha^{-1}(D_n^i)$  is a logarithmic tract above $D_n^i$,  and  $\overline{D_n^i}\setminus\{\frac{1}{\alpha}\}$ has no asymptotic values. Hence, by Lemma \ref{lem:jordantracts}, $D_{n+1}^i$ is  a Jordan domain.	Moreover, the sets $C_n^i$ are the boundary of $D_n^i$, so they are Jordan curves.	We can now apply Lemma \ref{lem:topo} to obtain that the sets $X_n^i$ are connected, and $\pi_1(X_n^i) \simeq \Z$. 
	
	Next, we construct a sequence of homeomorphisms $h_n: X_n^1 \to X_n^2$ such that $h_0=h$ and for all $n \geq 1$, 
	\begin{enumerate}
		\item $T_\alpha \circ h_{n}=h_{n-1} \circ T_{\alpha}$ on $X_n^1 \setminus \{\infty\}$  
		\item $h_n = h_{n-1}$ on $X_n^1 \cap X_{n-1}^1$. 
	\end{enumerate}
	We first treat the case $n=1$. The sets $X_0^i$ are closed annuli which do not contain any singular value, so $T_\alpha: T_\alpha^{-1}( X_0^i) \to X_0^i$ is a covering map. It has infinite degree, so by the classification of coverings of annuli, it must be a universal cover.  
	  Then both maps $T_\alpha: T_\alpha^{-1}(X_0^2) \to X_0^2$ 
	and $h_0 \circ T_\alpha : T_\alpha^{-1}(X_0^1) \to X_0^2$ are universal covers. Let $z_0 \in C_1^1 \setminus \{\infty\}$, hence $z_0 \in X_0^1 \cap X_1^1$.  By the unicity of universal covers and the unicity of the equivalence between universal covers, there is a unique homeomorphism $h_1: T_\alpha^{-1}(X_0^1) \to T_\alpha^{-1}(X_0^2)$ such that 
	$h_1(z_0)=h_0(z_0)$ and $T_\alpha \circ h_1 (z)=h_0\circ T_\alpha(z)$
for all $z \in  T_\alpha^{-1}(X_0^1)$.	
	We can extend continuously $h_1$ to $X_1^1=T_\alpha^{-1}(X_0^1) \cup \{\infty\}$ by  setting $h_1(\infty):=\infty$.
	This proves (1).

	It remains to prove that $h_1$ coincides with $h_0$ on $X_1^1 \cap X_0^1=C_1^1$.
	This follows from the same unicity argument: both $h_0  \circ T_\alpha: C_1^1 \setminus \{\infty\} \to C_0^2$ and 
	$T_\alpha \circ h_1  : C_1^1  \setminus \{\infty\} \to C_0^2$ are universal covers, and they coincide at $z_0 \in C_1^1$ by construction; so they coincide on all of  $C_1^1 \setminus \{\infty\}$. And since $h_1(\infty)=h_0(\infty)=\infty$, $h_1$ and $h_0$ coincide on all of $X_1^1 \cap X_0^1$. 	
	This proves (2).

	\medskip

	We now treat the general case, which is exactly the same. Let $n \geq 1$ be such that $h_{n}$ is constructed.
	As we have proved, the sets $X_n^i$ are connected and $\pi_1(X_{n}^i) \simeq \Z$, and moreover $X_n^i$ does not contain any asymptotic value. Therefore, $T_\alpha: T_\alpha^{-1}( X_n^i) \to X_n^i$ is a covering map. It has infinite degree, so by the classification of coverings, it must be a universal cover. Then both maps $T_\alpha: T_\alpha^{-1}(X_n^2) \to X_n^2$ 
	and $h_n \circ T_\alpha : T_\alpha^{-1}(X_n^1) \to X_n^2$ are universal covers. Let $z_n \in C_{n+1}^1 \setminus \{\infty\}$: then $z_n \in X_n^1 \cap X_{n+1}^1$.  By the unicity of universal covers and the unicity of the equivalence between universal covers, there is a unique homeomorphism $h_{n+1}: T_\alpha^{-1}(X_n^1) \to T_\alpha^{-1}(X_n^2)$ such that 
	$h_{n+1}(z_n)=h_n(z_n)$ and $ T_{\alpha} \circ h_{n+1} (z)=h_{n} \circ T_\alpha (z)$ for all $z \in  T_\alpha^{-1}(X_n^1)$.	
	We can extend continuously $h_{n+1}$ to $X_{n+1}^1=T_\alpha^{-1}(X_0^1) \cup \{\infty\}$ by setting $h_{n+1}(\infty):=\infty$.
	This proves (1).
	
	To show (2) we need to check that $h_{n+1}$ coincides with $h_n$ on $X_{n+1}^1 \cap X_n^1=C_{n+1}^1$.
	We apply  the same unicity argument. Both $h_n\circ T_\alpha: C_{n+1}^1 \setminus \{\infty\} \to C_n^2$ and 
	$ T_\alpha \circ h_{n+1} : C_{n+1}^1  \setminus \{\infty\} \to C_n^2$ are universal covers, and they coincide at $z_n \in C_{n+1}^1$ by construction; so they coincide on all of  $C_{n+1}^1 \setminus \{\infty\}$. And since $h_{n+1}(\infty)=h_n(\infty)=\infty$, $h_{n+1}$ and $h_n$ coincide on all of $X_{n+1}^1 \cap X_n^1$. 	

	\medskip
		
	Finally, we set $\tilde h(z):=h_n(z)$ if $z \in X_n^1$. By condition (2), the map $\tilde h$ is well-defined and continuous on $\bigcup_{n \geq 0} X_n^1 = V_1 \setminus K(T_\alpha)$. By construction, it is a homeomorphism 
	$\tilde h : V_1 \setminus K(T_\alpha) \to V_2 \setminus K(T_\alpha)$, and it satisfies 
		$$
		\tilde h \circ T_\alpha = T_\alpha \circ \tilde h \text{ on $ U_1 \setminus K(T_\alpha)$}.
		$$
\end{proof}

\section{The straightening theorem}\label{sec:straight}

The goal in this section is to prove the Straightening Theorem for tangent-like maps (Theorem A). 

\begin{theo}[Straightening theorem]\label{th:straight}
	Let $(f,U,V,u,v)$ be a tangent-like map. Then 
 there exists $\alpha \in \C \setminus \{1\}$ such that $f$ is hybrid conjugate to $T_\alpha$. If $K(f)$ is connected, then $\alpha\in\TT$ is unique, up to the action of the wringing map.\end{theo}

 In the following proof, we  take as model map the map $T_{\alpha_{\sharp}}$, where $\alpha_{\sharp}$ is the unique parameter for which $T_{\alpha_{\sharp}}$ is conjugate to $g: z \mapsto C\tan(\pi z)$ (see Remark \ref{rem:specialmap} and Figure \ref{switch} left). (We could in fact have chosen any $T_\alpha$, for $\alpha \in \tcal$). This map will play the role that is played by the external map $z \mapsto z^d$ in the classical Straightening Theorem by Douady and Hubbard. 

\begin{proof}
	Let $(f,U,V,u,v)$ be a tangent-like map.   Up to conjugating with a conformal map (which is  a hybrid conjugacy)  we can assume that $V$ is a Euclidean disk containing infinity and  that $u=\infty$.	
Indeed, since $V$ is a simply connected domain with analytic boundary, we can map it conformally to a Euclidean disk containing infinity in such a way that $u$ is mapped to infinity; the aforementioned map extends analytically to the boundary and the image of $U$ is a quasicircle. 

	 By Lemma~\ref{lem:tlmodel}, 
	$(T_{\alpha_{\sharp}},U_\sharp,V_\sharp,u_\sharp,v_\sharp)$ is a tangent-like map for  $V_\sharp=\rs \setminus \overline{\D(0,2)}$, $U_\sharp=T_{\alpha_{\sharp}}^{-1}(V_\sharp)$,  $u_\sharp=\infty$ and $v_\sharp=1/\alpha_{\sharp}$. 
	Let $t_1$ be a Möbius map which maps $V$ to $V_\sharp$ sending $v$ to $\infty$. 
	Note that $t_1$ is globally defined and maps $\rs\setminus \overline{V}$ to $\rs \setminus \overline{V_\sharp}$.

 	 Since $f:  \overline{U} \setminus \{\infty\} \to \overline{V} \setminus \{v\}$ and 
	$T_{\alpha_{\sharp}} :  \overline{U_\sharp} \setminus \{\infty\} \to \overline{V_\sharp} \setminus \{v_\sharp\}$ are universal covers, the homeomorphism  $t_1 :   \overline{V} \setminus \{v\} \to  \overline{V_\sharp} \setminus \{v_\sharp\}$  lifts as a  homeomorphism $t_2:  \overline{U} \setminus \{\infty\} \to  \overline{U_\sharp} \setminus \{\infty\}$ 
	such that 
	$$
	T_{\alpha_{\sharp}} \circ t_2 = t_1 \circ f.
	$$
	The lift  $t_2:\partial U \setminus \{\infty\} \to  \partial {U_\sharp} \setminus \{\infty\}$ is quasisymmetric since it is the boundary extension of a conformal map and $\partial U$ is a quasicircle (see Lemma \ref{lem:quasidisk}).

	By the interpolation lemma \ref{lem:interpolation}, there exists a quasiconformal extension $t:\overline{V} \setminus U\ra \overline V_\sharp\setminus U_\sharp$, interpolating between   ${t_1}_{|\partial V}$ and ${t_2}_{|\partial U}$. Extending as $t=t_1$ on the complement of $V$,	we now have a quasiconformal map $t:\rs\setminus U\ra\rs\setminus U_\sharp$, which is  a Moebius map on $\rs\setminus V$, and such that $T_{\alpha_{\sharp}}\circ t=t\circ f$ on $\partial U$.

	We shall use $t$ to glue $T_{\alpha_{\sharp}}$ on the complement of $U$. To that end, define the quasiregular map $F:\rs\ra\rs$ as 
	\begin{align*}
		F(z) = \left\{
		\begin{array}{ll}
			f(z) &\text{ if } z \in U \\
			t^{-1} \circ T_{\alpha_{\sharp}} \circ t(z) & \text{ if } z \notin U,
		\end{array}
		\right.
	\end{align*}
Since  $F$ is conformally conjugate to $T_{\alpha_{\sharp}}$  via $t$ on  $\rs\setminus \ov V$,    $z_*=t^{-1}(0)$ is an attracting fixed point for $F$ of multiplier $1/8$.

The rest of the proof is the same as in the case of the classical straightening theorem. 
Let $\mu_0\equiv 0$ denote the trivial Beltrami form and $A:=V\setminus \overline{U}$, the only part of the plane where $F$ is not holomorphic but quasiregular. We define an $F$-invariant Beltrami form as follows. Let $\mu:=t^*\mu_0$ on $\rs \setminus U$ (hence $\mu=\mu_0$ on $\rs \setminus V$).  Extend $\mu$ to $U\setminus K(f)$ by pull-backs under the map $F$ i.e., abusing notation, $\mu=(F^n)^*\mu$ on $F^{-n} (A)$ and, finally, $\mu=0$ on $K(f)$. Then $\mu$ has bounded dilatation and is $F-$invariant.  

Let $\conj:\rs\to\rs$ be a quasiconformal map integrating $\mu$ (i.e. satisfying $\conj^*\mu_0=\mu$) provided by the Measurable Riemann Mapping Theorem, normalized so that $\conj(\infty)=\infty$, $\conj(z_*)=0$ and $\conj(t^{-1}(1))=1$.  Then $G:=\conj\circ F\circ \conj^{-1}$ is meromorphic transcendental, and quasiconformally conjugate to $F$ under $\conj$ which, since $\mu=0$ on $K(f)$,  is actually a hybrid conjugacy. Moreover, $G$ has exactly two singular values ($1$ and $\conj(v)$),  both of which are omitted, and one fixed point 0 with multiplier $1/8$, since $\conj$ is conformal outside $V$. By Lemma \ref{lem:rigidity}, $G$ is affine conjugate to $T_\alpha$ for some  $\alpha$.  In fact, by our choice of normalization, $G=T_\alpha$ with $\alpha=1/\conj(v)$.
	
Finally, assume that $K(f)$ is connected so that $\alpha\in\TT$, and suppose that, by a different choice of interpolation, $f$ is hybrid conjugate to some other $T_{\alpha'}$. Then, $T_\alpha$ and $T_{\alpha'}$ are hybrid conjugate and, by Theorem \ref{th:unicity},  $\alpha'=\wring^k(\alpha)$ for some $k\in\Z$.
\end{proof}

We record here the following fact, which will play an important role later on in the proof of Proposition \ref{prop:chiinj}, when dealing with the parameter version of the Straightening Theorem. It is more a by-product of the proof of Theorem \ref{th:straight} than a consequence of its statement.

\begin{coro}[Towards B\"ottcher coordinates]\label{cor:bottcherlike}
	With the same notations as in the proof of Theorem \ref{th:straight}, the map $\Psi_\alpha:=t \circ \conj^{-1}$ is holomorphic on $\rs \setminus \overline{\conj(U)}$ and conjugates $T_\alpha$ to $T_{\alpha_{\sharp}}$.
	Moreover, $\Psi_\alpha$ depends only on $\alpha$, in the sense that if $(\tilde f, \tilde U, \tilde V, \tilde u, \tilde v)$ is another tangent-like map hybrid conjugate to $T_\alpha$ and $\tilde \Psi_\alpha:=\tilde t \circ \tilde \conj^{-1}$, then 
	$\tilde \conj (\tilde U) = \conj(U)$ and $\tilde \Psi_\alpha = \Psi_\alpha$ on $\rs \setminus \overline{\conj(U)}$.
\end{coro}

\begin{proof}
	Using the notation of the proof of the Straightening Theorem, define
	$\Psi_\alpha := t \circ \conj^{-1} : \rs \setminus \overline{\conj(U)} \to \rs \setminus  \overline{U_\sharp}$. 
	By construction, the maps $\conj$ and $t$ have the same Beltrami coefficient on the annulus $\overline{V} \setminus U$; moreover, $\conj$ and $t$ are holomorphic on $\rs \setminus \overline{V}$. Therefore $\Psi$ is holomorphic on $\rs \setminus \overline{\conj(U)}$. 
	The fact that $\Psi_\alpha$ conjugates $T_\alpha$ to $T_{\alpha_{\sharp}}$ is also by construction.
	
	It remains to prove the uniqueness. Let $\Psi:=(\tilde \Psi_\alpha)^{-1} \circ \Psi_\alpha: \rs \setminus \overline{\conj(U)} \to \rs \setminus \overline{\tilde \conj(\tilde U)}$. Then $\Psi$ is biholomorphic, and 
	commutes with $T_\alpha$. That is, for all $z \in  \rs \setminus \overline{\conj(U)}$, 
	$$\Psi \circ T_\alpha(z) = T_\alpha \circ \Psi(z).$$
	Since $T_\alpha: \rs \setminus \overline{\conj(U)} \to \left(\rs \setminus \overline{\conj(V)} \right) \setminus \{1\}$ is a universal cover, and similarly $T_\alpha: \rs \setminus \overline{\tilde \conj(\tilde U)} \to \left(\rs \setminus \overline{\tilde \conj(\tilde V)} \right)\setminus \{1\}$ is also a universal cover, it follows that $\Psi(1)=1$.
	Then, for all $n \in \N$, $\Psi \circ T_\alpha^n(1)=T_\alpha^n(1)$, and so by the identity principle, 
	$\Psi=\id$ and $\conj(U) = \tilde \conj(\tilde U)$.
	
\end{proof}

\section{Tangent-like Families}\label{sec:tlfam}

 In this section we consider holomorphic families of tangent-like maps. Under appropriate  conditions, each member of the family is hybrid conjugate to a member of the model family $T_\alpha$, by the Straightening Theorem (Theorem A). This defines a straightening map $\chi$ between the connectedness locus of the TL family and the Tandelbrot set. Our goal in this section is to prove that this map is actually a homeomorphism, proving Theorem E. 

The section is divided in three parts. Subsection \ref{defichi} introduces TL families and fixes the conditions under which we can define $\chi$, apart from discussing the family of hybrid conjugacies that arise from the Straightening Theorem. In subsection \ref{sec:famjstab}, we discuss $J-$stability for TL families and show that the boundary of the connected locus is the bifurcation locus. Finally, we use all our previous work to show, in subsection \ref{sec:chihomeo}, that the straightening map is a homeomorphism and prove Theorem C.

\subsection{TL families and definition of the straightening map.} \label{defichi} 
\ 

The following definition mimics the definition in \cite[p. 545]{LyubichBook}.

\begin{defi}[Tangent-like families]\label{def:tlfamily}
	Let $\Lambda\subset\C$ be a domain. A holomorphic family of TL maps is a  family of TL maps of the form 
	$f_\la : U_\la \to V_\la$, $\la \in \Lambda$, such that:
	\begin{enumerate}
		\item $\mathcal{U}:=\{(\la,z) : z \in U_\la\}$ and $\mathcal{V}:=\{(\la,z) : z \in V_\la\}$ are domains in $\C^2$
		\item $\mathbf{f}: \mathcal{U} \to \rs$ defined by $\mathbf{f}(\la,z)=f_\la(z)$ is holomorphic.	
	\end{enumerate}
\end{defi}

As in the case of holomorphic families of quadratic-like maps, we will make additional assumptions. We require, for a start, that the fundamental annulus moves holomorphically with the parameter.

\begin{defi}[Equipped TL family]\label{def:equipped}
	A holomorphic family of TL maps is \emph{equipped} if there exists $\la_0 \in \Lambda$ and a holomorphic motion ${\holom}_\la: \overline{V_{\la_0}} \setminus U_{\la_0} \to \overline{V_\la} \setminus U_\la$ such that for all $\la \in \Lambda$, for all $z \in \partial U_{\la_0}\setminus \{u_{\la_0}\}$, 
	$${\holom}_\la \circ f_{\la_0}(z) =  f_\la \circ {\holom}_\la(z).$$ 
	The holomorphic motion ${\holom}_\la$ will be called the {\em equivariant} holomorphic motion of the {\em fundamental annulus}  (where equivariance refers to the compatibility with the dynamics).
\end{defi}

\begin{defi}[Proper TL family]\label{def:proper}
	A holomorphic family of TL maps is {\em proper} if:
	\begin{enumerate}
		\item $\Lambda$ is a Jordan domain, and $\{f_\la\}_{\la \in \Lambda}$ is the restriction  of a slightly larger TL family $\{f_\la\}_{\la \in \tilde \Lambda}$, where $\Lambda \Subset \tilde \Lambda$.
		\item For all $\la \in \partial \Lambda$,  $v_\la \in \partial V_\la$.
		\item The winding number of $v_\la - u_\la$  around $0$  is 1, as $\la$ turns once along $\partial \Lambda$.
	\end{enumerate}
\end{defi}

From now on, we will always work with a proper equipped TL family.

\begin{defi}[Tandelbrot set of a TL family]
	The  {\em Tandelbrot set  of a TL family} is its connectedness locus
	\[
	\tcal(\mathbf{f}):=\{ \la \in \Lambda: v_\la \in K(f_\la)  \} { =
	\{ \la \in \Lambda: K(f_\la) \text{\ is connected} \}}
	\]
  where $K(f_\la)$ denotes the filled Julia set of $f_\la$ as defined in Definiton \ref{def:filledJset}.
\end{defi}

Similarly to the case of families of quadratic-like map, a \emph{tubing} will be a map $t: \overline{\mathcal{V}} \setminus  \mathcal{U} \to \Lambda \times A$ of the form $t(\la,z)=(\la, t_\la(z))$, where $A$ is some fixed  annulus. We will not require a general definition of tubing, and instead we will work only with  so-called standard tubings defined below. See Figure \ref{fig:tubing}.

\begin{defi}[Standard choice of tubing]\label{def:tubing}
	Let $\{f_\la\}_{\la \in \Lambda}$ be a proper equipped TL family: since it is equipped, there is a holomorphic motion ${\holom}_\la$ of the fundamental annulus $\overline{V_{\la}} \setminus U_\la$, with basepoint $\la_0 \in \Lambda$. Let $\alpha_{\sharp}$ be as in Remark \ref{rem:specialmap}, and  $U_\sharp,V_\sharp \subset \rs$ be such that $(T_{\alpha_{\sharp}},U_\sharp,V_\sharp, \infty, 1/\alpha_{\sharp})$ is a TL map. Let $A_\sharp:=V_\sharp \setminus \overline{U_\sharp}$.
	We choose an equivariant quasiconformal homeomorphism  $\tau: 
	\overline{V_{\la_0}} \setminus U_{\la_0} \to \overline{A_\sharp}$. We then define the tubing map
	\[
	t(\la,z):=(\la, t_\la(z)):=(\la, \tau\circ {\holom}_\la^{-1} (z)). 
	\]
\end{defi}
\begin{figure}[hbt!]
\centering
  \includegraphics[width=0.85\textwidth]{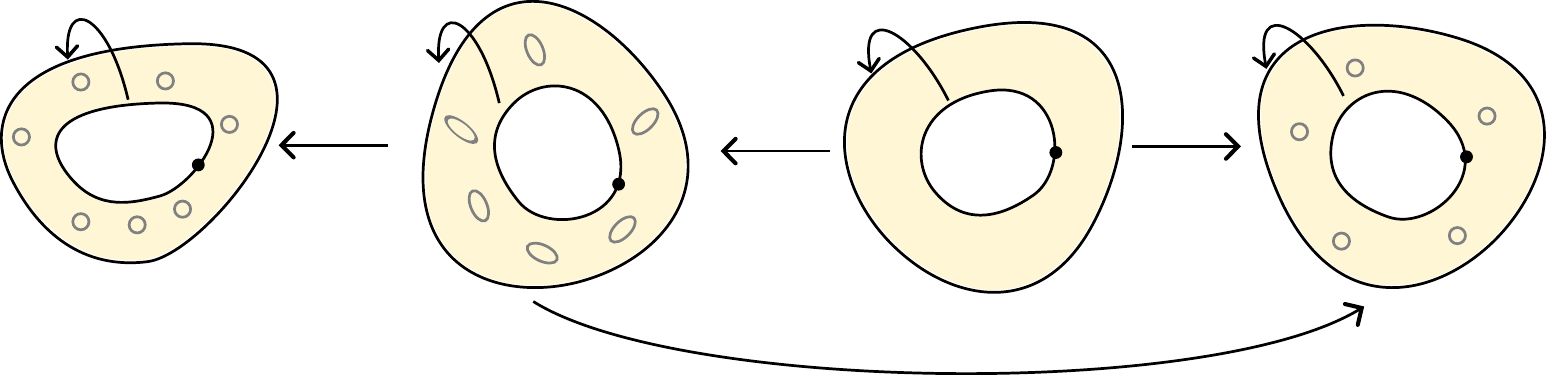}
 \setlength{\unitlength}{0.85\textwidth}
 \put(-0.4,0.015){\small $t_\la$}
  \put(-0.24,0.157){\small $\tau$}
  \put(-0.5,0.157){\small $\holom_\la$}
  \put(-0.79,0.157){\small $\conj_\la$}
  \put(-0.21,0.23){\small $T_\sharp$}
  \put(-0.045,0.139){\tiny $\infty$}
  \put(-0.1,0.19){\scriptsize $U_\sharp$}
  \put(-0.06,0.225){\scriptsize $V_\sharp$}
  \put(-0.17,0.122){\small $A_\sharp$}
  \put(-0.46,0.235){\small $T_{\la_0}$}
 \put(-0.36,0.195){\scriptsize $U_{\la_0}$}
  \put(-0.32,0.23){\scriptsize $V_{\la_0}$}
  \put(-0.314,0.139){\tiny $u_{\la_0}$}
\put(-0.63,0.193){\scriptsize $U_{\la}$}
  \put(-0.59,0.215){\scriptsize $V_{\la}$}
  \put(-0.745,0.23){\small $f_{\la}$}
  \put(-0.592,0.12){\tiny $u_{\la}$}
  \put(-0.99,0.23){\small $T_\alpha$}
  \put(-0.87,0.13){\tiny $u_{\alpha}$}
\put(-0.875,0.182){\scriptsize $U_{\alpha}$}
  \put(-0.85,0.22){\scriptsize $V_{\alpha}$}
 \caption{\small \label{fig:tubing} The tubing map $t_\lambda$ and the Straightening Theorem in the family $\{f_\la\}$.}
 \end{figure}

Observe that this tubing map may be extended to a homeomorphism $\tilde t : \overline{\mathcal{V}} \to \overline{\Lambda} \times \overline{V_\sharp}$, which maps $\mathcal U$ to $\Lambda \times U_\sharp$.
Indeed, using   Slodkowski's  extension Theorem~\ref{thm:Slodkowski} extend $\tau$ to a quasiconformal homeomorphism $\tilde \tau:\ov{ V_{\la_0}} \to \ov{V_\sharp}$ and  ${\holom}_\la$ to a holomorphic motion $\tilde h: \rs \to \rs$; then  set for all $(\la,z) \in \mathcal{V}$:
$$\tilde t(\la, z) = (\la, \tilde \tau \circ \tilde  {\holom}_\la^{-1}(z)).$$
The map $\tilde t$ is clearly continuous and injective on the compact set $\overline{\mathcal{V}}$, hence it is a homeomorphism onto its image.

\begin{propdef}[Straightening theorem in a family]\label{prop:straightening family}
	Let $\{f_\la\}_{\la \in \Lambda}$ be a proper equipped TL family. Let $t$ be the  tubing map defined above.
	For every $\la \in \Lambda$, there is a unique hybrid conjugacy $\conj_\la$ between $f_\la$ and some $T_\alpha$, $\alpha \in \D$ such that 
	\begin{equation}\label{eq:preferred conjugacy}
	\frac{\overline{\partial} \conj_\la}{\partial \conj_\la}= \frac{\overline{\partial}t_\la}{\partial t_\la}
\end{equation}
	a.e. on $V_\la \setminus U_\la$.
	We define the {\em straightening map }
	$$
	\chi: \Lambda  \to \D
	$$ by the relation 
	$$
	\conj_\la \circ f_\la = T_{\chi(\la)} \circ \conj_\la.
	$$
\end{propdef}

 \begin{proof}
	The existence and unicity of the hybrid conjugacies $\conj_\la$ follows from an inspection of the proof of the Straightening Theorem: Indeed, for every TL map $(f,U,V,u,v)$, 
	the construction of the hybrid conjugacy is completely determined by the choice of a quasiconformal homeomorphism   $\tau$ between 
	$\overline{V} \setminus U$ and $\overline{V_\sharp} \setminus U_\sharp$ such that $\tau \circ f = T_\sharp \circ \tau$ on $\partial U \setminus \{u\}$, where 
	$U_\sharp,V_\sharp$ are Jordan domains such that $(T_\sharp,U_\sharp, V_\sharp, \infty,1/\alpha_\sharp)$ is TL. In our case, this homeomorphism is the tubing map $t_\la$.
	
 \end{proof}
Naturally, the straightening map $\chi$ sends the Tandelbrot set of $\mathbf f$ to the actual Tandelbrot set. 

\begin{lem}
	For any $\la \in \Lambda$, we have:
$\la \in \tcal(\mathbf{f}) \Leftrightarrow \chi(\la) \in \tcal$.
\end{lem}

\begin{proof} 
	Indeed, recall that $\la \in \tcal(\mathbf{f})$ if and only if $K(f_\la)$ is connected (see Prop. \ref{prop:Kconnected}), 
	and that $\chi(\la) \in \tcal$ if and only if $K(T_{\chi(\la)})$ is connected. The claim is then clear, as the hybrid conjugacy gives a homeomorphism between $K(f_\la)$ and $K(T_{\chi(\la)})$.
\end{proof}

In order to prove the continuity of $\chi$, we will require  the following result.

\begin{prop}[Convergence of conjugacies]\label{prop:convergence of conjugacies}
Let $\{f_\la\}_{\la\in\Lambda}$ be an equipped  tangent-like family, $\la_n\ra\la_0\in\Lambda$ with $\la_n\in\TT({\bf f})$, $\alpha_n:=\chi(\la_n)$, $\conj_n$ be the hybrid conjugacies between  $f_{\la_n} $ and $T_{\alpha_n}$. Then, up to passing to  a subsequence, we can assume that $\alpha_{n_k}\ra\alpha_*\in\TT$ and that  $\conj_{n_k}$ converge to a quasiconformal conjugacy $\conj_*$ between $f_{\la_0}$ and $T_{\alpha_*}$.

\end{prop}

Notice that $\conj_*$ is not necessarily a hybrid conjugacy, since filled Julia sets do not always move continuously with the parameter.

\begin{proof}

Since the Tandelbrot set is compact, up to extracting a subsequence we can assume that $\alpha_n\ra\alpha_{*}$. 
We claim that the maps $\conj_n$ are uniformly quasiconformal. For $\la \in \Lambda$, recall that $\holom_\la : \overline{V_{\la_0}} \setminus U_{\la_0} \to \overline{V_{\la}} \setminus U_{\la}$ is the equivariant holomorphic motion of the fundamental annulus, coming from Definition \ref{def:equipped}.
For each $\la_n$ the  quasiconformality constant of the hybrid conjugacy $\conj_n$ depends only on the modulus of the annulus $A_{n}:=V_{\lambda_n}\setminus \overline{U_{\lambda_n}}$ and on the quasisymmetric constants of $\partial A_n$. Indeed, by Proposition \ref{prop:straightening family}, the quasiconformality constant of $\Phi_n$ depends only on that of $t_{\la_n}$, which in turn by the interpolation  Lemma~\ref{lem:interpolation} depends only on the modulus of $A_n$ and the quasisymmetric constants of  $\partial A_n$.

Since $\overline{V_{\la}} \setminus U_{\la} = \holom_{\la}(\overline{V_{\la_0}} \setminus U_{\la_0})$,  the modulus of $\overline{V_{\la}} \setminus U_{\la}$ varies continuously in $\lambda$, hence the modulus of $A_n$ is bounded uniformly in $n$. Let us justify properly that the two components of $\partial A_n$ are $\hat K$-quasicircles, where $\hat K$ can be chosen uniformly in $n$. Using Slodkowski's theorem, let us extend $\holom_\la$ to a holomorphic motion of the whole Riemann sphere, which we denote by $\tilde \holom_\la$. Let $\tilde K_\la$ be the quasiconformal constant of $\tilde \holom_\la$: then $\tilde K_\la$ depends continuously on $\la$ so we can choose $\tilde K > \sup_{n \in \N} \tilde K_{\la_n}>1$.
Then by construction, for all $n \in \N$ the components of $\partial A_n$ have quasisymmetric constants bounded by $\hat K=K \tilde K$, where $K$ is a constant depending only on $\overline{V_{\la_0}} \setminus U_{\la_0}$. 
This proves the claim that the maps $\conj_n$ are uniformly quasiconformal.

By compactness of quasiconformal maps with the same quasiconformality constant,  up to passing to a subsequence   $(\conj_n)$ converge to  a quasiconformal limit map $\conj_*$ uniformly on compact subsets of $U_{\lambda_0}$. Moreover, for each $n$ we have
$$
\conj_n\circ f_{\la_n}=T_{\alpha_n}\circ\conj_n   \quad  \text{ on $U_{\lambda_n}$}.
$$
Up to a passing to a  subsequence, the left hand side converges to $\conj_*\circ f_{\la_0}$ uniformly on compact subsets of $U_{\la_0}$ (we are using again that the annuli $A_n$ move holomorphically), while the right hand side converges to $T_{\alpha_*}\circ\conj_*$, implying that $\conj_*$ is a quasiconformal conjugacy between $f_{\la_0}$ and $T_{\alpha_*}$. 

\end{proof}

The following Proposition asserts that equipped TL families are natural families (see Section \ref{sec:natfam}).

\begin{prop}[Equipped TL families are natural]\label{prop:TL implies natural}
	Let $\{f_\la\}_{\la \in \Lambda}$ be an equipped TL family. Then it is a natural family, i.e. for any $\la_* \in \Lambda$, there exists holomorphic motions $\phi_\la: V_{\la_*} \to V_\la$ and $\psi_\la: U_{\la_*} \to U_\la$ such that 
	$$f_\la=\phi_\la \circ f_{\la_*} \circ \psi_\la^{-1}.$$
\end{prop}

\begin{proof}
	Let $\la_* \in \Lambda$ be arbitrary.
	By the assumption that $\{f_\la\}_{\la \in \Lambda}$ is equipped, there exists a holomorphic motion ${\holom}_\la: \overline{V_{\la_*}} \setminus U_{\la_*} \to \overline{V_{\la}} \setminus U_{\la}$. 
	Now consider the holomorphic motion $\phi_\la: (\partial V_{\la_*}) \cup \{v_{\la_*}\} \to (\partial V_{\la}) \cup \{v_{\la}\}$ defined by 
	$$\phi_\la(z) = \left\{
	\begin{array}{ll}
		{\holom}_\la(z) & \text{ if $z \in \partial V_{\la_*}$} \\
		v_\la & \text{ if $z=v_{\la_*}$}.
	\end{array}
	\right.$$
	It is indeed a holomorphic motion, since for all $\la \in \Lambda$, $v_\la \in V_\la$. Now extend $\phi_\la$ 
	using   Slodkowski's  extension Theorem~\ref{thm:Slodkowski}    to obtain a holomorphic motion of the whole Riemann sphere, which we still denote by $\phi_\la$. By construction, for any $\la \in \Lambda$, $\phi_\la$ maps $V_{\la_*}$ to $V_\la$
	and $v_{\la_*}$ to $v_\la$.

	Next, let us again apply   Slodkowski's  extension Theorem~\ref{thm:Slodkowski}  to ${\holom}_\la$, to obtain an extension of ${\holom}_\la$ to the whole Riemann sphere, which we denote by $ \tilde \psi_\la$. 
	
	For any $\la \in \Lambda$, $\tilde \psi_\la$ maps $U_{\la_*}$ to $U_\la$.  The rest of the proof is now similar to [\cite{ABF}, Theorem 2.6]. We include it for the convenience of the reader.

	Let $g_\la:=\phi_\la^{-1} \circ f_\la \circ \tilde \psi_\la$. For every $\la \in \Lambda$, 
	$g_\la : U_{\la_*} \to V_{\la_*} \setminus \{v_{\la_*}\}$ is a universal covering map.	
	Let $\la \in \Lambda$ and let $\gamma : [0,1] \to \Lambda$ a continuous path joining $\la_*$ to $\la$.
	By [\cite{ABF}, Lemma 2.7], there exists  a homotopy 
	$\tilde k_{t} : U_{\la_*} \to  U_{\la_*}$ such that for all $t \in [0,1]$, 
	$$g_{\gamma(t)} =  g_{\la_*} \circ \tilde k_{t}.$$
	Since $\Lambda$ is simply connected, for every $z \in U_{\la_*}$ the value $\tilde k_{1}(z)$ is independent from the choice of $\gamma$, and depends only on $\la$ and $z$: we denote it by $k_\la(z)$. This gives the existence of a continuous map $(\la,z) \mapsto k_\la(z)$ such that for all $\la \in \Lambda$,
	$$g_\la = g_{\la_*} \circ k_\la = f_{\la_*} \circ k_\la.$$
	For every $\la \in \Lambda$, the map $k_\la : U_{\la_*} \to U_{\la_*}$ is a homeomorphism.
	Let $\psi_\la:=\tilde \psi_\la \circ k_\la^{-1}$: then the previous relation may be rewritten as
	$$f_\la = \phi_\la \circ f_{\la_*} \circ \psi_\la^{-1}.$$
	
	It only remains to prove that $\psi_\la$ defines a holomorphic motion, i.e. that for every $z \in U_{\la_*}$, 
	$\la \mapsto \psi_\la(z)$ is holomorphic. This follows from the relation 
	$f_{\la} \circ \psi_\la = \phi_\la \circ f_{\la_*}$ and the Implicit Function Theorem applied 
	to the holomorphic map $(\lam,z) \mapsto f_{\la}(z) -  \phi_\la \circ f_{\la_*}(z)$.
\end{proof}

We also record here the following useful lemma, which will be used repeatedly in Section \ref{sec:chihomeo}  (Compare with the proof of Lemma~\ref{lem:tractmoveholo}).

\begin{lem}[Lifting a holomorphic motion]\label{lem:liftholmot}
	Let $\{f_\la\}_{\la \in \Lambda}$ be an equipped  holomorphic family of TL maps. Let $\la_0 \in \Lambda$, $D \subset \Lambda$ be a simply connected domain containing $\la_0$ and $X \subset \overline{V_{\la_0}}$.
	Let $h^{(0)}: D \times X \to \rs$ be a holomorphic motion with basepoint $\la_0$ such that for all $(\la,z) \in D \times f_{\la_0}^{-1}(X)$,
	$h_\la^{(0)} \circ f_{\la_0}(z) \neq v_\la$. Then there exists a unique holomorphic motion $h^{(1)}: D \times f_{\la_0}^{-1}(X) \to \rs$ such that for all $\la \in D$,
	$$f_\la \circ h_\la^{(1)} = h_\la^{(0)} \circ f_{\la_0}.$$
\end{lem}

\begin{proof}
	Since the family is equipped, by   Proposition \ref{prop:TL implies natural} we may write $f_\la = \phi_\la \circ f_{\la_0} \circ \psi_\la^{-1}$, 
	where $\phi_\la, \psi_\la$ are holomorphic motions of $\overline{V_{\la}}$ and $\overline{U_\la}$ respectively, 
	with basepoint $\la_0$.
	
	Let $z_0 \in  f_{\la_0}^{-1}(X)$. Consider the map $\la\mapsto\phi_\la^{-1} \circ h_\la^{(0)} \circ f_{\la_0}(z_0):D\ra V_{\la_0}$ (note that it is continuous in $\la$ but a priori not holomorphic) and observe that $h_\la^{(0)} \circ f_{\la_0}(z_0) \neq v_\la$ for all $\la \in D$ by assumption.  So, using that $D$ is simply connected, there exists a unique lift $\gamma_{z_0}:D\ra\C$  under $f_{\la_0}$ such  $\gamma_{z_0}(\la_0)=z_0$; in particular, $f_{\la_0}\circ\gamma_{z_0}(\la)=\phi_\la^{-1} \circ h_\la^{(0)} \circ f_{\la_0}(z_0)$.

Define $h_\la^{(1)}(z_0):=\psi_\la\circ \gamma_{z_0}(\la)$. By construction, we have
	$$
	f_\la \circ h_\la^{(1)}(z_0) = f_\la\circ\psi_\la\circ\gamma_{z_0}(\la)=\phi_\la\circ f_{\la_0}\circ\gamma_{z_0}(\la)=h_\la^{(0)} \circ f_{\la_0}(z_0)
	$$
	and $h_{\la_0}^{(1)}(z_0)=z_0$.
	
	By the Implicit Function Theorem, $\la \mapsto h_\la^{(1)}(z_0)$ is holomorphic on $D$.
	By the uniqueness of the lifting property, the map $z \mapsto h_\la^{(1)}(z)$ is injective for all $\la \in D$.
	Thus, $h^{(1)}: D \times f_{\la_0}^{-1}(X) \to \rs$ is indeed a holomorphic motion.
\end{proof}

\subsection{$J$-stability in TL families} \label{sec:famjstab}
\ 

In this section we define activity and passivity for the singular value $v_\la$ of a tangent-like family, and we show that the activity locus of $v_\la$ is $\partial \TT(\bf{f})$.  We will use several results from \cite{ABF} which are stated for meromorphic maps but are of local nature, hence they apply to tangent-like maps.  Compare with Section \ref{sec:natfam}, where these concepts were discussed for general meromorphic natural families.

\begin{defi}[Activity and passivity of $v_\la$]\label{def:passive}
	Let $\{f_\la\}_{\la \in \Lambda}$ be a proper equipped TL family, with asymptotic value $v_\la$. We say that $v_\la$ is passive at $\la_0 \in \Lambda$ 
	if one of the two following cases occur:
	\begin{enumerate}
		\item there exists an open neighborhood $W \subset \Lambda$ of $\la_0$ such that for all $n \geq 0$, for all $\la \in W$, $f_\la^n(v_\la) \in U_\la$ and $\{ \la \mapsto f_\la^n(v_\la) : n \in \N  \}$ is a normal family;
		\item or there exists $n_0 \in \N$ such that $f_{\la_0}^{ n_0}(v_{\la_0}) \in \overline{V_{\la_0}} \setminus \overline{U_{\la_0}}$.
	\end{enumerate}
	We say that it is active at $\la_0$ if it is not passive.  The activity locus is the set of $\la \in \Lambda$ 
	such that $v_\la$ is active at $\la$.
\end{defi}
 Observe that if (2) is satisfied, the same holds for all $\la$ in a neighborhood of $\la_0$, with the same $n_0$.

\begin{rem}[Virtual cycles imply activity]\label{rem:vsdense}
	 In the current setting, with only one free singular value, $f_\la$ has a virtual cycle if there exists $n \in \N$ such that $f_\la^n(v_\la) = \infty$.
	An important remark is that in a proper TL family, virtual cycles cannot be persistent.  Indeed, otherwise we would have that $f_\lambda^n(v_\lambda)\equiv\infty$ in some open set $D$ which, by the Identity Principle, implies that $f_\lambda^n(v_\lambda)\equiv\infty$ on $\Lambda$, which contradicts the properness assumption. 
	It then follows from the definition that the existence of a virtual cycle implies activity of $v_\la$.
	Moreover, using Montel's theorem  and the definition of activity,  one can see that parameters with virtual cycles are dense in the activity locus (c.f. Proposition \ref{prop:density_of_virtual cycles} or Proposition 5.5 in \cite{ABF}, observing that tangent-like maps always have non-omitted poles provided $u\neq v$.
\end{rem}

\begin{prop}[$J-$stability and activity locus]\label{prop:msstl}
	Let $\{f_\la\}_{\la \in \Lambda}$ be a proper and equipped TL family. Then $\{f_\la\}$ is $J$-stable near $\la_0 \in \Lambda$ if and only if $v_{\la}$ is passive at $\la_0$. Moreover, the activity locus of $v_\la$ is exactly $\partial \tcal(\mathbf f)$.
\end{prop}

The proof uses the same type of arguments as  the proof of Theorem E in \cite{ABF}, but it is simpler, since we will not have to deal with persistent virtual cycles or critical points.

\begin{proof}
	Let us first prove that the family $\{f_\la\}$ is $J$-stable near any parameter $\la_0$ in the passivity locus.
	Let $W \subset \Lambda$ be any simply connected neighborhood of $\la_0$ contained in the passivity locus.
	For any $n \in \N$, let 
	$$Z_n:=\{(\la,z), \la \in W \text{ and } z \in U_\la : f_\la^n(z)=z  \},$$ 
	and let $\pi : Z_n \to W$ denote the projection on the first coordinate.
	The set $Z_n$ is an analytic subset of the domain $\mathcal{U} \subset \C^2$  as defined in  Definition~\ref{def:tlfamily}. Since $W$ is in the passivity locus, there cannot be nonpersistent   parabolic parameters in $W$, so by the Implicit Function Theorem, $Z_n$ is a smooth complex 1-manifold and $\pi: Z_n \to W$  has no critical values.  
Also, by   \cite[Lemma 3.4]{ABF}, a parameter $\lambda_*\in\Lambda$ is an asymptotic value for  $\pi$ only if the map $f_{\la_*}$ has a  virtual cycle of period at most $n$. Such virtual cycle cannot be persistent by Remark  \ref{rem:vsdense}, and cannot be non-persistent because $W$ is in the passivity locus. So $\pi$ does not have asymptotic values either.

	This proves that for   any connected component $Z$ of $Z_n$, $\pi: Z \to W$ is a covering map; but since $W$ is simply connected, it is a conformal isomorphism; in other words, $Z$ is a holomorphic graph above $W$. Since this is true for any connected component of $Z_n$, we have proved that all periodic points move holomorphically over $W$. By the $\la$-lemma and the density of repelling periodic points in $J(f_\la)$ (which follows easily from the Straightening Theorem), 
	the family $\{f_\la\}_{\la \in W}$ is therefore $J$-stable.
	\medskip
	
	Let us now prove that if $W$ is a domain which intersects the activity locus, then $\{f_\la\}_{\la \in W}$ is not $J$-stable. By density of parameters with virtual cycles (see  Remark \ref{rem:vsdense}), $W$ contains a parameter $\la_0$ with a virtual cycle:
	$f_{\la_0}^n(v_{\la_0}) = \infty$,  for some $n\geq 0$. Assume for a contradiction that $\{f_\la\}_{\la \in W}$ is $J$-stable.  Since each $f_\la$ has exactly one asymptotic value, which is omitted, any holomorphic motion of the Julia set which respects the dynamics has to map the asymptotic value $v_0$ to the asymptotic value $v_\la$. Hence
 for all $\la \in W$, $f_\la^n(v_\la)=\infty$, 
 or in other words, there is a persistent virtual cycle.
	But as observed in Remark \ref{rem:vsdense}, this is not possible for a proper TL family.
	
	\medskip
	
	We now prove that the activity locus coincides with $\partial\tcal(\mathbf f)$. Observe that for $\lambda\in \Tf$,  $v_\lambda\in K_\lambda$ hence $f_\lambda^n(v_\la)\in U_\lambda$ for all $n\in\N$, while for $\lambda\in\Lambda\setminus \Tf$, there is $n$ such that  $f_\lambda^n(v_\lambda)\in \ov{V_\la}\setminus U_\la$  (in particular, $v_\lambda$ is passive for $\lambda\in\Lambda\setminus \Tf$). So if $\lambda\in \partial\Tf$, $v_\lambda$ must be active at $\lambda$: it cannot satisfy (1) in Definition~\ref{def:passive}  because the family of iterates of the singular value is not well defined for parameters in $\Lambda\setminus\Tf$, and it cannot satisfy (2) either because  $f_\lambda^n(v_n)\in U_\lambda$ for all $n\in\N$. 
It remains to show that if $\lambda_0$ is in the interior of $\Tf$ then $v_{\lambda}$ is passive at $\lambda_0$. For all $\lambda$ in a neighborhood $D$ of $\lambda_0$ and all $n\in\N$, $f_{\lambda}^n(v_\lambda)\in U_\lambda$, so the family $(f_\lambda^n(v_\lambda))$ takes values in a bounded set when $\lambda\in D$, and normality follows by Montel's Theorem.
\end{proof}

\begin{coro}\label{cor:boundary approximated by interior}
Any parameter $\lambda\in \partial \Tf$ is approximated by attracting parameters.
\end{coro}
\begin{proof}
 Let $\la_0\in\partial \Tf$. Since paramaters with virtual cycles are dense in the activity locus (see Remark \ref{rem:vsdense}), which is $\partial \Tf$, we can assume that $f_{\la_0}^n(v_{\la_0})=\infty$ for some $n\geq 0$.  Since virtual cycles cannot be persistent, the Accessibility Theorem \cite[Theorem B]{ABF} applies, providing a curve of parameters $\la(t)\to \la_0$ such that $f_{\la(t)}$ has an attracting cycle of period $n+1$. 
\end{proof}

\section{Regularity of the straightening map: proof of Theorem E} \label{sec:chihomeo}\ 

In this section we aim to show that for proper equipped tangent-like families,  the straightening map $\chi$ induced by Proposition~\ref{prop:straightening family} is a homeomorphism between $\Tf$ and $\TT$.
We first show continuity of  the straightening map on $\Lambda\setminus\partial \Tf$ 
(Subsection \ref{sec:straightening outside}). Subsection \ref{sec:onT} is devoted to show that $\chi$ in holomorphic on $\mathring{\Tf}$ and continuous on $\Tf$.  In Subsection \ref{sec:bijectivity} we show that fibers of $\chi$ are discrete (see Proposition~\ref{prop:discrete fibers} in Subsection \ref{sec:ff}). Finally, , we use this result, together with a topological argument by Douady and Hubbard, to deduce that $\chi$ is a homeomorphism, and conclude the proof of Theorem E.

\subsection{The straightening map outside of $\partial\Tf$}\label{sec:straightening outside}\ 

Our goal in this subsection is to prove the following statement.
\begin{prop}[Continuity in the stability locus]\label{prop:cont}
Both 	the straightening map $\chi$ and the hybrid conjugacy $\conj_\la$ between $f_\la$ and $T_{\chi(\la)}$ depend continuously on $\la$ if $\la \in \Lambda \setminus \partial \mathcal{T}(\mathbf f)$.
\end{prop}

Let us define
\[
X_n(\lambda)=\bigcap_{0 \leq k \leq n} f_\la^{-k}(U_\la) \setminus K(f_\la).
\]
The next Lemma shows that it suffices to prove that the sequence of area functions 
\[
\area \left( X_n(\la) \right)
\]
decrease uniformly to 0 as $n\to\infty$. 

\begin{lem}[Area condition]\label{lem:area}
	If the sequence of maps $\{\area(X_n(\lambda))\}_{n\geq 0}$	converges locally uniformly to $0$ outside $\partial\mathcal{T}(\mathbf f)$,
	then Proposition \ref{prop:cont} holds.
\end{lem}

\begin{proof}
	Let $C \subset \Lambda \setminus \partial \mathcal{T}(\mathbf f)$ be a compact set.
If we can prove that the hybrid conjugacy $\conj_\la$ depends continuously on $\la \in C$, then the continuity of $\chi$ will follow, 
	since $\chi(\la)=\frac{1}{\conj_\la(v_\la)}$.
		Recall that $\conj_\la$ is the restriction to $V_\la$ of a quasiconformal homeomorphism of $\rs$, whose Beltrami coefficient we denote by $\mu_\la$.
	To prove continuity of $\la \mapsto \conj_\la$, it is enough to prove 
	\begin{itemize}
		\item[(a)] that there exists $0<c<1$ and $\delta>0$ such that for all $\la \in \D(\la_0,\delta)$, 
		$\|\mu_\la\|_{L^\infty}\leq c$;
		\item[(b)] and that the map $\la \mapsto \mu_\la$ is continuous in the $L^1$ topology.
	\end{itemize}
	(See \cite[I. 3)]{douhub}).	  
	In order to prove (a) and (b), let us review the construction of $\conj_\la$.
	
	For all $\la \in \Lambda$, the tubing map  $t_\la$  (recall Definition \ref{def:tubing}) is a quasiconformal homeomorphism between the annulus 
	$V_\la \setminus U_\la$ and the  annulus $V_\sharp \setminus T_{\alpha_{\sharp}}^{-1}(V_\sharp)$, where $T_{\alpha_{\sharp}}:  T_{\alpha_{\sharp}}^{-1}(V_\sharp ) \to V_\sharp $ is TL.
	More precisely, $t_\la$ is  of the form $t_\la = t_{\la_0} \circ {\holom}_\la^{-1}$, where ${\holom}_\la : \overline{V_{\la_0}} \setminus U_{\la_0} \to \overline{V_\la} \setminus U_\la$ is the equivariant motion of the fundamental annulus. Let $\mu_{\la,0}:=\frac{\overline{\partial} t_\la}{\partial t_\la}$ be the Beltrami form on $V_\la \setminus U_\la$, and for any $n \in \N$, let
\[
\mu_{\la,n}:=\left\{
	\begin{array}{ll}
		(f_\la^j)^*\mu_{\la,0} \quad &\text{ on } f_\la^{-j}(V_\la \setminus U_\la) \text{ for $0 \leq j \leq n$} \\
		0 \quad &\text{ elsewhere. }
	\end{array}
	\right.
\]
In particular  $\mu_{\la, n-1}$ and  $\mu_{\la, n}$ only differ in  $f_\la^{-n}(V_\la \setminus U_\la)$. By construction, $\mu_\la = \limn \mu_{\la,n}$, pointwise a.e.
	Property (a) follows from the fact that 
	$\| \mu_\la \|_{L^\infty} = \| \mu_{\la,0} \|_{L^\infty}$, which is locally bounded since 
	$({\holom}_\la)_{\la \in \Lambda}$ is a holomorphic motion. 
	
	Let us now prove (b). By the explicit expression of $\mu_{\la, 0}$ in terms of $\mu_{\la_0,0}$, if $(\la_k)_{k \in \N}$ is a sequence of parameters with $\lim_{k \to +\infty} \la_k = \la_0$, 
	then $\mu_{\la_k, 0}$ converges pointwise a.e. to $\mu_{\la_0, 0}$. Moreover, by definition of a quasiconformal map, we have $|\mu_{\la_k, 0}(z)| <1$ for all $k \in \N$
	and for a.e. $z \in \rs$. By the dominated convergence theorem, it follows that $\lim_{k \to +\infty} \|  \mu_{\la_k,0} - \mu_{\la_0,0} \|_{L^1} = 0$. 
	In other words, $\la \mapsto \mu_{\la, 0}$ is continuous in the $L^1$ topology.
	More generally, by a similar argument, for any fixed $n$, the map $\la \mapsto \mu_{\la, n}$ is also continuous in the $L^1$ topology.
	To ensure continuity of $\la \mapsto \mu_\la$, we must therefore prove that the sequence of maps $(\la \mapsto \mu_{\la, n})_{n \in \N}$ converges uniformly 
	in $C^0(W, L^1(\rs))$, where $W$ is some fixed neighborhood of $\la_0$; but observe that 
	$$\| \mu_{\la, n} - \mu_\la\|_{L^1} \leq  \text{area}\left( \bigcap_{0 \leq k \leq n} f_\la^{-k}(U_\la) \setminus K(f_\la) \right)=\area(X_n(\la)).$$
	It is therefore indeed sufficient to prove that  $\text{area}\left( X_n(\la) \right)$	
	 converges uniformly to $0$ on a neighborhood $W$ of $\la_0$.
	 \end{proof}

We therefore need to estimate the area of the sets $X_n(\la)$. To this end, we shall relate $X_n(\la)$ to $X_n(\la_0)$ via an extension of the equivariant holomorphic motion of the fundamental annulus (see Definition \ref{def:equipped}.) We need to treat two cases separately: $\la_0\in \Lambda\setminus \Tf$ and $\la_0\in\inter(\Tf)$.

In the first case, on $\Lambda \setminus \Tf$,  $V_\la$ moves locally holomorphically with a holomorphic motion that  respects  the dynamics on $U_\la$. In other words, locally on $\Lambda \setminus \Tf$ the maps $f_\la $ are quasiconformally conjugate on $U_\la$. This is the content of the next lemma.

\begin{lem}[Equivariant holomorphic motion of $V_{\la_0}$, for $\la_0\notin \Tf$] \label{lem:existholmot}
	Let $\la_0 \in \Lambda \setminus \Tf$. There exists a neighborhood $W$ of $\la_0$ and a holomorphic motion $H_\la: V_{\la_0} \to V_\la$ 
	 over $W$ such that for all $z \in U_{\la_0}$,
	$$H_\la \circ f_{\la_0}(z) = f_\la \circ H_\la(z).$$
\end{lem}

\begin{proof}[Proof of Lemma \ref{lem:existholmot}]
	Let $\la_0 \notin  \Tf$: then there exists $N \in \N$ such that $f_{\la_0}^N(v_{\la_0}) \in V_{\la_0} \setminus \left( U_{\la_0}  \cup \{u_{\la_0}\} \right)$.
	
	If $f_{\la_0}^N(v_{\la_0})  \in \partial U_{\la_0} \setminus \{u_{\la_0}\}$, then up to reducing slightly 
	$V_{\la_0}$, we may reduce to the case where $f_{\la_0}^N(v_{\la_0}) \in V_{\la_0} \setminus \overline{U_{\la_0}}$,
	which we assume from now on.
	Let $D$ be a small simply connected neighborhood of $\la_0$ such that for all $\la \in D$, 
	$f_\la^N(v_\la) \in V_{\la} \setminus \overline{U_\la}$. Let ${\holom}_\la$ denote the equivariant holomorphic motion of $\overline{V_\la} \setminus U_\la$ given by the definition of equipped family of TL maps. We will modify ${\holom}_\la$ so that  it maps $f_{\la_0}^N(v_{\la_0})$ to  $f_\la^N(v_\la)$. To this end, 
	let $h_\la^{(0)}$ denote the holomorphic motion of $\partial V_{\la_0} \cup \partial U_{\la_0} \cup \{ f_{\la_0}^N(v_{\la_0})\}$ 	over $D$ defined by 
	$$
 h_\la^{(0)}(z)=\left\{
	\begin{array}{ll}
		 {\holom}_\la(z) & \text{ if $z \in\partial V_{\la_0} \cup \partial U_{\la_0}$ } \\
		f_\la^N(v_\la) & \text{ if $z=f_{\la_0}^N(v_{\la_0})$ }.
	\end{array}
	\right.$$
	By  Slodkowski's  extension Theorem~\ref{thm:Slodkowski}, it may be extended to a holomorphic motion of $\rs$; and by continuity, 
	this extension must map $\overline{V_{\la_0}} \setminus U_{\la_0}$ to $\overline{V_\la} \setminus U_\la$.
	We still denote by  $h_\la^{(0)}$ this extended holomorphic motion over $D$ of $\overline{V_{\la_0}} \setminus U_{\la_0}$.
	
	Applying successively Lemma \ref{lem:liftholmot}, we obtain holomorphic motions $h^{(j)}$ (for $0 \leq j \leq N$)
	of $f_{\la}^{-j}(\overline{V_\la} \setminus U_\la)$ over $D$, such that 
	$$f_\la \circ h_\la^{(j+1)} = h_\la^{(j)} \circ f_{\la_0}.$$
	For $j=N+1$, observe that since $h_\la^{(N)}(v_{\la_0}) = v_\la$ and since $v_\la$ is an omitted value, 
	we have $h_\la^{(N)} \circ f_{\la_0}(z) \neq v_\la$ for all $z \in f_{\la_0}^{-N-1}( \overline{V_{\la_0} } \setminus U_{\la_0}   )$. 
	We may therefore still apply Lemma \ref{lem:liftholmot} to construct $h_\la^{(N+1)}$ on  $f_{\la_0}^{-N-1}( \overline{V_{\la_0} \setminus U_{\la_0} })$, and more generally $h_\la^{(j)}$ for all $j \geq N+1$.

	Moreover, since $h_{\la}^{(0)}$ is equivariant and by unicity of the lift, the $h_\la^{(j)}$ glue together continuously, i.e. we have 
	$h_\la^{(j+1)}=h_\la^{(j)}$ on $f_\la^{-j}(\partial U_{\la_0})$.
	Therefore, for every $\la \in D$, the map $H_\la$ defined by
	$$H_\la(z):=h_\la^{(j)}(z) \text{ if $z \in f_{\la_0}^{-j}(\overline{V_{\la_0}} \setminus U_{\la_0}  )$ } $$
	is a quasiconformal homeomorphism satisfying 
	$$f_\la \circ H_\la(z) = H_\la \circ f_{\la_0}(z)$$
	for all $z \in \bigcup_{j \geq 1} f_{\la_0}^{-j}(\overline{V_{\la_0}} \setminus U_{\la_0}) = U_{\la_0} \setminus K(f_{\la_0})$.
	It is also a holomorphic motion of $\bigcup_{j \in \N} f_{\la_0}^{-j}(\overline{V_{\la_0}} \setminus U_{\la_0}) = V_{\la_0} \setminus K(f_{\la_0})$ over $D$, which extends  by Slodkowski's  extension Theorem~\ref{thm:Slodkowski} to a holomorphic motion 
	of $V_{\la_0}$ (here, we use the fact that $J(f_\la) = K(f_\la)$ has no interior if $\la \notin \Tf$).
	By continuity, we still have 
	$$f_\la \circ H_\la(z) = H_\la \circ f_{\la_0}(z)$$
	for all $z \in U_{\la_0}$, which finishes the proof.
\end{proof}

 The proof in the second case is even simpler. For parameters in a  stability component inside $\Tf$,  we do  not need to modify ${\holom}_\la$ to accomodate  the singular value (which never leaves $K(f_\la)$). However, we cannot ensure equivariance in the interior of $K(f_{\la_0})$. 
 
\begin{lem}[Equivariant holomorphic motion of $V_{\la_0}\setminus K(f_{\la_0})$, for $\la_0\in \inter( \Tf)$]\label{lem:sstableq}
	Let $\la_0$ be in a stability component $\SS \subset \tcal(\mathbf{f})$. Then there exists a neighborhood $W$ of $\la_0$ and a holomorphic motion $H: W \times V_{\la_0}\setminus  K(f_{\la_0})  \to \rs$ 
	such that for all $\la \in W$, for all $z \in U_{\la_0} \setminus K(f_{\la_0}) $,
	$$H_\la \circ f_{\la_0}(z) = f_\la \circ H_\la(z).$$
	Moreover, if $\SS$ is a non-hyperbolic component, then $H_\la$ extends further to $ K(f_{\la_0})$ and still satisfies 
	$$H_\la \circ f_{\la_0}(z) = f_\la \circ H_\la(z)$$
	for all $z  \in U_{\la_0}$.
\end{lem}

\begin{proof}[Proof of Lemma \ref{lem:sstableq}]
	Let $W$ be any simply connected domain contained in  $\SS$ and containing $\la_0$.
	Since the family is equipped, there is a holomorphic motion ${\holom}_\la: \overline{V_{\la_0}}\setminus U_{\la_0} \to \overline{V_\la} \setminus U_\la$ such that for all $z \in \partial U_{\la_0} \setminus \{\infty\}$,
	$$
	{\holom}_\la \circ f_{\la_0}(z) = f_\la \circ {\holom}_\la(z).
	$$
	For $n\geq 0$, let $A_n(\la):=f_\la^{-n}(\overline{V_\la} \setminus U_\la)$. Observe that the interior of
	the $A_n(\la)$ are pairwise disjoint, and that 
	$$
	\bigcup_{n \geq 0} A_n(\la)=  \overline{V_\la} \setminus K(f_\la).
	$$
	By Lemma \ref{lem:liftholmot}, there exists a unique lift $h_\la^{(1)}$ to $A_1(\la)$ of the holomorphic motion ${\holom}_\la$, satisfying 
	$$
	f_\la \circ h_\la^{(1)} ={\holom}_\la\circ  f_{\la_0}.
	$$
	But since, by assumption, ${\holom}_\la$ is  continuous and equivariant, i.e. $f_\la \circ {\holom}_\la(z) ={\holom}_\la(z)\circ  f_{\la_0}$ for all $z \in \partial U_{\la_0} \setminus \{\infty\}$, we have in fact $h_\la^{(1)}(z)={\holom}_\la(z)$ for all $z \in \partial U_{\la_0} \setminus \{\infty\}$. Additionally, $h_\la^{(1)}$ extends to $\overline{A_1(\la)}$ by the $\la$-lemma, with  $h_\la^{(1)}(\infty)=\infty$. In other words, ${\holom}_\la$ and $h_\la^{(1)}$ coincide on their common boundary.
	Moreover, by definition ${\holom}_\la$ maps $A_0(\la_0)$ to $A_0(\la)$, and $h_\la^{(1)}$ maps $A_1(\la_0)$ to $A_1(\la)$, whose interior are pairwise disjoint;
	therefore, for all $\la \in W$, the following map is well-defined and injective:
	$$H_\la(z)=\left\{
	\begin{array}{ll}
		{\holom}_\la(z) &\text{ \quad if $z \in A_0(\la)$}\\
		h_\la^{(1)}(z) &\text{ \quad if $z \in A_1(\la)$}\\
	\end{array}
	\right.$$
	It is also clearly holomorphic in $\la$, hence it defines a holomorphic motion.
	By construction, we have $H_\la \circ f_{\la_0}(z) = f_{\la} \circ H_\la(z)$ for all $z \in A_1(\la)$.

	Proceeding in a similar way, we define inductively $h_\la^{(n+1)}$ as a lift of $h_\la^{(n)}$ using Lemma \ref{lem:liftholmot}. This is possible since for all $\la \in W$, $v_\la \in K(f_\la)$. We also define $H_\la(z):=h_\la^{(n)}(z)$ for all $z \in A_n(\la)$ and all $\la \in W$.
	Arguing as above, this defines a holomorphic motion on $W$, 
	and for all $(\la,z) \in W \times U_{\la_0} \setminus K(f_{\la_0})$,
	\[
	f_\la \circ H_\la(z) = f_{\la_0} \circ H_\la(z).
	\]

	Let us now assume that $\SS$ is a non-hyperbolic component.
	Again,  we may extend this holomorphic motion thanks to Slodkowski's theorem to a holomorphic motion 
	of $V_{\la_0}$; and using the fact that  $J(f_\la) = K(f_\la)$ in a non-hyperbolic stability component, we still have by continuity
	$$f_\la \circ H_\la(z) = H_\la \circ f_{\la_0}(z)$$
	for all $z \in U_{\la_0}$ and all $\la \in W$, which finishes the proof.	
\end{proof}

We are finally ready to conclude the proof of continuity of the straightening map  on  $\Lambda \setminus \partial \Tf$.

\begin{proof}[Proof of Proposition \ref{prop:cont}]
	Recall that $X_n(\la):= \bigcap_{0 \leq k \leq n} f_\la^{-k}(U_\la) \setminus K(f_\la) $, and  that we want to show that $\area(X_n(\la))$ decreases uniformly to $0$.
	We have that 
	\begin{align*}
		\text{area}\left( X_n(\la) \right) &=
		\text{area}\,  H_\la\left(  X_n(\la_0) \right) \\	
		&=\int_{X_n(\la_0)} \mathrm{Jac} \, H_\la(z) |dz|^2. 
	\end{align*}
	
	Up to reducing a little bit the set $W$ given by the lemma above, we may assume without loss of generality that there exists a constant $K>0$ such that 
	for all $\la \in W$, the quasiconformal distortion of $H_\la$ is less than $K$.
	Recall that since  $\jac H_\la = |\partial H_\la|^2 - |\overline{\partial} H_\la|^2$ 
	and $\| dH_\la\| =  |\partial H_\la| + |\overline{\partial} H_\la|$,
	there exists a constant $C=C(K)>1$ (explicit) such that for all $\la \in W$,
	\begin{equation}\label{eq:comparjac}
		C^{-1} \| dH_\la\|^2 \leq \jac H_\la \leq C \| dH_\la\|^2 .
	\end{equation}

	Let $N_n(\la):= \int_{X_n(\la_0)}\| dH_\la(z)\|^2  |dz|^2$.   For every fixed $\la \in W$, the map $z \mapsto \jac H_\la(z)$ is integrable,
	therefore by \eqref{eq:comparjac}, so is the map $z \mapsto\| dH_\la(z)\|^2.$ Since the $X_n(\la)$ are decreasing and their intersection is empty, it follows that 
	$N_n$ converges pointwise a.e. to $0$. 
	By  \cite[Theorem 5.7.4 p. 188]{astala2008elliptic},
	then for a.e. $z$, the maps $\la \mapsto  \partial H_\la(z)$ and $\la \mapsto \overline{\partial} H_\la(z)$ are holomorphic; therefore $N_n$ is subharmonic. 
	Let $\epsilon>0$ be small enough that $\D(\la_0,2\epsilon) \subset W$. Then, since $N_n$ is subharmonic, we have for all $\la \in \D(\la_0, \epsilon)$:
	\begin{equation}\label{eq:ineqN}
		N_n(\la) \leq \frac{1}{\pi \epsilon^2} \int_{\D(\la,\epsilon)} N_n(u) \cdot |du|^2.
	\end{equation}
	
	The sequence $(N_n)_{n \in \N}$ is pointwise decreasing, $N_0$ is integrable and $(N_n)_{n \in \N}$ converges a.e. to $0$.
	Therefore, by the dominated convergence theorem, $\limn  \int_{\D(\la_0,2\epsilon)} N_n(u) \cdot |du|^2=0$.
	By \eqref{eq:comparjac} and \eqref{eq:ineqN}, we therefore finally have, for all $\la \in \D(\la_0, \epsilon)$:
	\begin{equation*}
	\int_{X_n(\la_0)} \mathrm{Jac} \, H_\la(z) |dz|^2 \leq 
	\frac{ C}{\pi \epsilon^2} \int_{\D(\la_0, \epsilon)} N_n(u) \cdot |du|^2 \to_{n\to\infty} 0.
	\end{equation*}
	This proves that $\limn\sup_{\la \in \D(\la_0,\epsilon)} \text{area}(X_n(\la)) = 0$,
	and we are done by Lemma \ref{lem:area}.
\end{proof}

\subsection{Regularity of the straightening map on $\tcal(\mathbf{f})$} \label{sec:onT}\

The goal in this section is to prove that the straightening map $\chi$ is continuous in $\partial \Tf$, holomorphic and proper in hyperbolic components of $\Tf$ and holomorphic in non-hyperbolic components.  We start by the first statement (c.f. \cite[Prop. 14, p.313]{douhub}).

\begin{lem}[Continuity on the boundary]\label{lem:continuitybt}
	The straightening map $\chi$ is continuous at every point $\la \in \partial \tcal(\mathbf{f})$, and $\chi(\partial \tcal(\mathbf{f})) \subset \partial \tcal$. 
	\end{lem}

\begin{proof}
	Let $\lambda_* \in \partial \mathcal T(\mathbf f)$, and let $\lambda_n \to \lambda_*$.
	Let $\conj_n$ denote the hybrid conjugacies between $f_{\lambda_n}$ and $T_{\alpha_n}$, where $\alpha_n=\chi(\lambda_n)$.
	By Proposition \ref{prop:convergence of conjugacies}, assume without loss of generality that $\conj_n$ converges locally uniformly to some quasiconformal homeomorphism $\conj$, that $\alpha_n$ converges to some $\alpha \in \mathcal T$. We must prove that $\alpha=\chi(\la_*)$.
	
	As above, recall from the straightening construction that the hybrid conjugacies $\conj_n$ conjugate a quasiregular map $F_{\la_n}$ (the straightening model) to $T_{\alpha_n}$ on all of $\hat \C$. Moreover, $F_{\la_n}$ converges to $F_{\la_*}$ (since choices were made once and for all using the tubing maps $t_{\la_n}$). Therefore, $\conj$ conjugates $F_{\la_*}$ to $T_\alpha$ and $\conj_{\la_*}$ conjugates   $ f_{\la_*}$ to $T_{\chi(\la_*)}$, globally on $\hat \C$; hence $\conj \circ \conj_{\la_*}^{-1}$ globally conjugates $T_{\alpha}$ to $T_{\chi(\la_*)}$.	Moreover, the Beltrami forms $\mu_n$ of $\conj_n$ all vanish in a definite neighborhood of $t_{\la_n}^{-1}(0)$, and so does the Beltrami form of $\conj_{\la_*}$.
	Therefore, $\conj \circ \conj_{\la_*}^{-1}$ is conformal in a neighborhood of $0$, and hence it is conformal on the whole basin. Indeed,  the Beltrami form of $\conj \circ \conj_{\la_*}^{-1}$ vanishes in a neighborhood of $0$ and it is $T_{\chi(\la_*)}-$invariant, hence it vanishes in the whole basin. 
		
	But now this implies that $\alpha=\chi(\la_*)$ by Proposition \ref{prop:quasiconformalclasses}, since the quasiconformal conjugacy is conformal in the basin of $0$. 
	\end{proof}

In fact we see below that $\chi$ maps hyperbolic components holomorphically and properly onto hyperbolic components.

\begin{prop}[$\chi$ is holomorphic and proper in hyperbolic components]\label{prop:hyperbolic}
	Let ${\SS}$ be a hyperbolic component of $\tcal(\mathbf{f})$.
	The straightening map $\chi_{|{\SS}}$ is holomorphic and maps ${\SS}$  properly into a hyperbolic component ${\SS'}$ of $\tcal$.
\end{prop}

\begin{proof}
	Clearly, if $f_\la$ has an attracting cycle, then so does $T_{\chi(\la)}$. Therefore, $\chi$ maps ${\SS}$ to the subset of $\mathcal T$ of hyperbolic parameters.
	By Proposition \ref{prop:cont}, $\chi$ is continuous on ${\SS}$, so there is a hyperbolic component ${\SS'}$ of $\inter(\tcal)$
	such that $\chi({\SS}) \subset {\SS'}$ and $\chi: {\SS} \to {\SS'}$ is continuous. Moreover, the multiplier of the attracting cycle is invariant under hybrid conjugacy, so we have $\rho_{\SS} = \rho_{\SS'} \circ \chi$, where 
	$\rho_{\SS} : {\SS} \to \D^*$ and $\rho_{\SS'} : {\SS'} \to \D^*$ are the multiplier maps on ${\SS}$ and ${\SS'}$ respectively.
	
	By \cite[Theorem 6.4]{FK}, the map $\rho_{\SS'}$ is a holomorphic universal cover; in particular, it is locally invertible.
	By the Implicit Function Theorem,
	the map $\rho_{\SS} : {\SS} \to \D^*$ is holomorphic. Therefore, we may locally write $\chi  = \rho_{\SS'}^{-1} \circ \rho_{\SS}$, 
	which implies that $\chi$ is holomorphic on ${\SS}$.
	
	By Lemma \ref{lem:continuitybt}, the map $\chi$ is moreover continuous on $\overline{{\SS}}$,
	and $\chi(\partial {\SS}) \subset \partial  \tcal \cap \overline{{\SS'}} = \partial {\SS'}$.
	Therefore, $\chi: {\SS} \to {\SS'}$ is proper.
\end{proof}

To end this section, we prove that the straightening map is also holomorphic in the (conjecturally non-existent) non-hyperbolic components of $\Tf$ (c.f.~\cite[Lemma 42.13 p. 553]{LyubichBook}).

\begin{prop}[Holomorphic in non-hyperbolic components]\label{prop:non-hyperbolic}
		Let $\SS$ be a non-hyperbolic component of $\tcal(\mathbf{f})$.
	The straightening map $\chi_{|\SS}$ is holomorphic.
\end{prop}

\begin{proof}
	Let $\la_0 \in \SS$, and let $W, H_\la$ be given by Lemma \ref{lem:sstableq}.
	Up to restricting $W$ if necessary, assume that $W \subset \SS$ and that 
	$W$ is simply connected. 
	
	Let $\mu_\la:=\frac{\dbar H_\la}{\partial H_\la}$.  Then, 
	$f_{\la_0}^*\mu_\la = \mu_\la$ and $\la \mapsto \mu_\la$ is holomorphic. Let $\nu_\la(z):=\conj_* \mu_\la(z)$ for $z \in J(T_{\chi(\la_0)})$ and $0$ outside, where $\conj=\conj_{\la_0}$ is the hybrid conjugacy 
	between $f_{\la_0}$ and $T_{\chi(\la_0)}$.
	Then $T_{\chi(\la_0)}^* \nu_\la=\nu_\la $.

	We claim that $\la \mapsto \nu_\la$ is holomorphic. If $\Phi$ were holomorphic, this would be immediate, but in general pulling back (or pushing forward) Beltrami forms by a quasiconformal homeomorphism does not define a holomorphic operator from the space of Beltrami forms to itself. However,  $\Phi$ is a hybrid conjugacy, so it behaves like a holomorphic map on the support of $\mu_\la$, which turns out to be enough.  
	Indeed, since $\dbar \conj=0$ a.e. on $J(f_{\la_0})$, for any $z\in J(f_{\la_0})$ we have that (see e.g. \cite{bf}, (1.9) and (1.10) p. 16):
		$$
	\nu_\la(\conj(z)) = \mu_\la(z)\cdot \frac{\partial \conj(z)}{\overline{\partial \conj(z)}} 
	$$
	
	from which it easily follows that $\la \mapsto \nu_\la$ is indeed holomorphic.

	Let $h_\la$ be the quasiconformal homeomorphisms which integrates $\nu_\la$, normalized so that $h_\la$ fixes $0, 1, \infty$. Since $\nu_\la$ is supported on $J(T_{\chi(\la_0)})$, it follows that $h_\la$ is conformal on the basin of $0$. Moreover,  $\la\to h_\la(z)$ is holomorphic, since $\nu_\la$ depends holomorphically on $\la$.
	 Then $h_\la \circ T_{\chi(\la_0)} \circ h_\la^{-1}$  is holomorphic with respect to $z$, has two asymptotic values  and has a fixed point of multiplier $\frac{1}{8}$ at $z=0$ (because $h_\la$ is conformal in the basin of $0$). Hence, by Lemma \ref{lem:rigidity}, it is of the form 
	$h_\la \circ T_{\chi(\la_0)} \circ h_\la^{-1} = T_{\alpha(\la)}$,
	for some  map $\alpha: W \to \C$. It follows that $\alpha(\la)=\frac{1}{h_\la(1/\chi(\la_0))}$ and hence $\lambda\mapsto \alpha(\lambda)$ is holomorphic.

	Finally, observe that by construction, $T_{\alpha(\la)}$ is hybrid conjugate to $T_{\chi(\la)}$ via the quasiconformal homeomorphism $\conj_\la \circ H_\la \circ \conj_{\la_0}^{-1} \circ h_\la^{-1}$. By Theorem \ref{th:unicity}, 
	$\conj_\la \circ H_\la \circ \conj_{\la_0}^{-1} \circ h_\la^{-1}$ therefore coincides with $h_{{k(\la)}}$, for some $k(\la) \in \Z$ depending a priori on $\la$ (where $h_k$ is the hybrid conjugacy induced by the wringing map, see Proposition \ref{prop:wringing} for details). By Proposition \ref{prop:cont}, $\la \mapsto \conj_\la$ is continuous (for the topology of uniform convergence), and since $h_\la$ and $H_\la$ are holomorphic motions, the map $\la \mapsto \conj_\la \circ H_\la \circ \conj_{\la_0}^{-1} \circ h_\la^{-1}$ is also continuous. But this implies that $\la \mapsto k(\la)$ is constant, and since $k(\la_0)=0$, 
	we have $\conj_\la \circ H_\la \circ \conj_{\la_0}^{-1} \circ h_\la^{-1}= h_0=\id$ for all $\la \in W$.
	
	Therefore, $\alpha(\la) = \chi(\la)$ and in particular, $\chi$ is holomorphic on $\SS$.
\end{proof}

Note that at this point, we have merely proved that $\chi$ is holomorphic on any non-hyperbolic  component $\SS$, and that is continuous on its closure $\overline{\SS}$. We have not yet excluded the possibility of $\chi$ being constant on such a component; however, this will be ruled out by Proposition \ref{prop:discrete fibers} below.

\subsection{Finiteness of the fibers of $\chi$ above $\tcal$}
\label{sec:ff} \ 

A crucial part of the proof of bijectivity of $\chi : \tcal(\mathbf f) \to \tcal$ is to show that its fibers are discrete (and hence finite, by compactness).  We will use the following preliminary result, that connects the equivariant motion of the fundamental annulus with the hybrid conjugacies between the maps in the TL-family and the members of the  model family.

\begin{prop}[Matching hybrid conjugacies with external homeomorphisms] \label{prop:keytool}
	Let $\la_0 \in \partial \tcal(\mathbf f)$, such that $u_{\la_0} \neq v_{\la_0}$. Let ${\holom}_\la$ denote the equivariant holomorphic motion of $\overline{V_\la} \setminus U_\la$, with basepoint $\la_0$. There exists $\eps_0>0$
	such that if  $\la \in \D(\la_0,\eps_0)$ and $\chi(\la)=\chi(\la_0)$, then 
	${\holom}_\la$ extends to a quasiconformal homeomorphism $h_\la: V_{\la_0} \to V_\la$ such that for all $z \in U_{\la_0}$, we have
	$$f_\la \circ h_\la(z) = h_\la \circ f_{\la_0}(z)$$
	and moreover $h_\la$ coincides with  $\conj_\la^{-1} \circ \conj_{\la_0}$ on $J(f_{\la_0})$, 
	where $\conj_\la, \conj_{\la_0}$ denote the hybrid conjugacies associated to $\la$ and $\la_0$ respectively (see Proposition \ref{prop:straightening family}).
\end{prop}

\begin{proof}
Let $\alpha:=\chi(\la_0)$, and let $\la_1 \in \Lambda$ such that $\chi(\la_1) = \chi(\la_0)$.	
We have the following commuting diagram:

\[
\begin{CD}
\conj_{\la_1}(\partial U_{\la_1} )\setminus \{\infty\}   @>T_\alpha>> \conj_{{\la_1}}(\partial V_{{\la_1}}) \\
@A{\conj_{{\la_1}}}AA  @AA{\conj_{{\la_1}}}A \\
\partial U_{\la_1}\setminus \{u_{\la_1}\}  @>f_{{\la_1}}>> \partial V_{\la_1}\\
@Ah_{\la_1} AA  @AAh_{\la_1} A \\
\partial U_{\la_0}\setminus \{u_{\la_0}\}  @>f_{\la_0}>> \partial V_{\la_0}\\
@V{\conj_{\la_0}}VV  @VV{\conj_{\la_0}}V \\
\conj_{\la_0}(\partial U_{\la_0}) \setminus \{\infty\}  @>T_\alpha>> \conj_{\la_0}(\partial V_{\la_0})
\end{CD}
\]

Consider the map $\Phi_{\la_1}:= \conj_{{\la_1}} \circ h_{\la_1} \circ \conj_{\la_0}^{-1}:\conj_{\la_0}(\overline{V_{\la_0}} \setminus U_{\la_0})\ra \conj_{\la_1}(\overline{V_{\la_1}}\setminus U_{\la_1})$. By the diagram above, 
$$T_\alpha \circ \Phi_{\la_1}(z) = \Phi_{\la_1} \circ T_\alpha(z)$$
for all $z \in \partial U_{\la_0} \setminus \{u_{\la_0}\}$.
By Lemma \ref{lem:extendh},  $\Phi_{\la_1}$ extends to a homeomorphism $ \Phi_{\la_1} : \conj_{\la_0}(\overline{V_{\la_0}}) \setminus K(T_\alpha) \to \conj_{\la_0}(\overline{V_{\la_0}}) \setminus K(T_\alpha)$ 
which commutes with $T_\alpha$ on $\conj_{\la_0}(U_{\la_0}) \setminus K(T_\alpha)$.

 Let $p$ be the marked pole of $T_\alpha$ (see Lemma \ref{lem:labelling}),  and  let   $p_\la:=\conj_\la^{-1}(p)$. 
Let $g_{\la_0}$ denote the inverse branch of $f_{\la_0}$  mapping $\infty$ to $p_{\la_0}$, and let $g_\la$ be the corresponding inverse branch for $f_\la$. There exists $\delta>0$ such that for all $\la$ sufficiently close to $\la_0$, $g_\la$ is well-defined on $\D(u_{\la_0}, 2\delta)$, and $\la \mapsto g_\la(z)$ is holomorphic for all $z \in \D(u_{\la_0}, 2\delta)$.

Let $\mathcal{N} \subset \conj_{\la_0}( f_{\la_0}^{-1}( V_{\la_0} \setminus U_{\la_0}   )  ) 
= T_\alpha^{-1}(\conj_{\la_0}( V_{\la_0} \setminus U_{\la_0}    ))$ be a domain such that $p \in \partial \mathcal{N}$, and such that $\mathcal{N}$ is compactly contained in $\conj_{\la_0}(U_{\la_0})$.
This is possible since 
$$p_{\la_0} \in f_{\la_0}^{-1}(\partial U_{\la_0}) \subset U_{\la_0}.$$

Moreover, up to reducing $\mathcal{N}$ if necessary, we may assume that for all $\la$ close enough to $\la_0$, 
$$h_\la \circ f_{\la_0} \circ \conj_{\la_0}^{-1}(\mathcal{N}) \subset \D(u_{\la_0}, \delta).$$
Observe that this expression is well-defined not just for $\la=\la_1$ but for all $\la$ close enough to $\la_0$, 
since by construction  $ f_{\la_0} \circ \conj_{\la_0}^{-1}(\mathcal{N}) \subset V_{\la_0} \setminus U_{\la_0}$.

Let $\eps>0$ be the constant given by Proposition \ref{prop:extend}.
We will prove that there exists $\eps_0>0$ such that if $|\la_1-\la_0|<\eps_0$, then 
$\sup_{z \in \mathcal{N} } | \Phi_{\la_1}(z)-z| < \eps$.

Recall that the hybrid conjugacies $\conj_\la$ depend continuously on $\la$ at $\la_0$ (Lemma \ref{lem:continuitybt}).
By some diagram chasing, we may write, for $z \in \mathcal{N}$,
\begin{equation}
	 \Phi_{\la_1}(z) = \conj_{\la_1} \circ g_{\la_1} \circ h_{\la_1} \circ f_{\la_0} \circ \conj_{\la_0}^{-1}(z) = 
	\left(\conj_{\la_1} \circ g_{\la_1}\right) \circ h_{\la_1} \circ  \left(\conj_{\la_0} \circ g_{\la_0}\right)^{-1}(z).
\end{equation}

As noted above, this expression is defined not just for $\la=\la_1$ but for all $\la$ close enough to $\la_1$.
Moreover, since $\conj_\la, g_\la, h_\la$ all depend continuously on $\la$ (uniformly in $z$), the map 
$$\la \mapsto \conj_\la \circ g_\la \circ h_\la \circ  f_{\la_0} \circ \conj_{\la_0}^{-1} $$
depends continuously on $\la$, uniformly on $z \in \mathcal{N}$. Since 
$$\conj_{\la_0} \circ g_{\la_0} \circ h_{\la_0} \circ f_{\la_0} \circ \conj_{\la_0}^{-1} = \id,$$
it follows that there exists $\eps_0>0$ such that if $|\la_1-\la_0|<\eps_0$, then 
$\sup_{z \in \mathcal{N} } | \Phi_{\la_1}(z)-z| < \eps$.
We can then apply Proposition \ref{prop:extend} to conclude that  the map $ \Phi_{\la_1}$ extends continuously to $J(T_\alpha)$ by the identity, hence it extends continuously to $K(T_\alpha)$ by the identity. The desired extension of $h_{\la_1}$ is then given by 
$$
\tilde h_{\la_1}(z):= \left\{
\begin{array}{ll}
	\conj_{\la_1}^{-1} \circ  \Phi_{\la_1} \circ \conj_{\la_0}(z) & \text{ if $z \in \overline{V_{\la_0}} \setminus K(T_\alpha)$ } \\
	\conj_{\la_1}^{-1}  \circ \conj_{\la_0}(z) & \text{ if $z \in K(f_{\la_0})$. }
\end{array}
\right.
$$

\end{proof}

We can now prove that the fibers of $\chi$ are finite (c.f.  \cite[Lemma 42.15, p. 554]{LyubichBook}).

\begin{prop}[Finiteness of fibers]\label{prop:discrete fibers}
	The fibers of $\chi : \tcal(\mathbf f) \to \tcal$ are finite.
\end{prop}
\begin{proof}	
	By compactness of $\tcal(\mathbf f)$, it is enough to prove that fibers are discrete. 	
	We  first deal with the case where $T_\alpha$ has a virtual cycle. Then, there exists $n \in \N$ such that for any $\la \in \chi^{-1}(\{\alpha\})$, $f_\la^n(v_\la)=u_\la$; and since $\la \mapsto f_\la^n(v_\la)$ is holomorphic and non-constant (by the definition of a proper TL family), $\chi^{-1}(\{\alpha\})$ is discrete.

	From now on, we therefore assume that $T_\alpha$ has no virtual cycle.
	Assume, for contradiction,  that there exists $\alpha \in \tcal$ such that $\chi^{-1}(\{\alpha\})$ is infinite. Then, there exists an injective sequence $\la_n \in  \tcal(\mathbf f)$  such that $\chi(\la_n) = \alpha$.

	Up to extracting a subsequence, we may assume that it converges to some $\la_\infty \in  \tcal(\mathbf f)$. Moreover,  since $\chi$ is holomorphic on $\inter(\TT(\mathbf f))$ (Propositions \ref{prop:hyperbolic} and \ref{prop:non-hyperbolic}), we 
	may assume without loss of generality that $\la_\infty \in \partial \tcal(\mathbf f)$. Since $\chi$ is continuous on $\TT(\mathbf f)$ (Lemmas \ref{lem:continuitybt} and Proposition \ref{prop:cont}), $\chi(\la_\infty)=\alpha$, so $f_{\la_\infty}$ is hybrid conjugate to $f_{\lambda_n}$ for all $n$. 	
	Let 
	$$
	h_\la : \overline{V_{\la_\infty}} \setminus U_{\la_\infty} \to \overline{V_{\la}} \setminus U_{\la}
	$$
	be  the equivariant holomorphic motion of the fundamental annulus,	given by the definition of an equipped TL family with the choice of  $\lambda_\infty$ as a base point. Recall that $h_\la$ conjugates $f_{\la_\infty}$ to $f_\la$ on $\partial U_{\la_\infty}\setminus \{u_{\la_\infty}\}$. Let us extend it by   Slodkowski's  extension Theorem~\ref{thm:Slodkowski} to a holomorphic motion of $\rs\setminus U_\lambda$ over $\Lambda$, which we still denote by $h_\la$.

We define a family of Beltrami forms $(\mu_\lambda)_{\lambda\in\Lambda}$ on $\rs$	in the following way. For each $\lambda$ first consider the pullback under $h_\lambda$ of the standard complex structure on $\rs \setminus U_\lambda$.

Then extend $\mu_\lambda$   to $\rs\setminus K(f_{\la_\infty})$ by pulling it back successively under $f_{\lambda_\infty}$, in such a way that it is   $f_{\lambda_\infty}$-invariant on $U_{\lambda_\infty} \setminus K(f_{\la_\infty})$; finally extend it to  $K(f_{\la_\infty})$ by setting $\mu_\la=0$ on $K(f_{\la_\infty})$.

Choose any 3 distinct points $z_1, z_2, z_3 \in \overline{V_{\la_\infty}} \setminus U_{\la_\infty}$.
By the Measurable Riemann Mapping Theorem, there exists a holomorphic family of quasiconformal maps $H_\lambda:\rs\to\rs$,  which integrate $\mu_\la$ and such that $H_{\la}(z_i) = h_\la(z_i)$, $1 \leq i \leq 3$.
This induces a family of holomorphic maps  $G_\la: H_\la(U_{\la_\infty} )\to H_\la(V_{\la_\infty})$, defined as 
$$
G_\la:=H_\la \circ f_{\la_\infty} \circ H_\la^{-1}, 
$$ 
which are all holomorphic since  $\mu_\la = f_{\la_\infty}^*\mu_\la$ on $U_{\la_\infty}$,
and are all  hybrid conjugate to $f_{\la_\infty}$ since $\mu_\la=0$ a.e. on $K(f_{\la_\infty})$. Additionally,
 $G_\la(z)$ depends holomorphically on $\la$ (see e.g. \cite[Lemma 1.40]{bf}).

	On the other hand, recall that for all $n \in \N$, $f_{\la_n}$ and $f_{\la_\infty}$ are both hybrid conjugate to $T_\alpha$ by respective hybrid conjugacies $\conj_n$ and $\conj_\infty$. Set $\psi_n:= \conj_n^{-1}\circ \conj_\infty$, which is a hybrid conjugacy between $f_{\la_\infty}$ and $f_{\la_n}$. By Proposition \ref{prop:keytool}, for all $n$ large enough, $h_{\la_n}:\rs\setminus U_{\la_\infty} \ra \rs\setminus U_{\la_n}$ extends 
 	to a quasiconformal homeomorphism $\tilde{h}_{\la_n}:\rs\to\rs$, which coincides with $\psi_n$ on $K(f_{\la_\infty})$  and such that, for all $z\in U_{\la_n}$,
	 $$
	 f_{\la_n}(z):=\htil_{\la_n} \circ f_{\la_\infty} \circ \htil_{\la_n}^{-1}(z).
	 $$

Hence we have the following commuting diagram.
\[
\begin{CD}
U_{\la_n} @>{f_{\la_n}}>> V_{\la_n}\\
@A{\htil_{\la_n}}AA @AA{\htil_{\la_n}}A\\
U_{\la_\infty} @>{f_{\la_\infty}}>> V_{\la_\infty}\\
@V{H_{\la_n}}VV @VV{H_{\la_n}}V\\
H_{\la_n}(U_{\la_\infty}) @>{G_{\la_n}}>> H_{\la_n}(V_{\la_\infty})\\
\end{CD}
\]

	 We claim that $\tilde{h}_n = H_{\la_n}$. Let $\sigma_n:=\frac{\dbar \htil_{\la_n}}{\partial \htil_{\la_n}}$ be the Beltrami form associated to $\htil_{\la_n}$.
	Since $\htil_{\la_n}$ coincides with $h_{\la_n}$ on $\rs\setminus U_{\la_\infty}$, 
	we have $\sigma_n = \mu_{\la_n}$ on $V_{\lambda_\infty}\setminus U_{\la_\infty}$. 
	Therefore, since both are $f_{\la_\infty}$-invariant, we also have $\sigma_n = \mu_{\la_n}$ on	 $ \rs  \setminus K(f_{\la_\infty})$; and $\sigma_n = \mu_{\la_n}=0$ a.e. on $K(f_{\la_\infty})$. Therefore $\sigma_n = \mu_{\la_n}$. 
	Finally, both integrating maps $H_{\la_n}$ and $\htil_{\la_n}$ must be equal since they coincide on the 3 points $z_1, z_2, z_3$ introduced above.
	
	Hence $f_{\la_n} = G_{\la_n}$ for all $n \in \N$.
	By the   Identity Theorem  applied to $\lambda \mapsto f_{\la}(z) -  G_\la(z)$, we therefore have $f_{\la} =  G_\la=H_{\la} \circ f_{\la_\infty} \circ H_{\la}^{-1}$
	for all $\la$ in a neighborhood of $\la_\infty$. But this would imply that $J(f_\la)$ is connected for all 
	$\la$ in a neighborhood of $\la_\infty$, contradicting $\la_\infty \in \partial \tcal(\mathbf f)$.
\end{proof}

\subsection{Bijectivity and proof that  $\chi$ is a homeomorphism} \label{sec:bijectivity}\ 

We start by proving surjectivity, via topological holomorphicity.

\begin{defi}[Local degree]
	Let $X$ and $Y$ be two oriented (real) surfaces, and $h: X \to Y$ a continuous map. Let $x \in X$ be such that $x$ is isolated in $h^{-1}(h(x))$.
Let $U$ and $V$ be simply connected neighborhoods of $x$ and $y:=h(x)$ respectively,  	and such that $h(U) \subset V$ and $\{x\} = U \cap h^{-1}(y)$. If $\gamma$ is a loop in $U\setminus \{x\}$ with winding number 1 around $x$ then  the local degree  $i_x(h)$ is the winding number of $h \circ \gamma$ around $y$, that is 
	$$
	i_x(h)=\ind(h \circ \gamma,y).
	$$	
\end{defi}

We recall the topological argument principle, see e.g. \cite[Proposition 3.9]{LyubichBook}. 

\begin{prop}[Topological Argument Principle]
Let $D\subset \R^2$ be a domain bounded by a Jordan curve $\gamma$, $h:\ov{D}\ra \R^2$ be a continuous map such that $h$ does not take a value $y$ on $\gamma$. Assume that the preimage $\{h^{-1}(y)\}$ form a discrete set in $D$.  Then  
$$
\ind(h\circ\gamma,y)=\sum_{x\in h^{-1}(y)}i_x(h)
$$
\end{prop}

 Notice that given $\gamma$, for any $\tilde x$ sufficiently close to $x$, 
\begin{equation}\label{eq:continuity of winding number}
\ind(h \circ \gamma,h(x))=\ind(h \circ \gamma,h(\tilde x)).
\end{equation}

\begin{defi}[Topologically holomorphic]
	Let  $X$ and $Y$ be two oriented (real) surfaces, and $h: X \to Y$ be a continuous map. Let $T \subset Y$ be closed, and let $P:=h^{-1}(T)$. We will say that $h$ is topologically holomorphic over $T$  if for all $x \in P$, $h^{-1}(\{h(x)\})$ is discrete and $i_x(h)>0$.
\end{defi}

The following propositions say that in our case it suffices to prove topological holomorphicity to have surjectivity of $\chi$ (c.f. \cite[Prop.19, p. 323]{douhub} and \cite[Thm. 4, p. 326]{douhub}).

\begin{prop}[Topological holomorphicity implies surjectivity]\label{prop:surjectivity}
	Let $X,Y, h, P,T$ be as above, and assume that  $h:P\to T$ is proper, topologically holomorphic over $T$ and that $T$ is connected.  Then $h : P \to T$ is  surjective, and for all   $y \in T$,  its topological degree
	\begin{equation}\label{eq:topological degree}
	\delta := \sum_{x_j \in h^{-1}(y)} i_{x_j}(h) \geq 1
	\end{equation}
	is independent from $y$ and is called the  degree of $h$.
\end{prop}

\begin{prop}\label{prop: chi surjective}
	The map $\chi$ is topologically holomorphic over $\mathcal{T}$,  and proper, hence surjective.
\end{prop}

\begin{proof}
Since $\chi^{-1}(\mathcal T)=\mathcal{T}(\mathbf f)$ is closed and it is disjoint from $\partial \Lambda$ by assumption, hence compact, $\chi$ is proper. 
 Fibers of $\chi$ are discrete by Proposition \ref{prop:discrete fibers}, so we only need to prove that if $\chi(\la) \in \tcal$, then $i_\la(\chi)>0$.
	We will distinguish two cases, depending on whether $\chi(\la) \in \inter(\tcal)$ or $\chi(\la) \in \partial \tcal$.
	
	Assume first that $\chi(\la) \in \inter(\tcal)$. By Lemma \ref{lem:continuitybt}, $\chi(\partial \tcal(\mathbf{f})) \subset \partial \tcal$, 
	therefore $\la \in \inter(\tcal(\mathbf{f}))$. Then, by Propositions \ref{prop:hyperbolic} and \ref{prop:non-hyperbolic}, the map $\chi$ is holomorphic near $\la$,
	from which it follows that $i_\la(\chi) \geq 1$.
	
	Let us now assume that $\chi(\la) \in \partial \tcal$.  Then, $\la\in\partial \Tf$, by the open mapping principle applied to parameters in $\inter(\Tf)$. Let $D$ be a disk around $\la$ small enough that $D \cap \chi^{-1}(\{\chi(\la)\}) = \{\la\}$, and let $\gamma:=\partial D$.

	Since $\la \in \partial \tcal(\mathbf{f})$,  by Corollary \ref{cor:boundary approximated by interior} there exists $\tilde \la \in \inter(\tcal(\mathbf{f}))$ arbitrarily close to $\la$. By (\ref{eq:continuity of winding number}), we can choose  $\tilde \la$ close enough to $\la$ so  that 
	$$
i_\la(\chi):=\mathrm{ind}(\chi \circ \gamma,  \chi(\la))=\mathrm{ind}(\chi \circ \gamma,  \chi(\tilde\la)).
	$$
	By the Topological Argument Principle we have 
	$$
	\mathrm{ind}(\chi \circ \gamma,  \chi(\tilde \la)) =\sum_{\la_j \in D \cap \chi^{-1}(\chi(\tilde\la))} i_{\la_j}(\chi)
	$$
	Since all preimages of $\tilde \la$ are in $\inter(\Tf)$ (Lemma \ref{lem:continuitybt}), and $\chi$ is holomorphic in $\inter(\Tf)$, all summands are greater or equal than 1.
	So we conclude that $i_\la(\chi)\geq 1$ as desired. 
	This proves that $\chi$ is indeed topologically holomorphic, and surjectivity follows from Proposition~\ref{prop:surjectivity}.	
\end{proof}

\begin{prop}
\label{prop:chiinj}
	The map $\chi$ is injective on $\Lambda_0:=\{ \la \in \Lambda : v_\la \in V_\la \setminus \overline{U_\la}  \}$.
\end{prop}
Notice that by the properness assumption $\Lambda_0$ contains a inner neighborhood of $\partial \Lambda$.
\begin{proof}
	 Let $\holom_\la : \overline{V_{\la_0}}\setminus U_{\la_0} \to \overline{V_{\la}}\setminus U_{\la}$ denote the equivariant holomorphic motion of the fundamental annulus over $\Lambda$ given by the definition of equipped family,  with basepoint $\la_*$. 	 	  We first claim that $\la \mapsto \holom_\la^{-1}(v_\la)$ is injective on $\Lambda_0$. 
	To that end, assume that $z:=\holom_{\la_1}^{-1}(v_{\la_1}) = \holom_{\la_2}^{-1}(v_{\la_2})$, for some $\la_1, \la_2 \in \Lambda_0$. 
	Then $z \in V_{\la_*} \setminus \overline{U_{\la_*}}$. Let $s \mapsto z_s$, $s\in [0,1]$, be a continuous injective curve joining $z_0:=z$ to $z_1:=u_{\la_*}$ in the annulus  $(V_{\la_*} \setminus \overline{U_{\la_*}})\cup \{u_{\la_*}\}$ . 
	
	Consider the family of  maps
	$\{F_s: \la \mapsto \holom_{\la}(z_s)-v_\la\}_{s\in[0,1]}$, defined on $\overline{\Lambda}$.
	 For all $\la \in \partial \Lambda$, $v_\la \in \partial V_\la$, hence $v_\la \neq \holom_\la(z_s)$ for all $s \in [0,1]$. If we denote by  $\{\lambda(t)\}_{t,\in[0,1]}$ a continuous parametrization of $\partial \Lambda$, a simple closed curve, this means that $F_s(\lambda(t))\neq 0$, for all $s,t\in[0,1]$.
	 	 
	 Moreover, the curves $\{\Gamma(t):=F_s(\gamma(t))\}_{t \in [0,1]}$ are homotopic to each other by definition,  hence have the same winding number around $0$.
	Since $\holom_\la(u_{\la_*})=u_\la$ (by the equivariance property), $F_1(\la)=u_\la-v_\la$, and hence the winding number of $F_1(\lambda(t))=1$ as $\la(t)$ goes once around $\partial \Lambda$, by the properness assumption of $\{f_\la\}_{\la\in\Lambda}$. It  follows that the winding number is 1 also for $F_0(\lambda(t))$ and,  by the Argument Principle,  $F_0$ has a unique zero in $\Lambda$. But since $F_0(\la_1)=F_0(\la_2)=0$, we have in fact $\la_1=\la_2$. Thus the claim that $\la \mapsto \holom_\la^{-1}(v_\la)$ is injective on $\Lambda_0$ is proved.

	\medskip
	
	We will now prove that $\chi$ is injective on $\Lambda_0$. Let $\la_1, \la_2 \in \Lambda_0$
	be such that $\chi(\la_1)=\chi(\la_2)=:\alpha$.
	
	  Let $t_\la$  denote the tubing map and $\conj_\la$ denote the preferred conjugacy between  $f_{\la_*}$ and $T_{\alpha_{\sharp}}$ satisfying \eqref{eq:preferred conjugacy}.
	   Recall also from Corollary \ref{cor:bottcherlike}, that the map $\Psi_{\chi(\la)}(z)= t_\la \circ \conj_\la^{-1}(z)$, conjugates $T_{\chi(\la)}$ to $T_{\alpha_{\sharp}}$ for all $\la \in \Lambda$ and for all $z \in \conj_\la(V_\la \setminus \overline{U_\la})$, where $T_{\alpha_{\sharp}}$ is the model map in Remark \ref{rem:specialmap}. In particular, for all $\la \in \Lambda_0$, taking $z=\frac{1}{\chi(\la)}$ we obtain
	$$t_{\la}(v_\la) = \Psi_{\chi(\la)}\left(\frac{1}{\chi(\la)}\right).$$
	See Figure \ref{fig:injectivityexterior}. 
	
	By the definition of the tubing map $t_\la$, we therefore have that
	\[
	\holom_{\la_1}^{-1}(v_{\la_1}) = \holom_{\la_2}^{-1}(v_{\la_2})= \conj_{\la_*}^{-1}\left(\Psi_\alpha\left(\frac{1}{\alpha}\right)\right),
	\]
 Therefore, since the right hand side is independent of $\lambda$, $\la_1=\la_2$ by the above.
\end{proof}

\begin{figure}[hbt!]
\centering
   \includegraphics[width=0.85\textwidth]{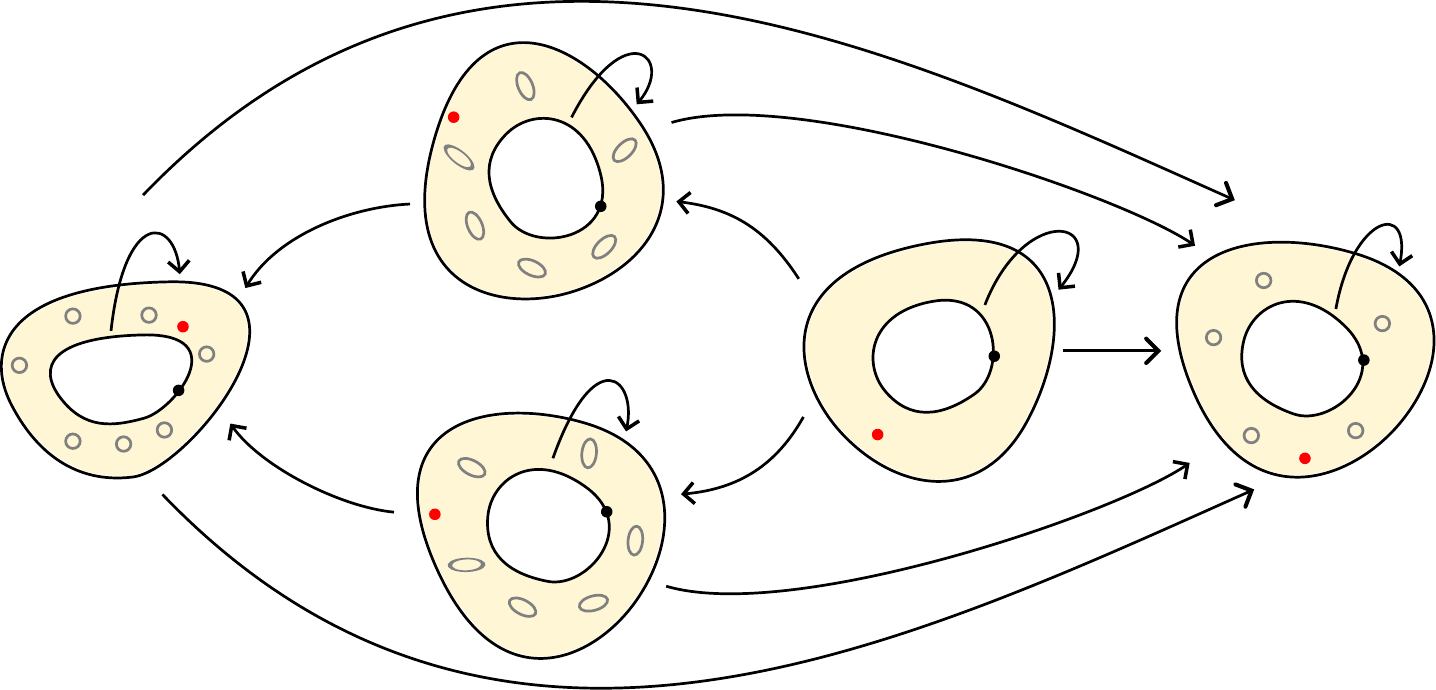}
   \setlength{\unitlength}{0.85\textwidth}
   \put(-0.33,0.43){\small $\Psi_\alpha$}
      \put(-0.33,0.035){\small $\Psi_\alpha$}
   \put(-0.33,0.38){\small $t_{\la_1}$}
      \put(-0.33,0.075){\small $t_{\la_2}$}
      \put(-0.23,0.242){\small $\tau$}
      \put(-0.04,0.33){\small $T_{\alpha_\sharp}$}
      \put(-0.11,0.13){\tiny \r $\Psi_{\alpha}(1/\alpha)$}
   \put(-0.29,0.33){\small $f_{\la_0}$}
   \put(-0.475,0.33){\small $\holom_{\la_1}$}
   \put(-0.475,0.13){\small $\holom_{\la_2}$}
   \put(-0.8,0.34){\small $\Phi_{\la_1}$}
   \put(-0.8,0.12){\small $\Phi_{\la_2}$}
      \put(-0.875,0.31){\small $T_{\alpha}$}
      \put(-0.875,0.26){\tiny \r $1/\alpha$}
      \put(-0.675,0.395){\tiny \r $v_{\la_1}$}
      \put(-0.69,0.12){\tiny \r $v_{\la_2}$}
   \put(-0.57,0.45){\small $f_{\la_1}$}
   \put(-0.58,0.22){\small $f_{\la_2}$}

\caption{\small \label{fig:injectivityexterior} Proof of injectivity of the Straightening map. Assuming $\chi(\la_1)=\chi(\la_2)=\alpha$, with $\la_1\neq \la_2$,  we show that necessarily $\holom^{-1}_{\la_1}(v_{\la_1})= \holom^{-1}_{\la_2}(v_{\la_2})$, a contradiction.}
 \end{figure}

We are now finally ready to finish the proof of Theorem E.

\begin{proof}[Proof of Theorem E]
		We have already proved that $\chi: \Lambda \to \C$ is continuous  (Proposition \ref{prop:cont} and Lemma~\ref{lem:continuitybt}).
		We also proved that $\chi: \tcal(\mathbf f) \to \tcal$ is surjective (Proposition~\ref{prop: chi surjective}). Let us now prove that $\chi: \tcal(\mathbf{f}) \to \tcal$ is injective,
		or equivalently, that the topological degree $\delta$ of $\chi: \tcal(\mathbf{f}) \to \tcal$ is 1 (See Proposition~\ref{prop:surjectivity}). 
		
		Since the natural $\{f_\la\}_{\la \in \Lambda}$ is proper, the map $G(\la):=u_\la -v_\la$ is holomorphic in a neighborhood of  $\overline{\Lambda}$. Moreover, the winding number of $G(\la)$ around 0 equals 1, as $\la$ turns once along $\partial \Lambda$ (see Definition \ref{def:proper}). Hence, by the Argument Principle, there is a unique $\tilde\lambda \in \Lambda$ such that $u_{\la}=v_\la$. It follows that   $\chi^{-1}(\{0\}) = \{\tilde\la\}$ and, 
by Proposition \ref{prop:surjectivity},  it is then enough to prove that $i_{\la_0}(\chi)=1$.		
		
		Let $\gamma \subset \Lambda$ be a Jordan curve close enough to $\partial \Lambda$ that $\gamma \subset \Lambda_0$, and that $\gamma$ winds once around $\la_0$.		Then $i_{\la_0}(\chi)= \ind(\chi \circ \gamma, \chi(\la_0) ) = \ind(\chi \circ \gamma, 0)$.
		By Proposition \ref{prop:chiinj}, $\chi$ is injective on $\gamma$; therefore $\chi \circ \gamma$ is a Jordan curve, and the domain it bounds contains $0$.
		Therefore 
		$$\delta = i_{\la_0}(\chi) = \ind(\chi \circ \gamma, 0)=1,$$
		and $\chi:  \tcal(\mathbf f) \to \tcal$ is indeed injective. Since $\tcal(\mathbf f)$ is compact, it follows that 
		$\chi:  \tcal(\mathbf f) \to \tcal$ is a homeomorphism.

		Finally, by Propositions \ref{prop:hyperbolic} and \ref{prop:non-hyperbolic}, $\chi$ is holomorphic on $\inter(\tcal(\mathbf f))$.
\end{proof}

\section{Finding proper equipped TL families.}\label{sec:existtlfam}

Our goal in this section is to show that, for any natural family of meromorphic maps (under some non-exceptionality conditions),  copies of the Tandelbrot set appear near almost any parameter in the activity locus of a given asymptotic value. This is the content of Theorem C. Furthermore, we shall prove Corollary D, stating that in the presence of critical points and asymptotic values, any open set intersecting the bifurcation locus of a natural family will contain copies of the Tandelbrot set or of the connectedness locus of $z^d+c$ for some $d\geq 2$.

We begin with the following observation, which we will use repeatedly in the proof of Theorem C and of Corollary D.

Let $\{f_\la\}_{\la \in\Lambda}$ be a natural family of meromorphic maps, and let $\la_0 \in\Lambda$, $v_{\la_0}$ an isolated asymptotic value be such that for every quasidisk containing  $v_{\la_0}$, $f_{\la_0}$ has a tract above $v_{\la_0}$ which is a quasidisk. 
We claim that this property automatically extends to all $f_\la$.	
	
Indeed,	since the family is natural, there exists holomorphic motions $\phi_\la$, $\psi_\la$ of $\rs$ with basepoint $\la_0$ such that for all $\la \in\Lambda$, 
$$f_\la = \phi_\la \circ f_{\la_0} \circ \psi_\la^{-1}.$$
Now, if $D$ is a quasidisk around $v_\la$ for some $\la \in\Lambda$, 
then $\phi_\la^{-1}(D)$ is a quasidisk around $v_{\la_0}$, so by assumption there is a tract $T$ for $f_{\la_0}$ above 
$\phi_\la^{-1}(D)$ which is a quasidisk. But then $\psi_\la(T)$ is also a quasidisk, and is a tract above $D$ for $f_\la$.

\subsection{Universality of the Tandelbrot set: Proof of Theorem C}

In the assumptions of Theorem C, $\Lambda=\D$ and the base point is $0$. The next lemma allows us to assume that the starting parameter is a virtual cycle parameter. 

\begin{lem}[Density of non-critical virtual cycles]\label{lem:exsimplevc}
	Under the assumptions of Theorem C, there exists $\la_0$ arbitrarily close to $0$ and $m \in \N^*$ such that 	$f_{\la_0}^m(v_{\la_0}) = \infty$, and $|(f_{\la_0}^m)'(v_{\la_0})| \neq 0$ (in the spherical metric).
\end{lem}

\begin{proof}
		
	Let  $p_{0}$ be a  simple pole for $f_{0}$ which is not a singular value.  
	Let $p_\la$ denote its holomorphic motion as a pole given by the Implicit Function Theorem, near $\la=0$.
	In view of Proposition 2.23 in \cite{ABFAhlfors} (see Proposition~\ref{prop:density_of_virtual cycles}), there exists  $\la_1$ arbitrarily close to 0, such that $f_{\la_1}^n(v_{\la_1})=\infty$ for some $n\in\N$.

	By assumption, $v_0$ does not persistently have a critical point in its forward orbit; and since critical points form a discrete subset of $\C$ and $\psi_\la$ maps the critical set of $f_0$ to those of $f_\la$, the set 
	$$C_n:=\{\la \in \D: (f_{\la}^{n})'(v_{\la}) =0  \}$$
	 is discrete near $\la_1$ (note however that the set $C_n$ is not discrete in $\D$, as it accumulates on parameters $\la$ such that $f_\la^{k}(v_\la)=\infty$ for some $k <n$,  and those cannot accumulate to $\la_1$).
	If $\la_1 \notin C_n$, then we are done, taking $\la_0=\la_1$ and $m=n$. 
	
	Otherwise, there exists $\delta>0$ such that $\D(\la_1,\delta) \cap C_n = \{\la_1\}$. By \cite[Prop. 2.13]{ABF} (see the Shooting Lemma, Proposition~\ref{shooting}), there exists $\la_0 \in \D(\la_1,\delta) \setminus \{\la_1\}$ such that $f_{\la_0}^{n+1}(v_{\la_0})=p_{\la_0}$; although not stated explicitly, the proof of \cite[Prop. 2.13]{ABF} also implies that $f_{\la_0}'(f_{\la_0}^n(v_{\la_0})) \neq 0$.  
	Finally, by definition, $p_{\la_0}$ is a simple pole, hence not a critical point. 
	The lemma is then proved, with $m:=n+2$.	
\end{proof}

 By virtue of the  parameter version of the Straightening Theorem (Theorem E), to  prove Theorem C, it suffices to prove the existence of proper equipped TL families in a region arbitrarily close to $0$.  Hence our goal is to prove the following.

\begin{theo}[Existence of proper equipped TL families]\label{th:existtlfam}
	Under the assumptions of Theorem C, there exists  $\Lambda \Subset \D$  simply connected, $U_\la \Subset  V \subset \rs$ and $N \in \N^*$ 
	such that $\{f_\la^{N}: U_\la\ra  V\}_{\lambda\in\Lambda}$   is a proper equipped TL family, with  singular values $f_\la^{N-1}(v_\la)$.
\end{theo}

The remaining of this section is devoted to the proof of Theorem \ref{th:existtlfam}.

	In the following, all distances between points in the Riemann sphere are measured in the spherical metric. In particular, we will denote by $\D(\infty,r)$ the open ball of spherical radius $r>0$ at $\infty$ in the Riemann sphere. 
	
By Lemma \ref{lem:exsimplevc} there exists  $\la_0 \in \D$ arbitrarily close to 0 and $n\in \N $ such  that $f_{\la_0}^n(v_{\la_0})=\infty$ with a non-zero derivative. 	Since the family is natural, we may write $f_\la=\phi_\la \circ f_{\la_0} \circ \psi_\la^{-1}$, where $\phi_\la, \psi_\la$ are holomorphic motions of the Riemann sphere with base point $\la_0$ which tend to the identity as $\la\ra\la_0$.
	Let
	$$
	G(\la):=\psi_\la^{-1} \circ f_\la^n(v_\la).
	$$
	 By \cite[Remark 2.12]{ABF}, there exists a neighborhood $W_G$ of $\infty$ such that $G$ is a topological branched cover of finite degree over $W_G$,  whose only critical point (if any) is $\la_0$.
	
 We start by setting the stage. Let  $p_{\la_0}$ be a  simple pole for $f_{\lambda_0}$ which is not a singular value (such a pole exists for $\la_0$ sufficiently close to $0$ by assumption (3) in Theorem C).  
Let $p_\la$ denote the corresponding simple pole of $f_\la$ for $\la$ close enough to $\la_0$, given by the Implicit Function Theorem.   For $r>0$, let $V_\la(r)$ denote the connected component of $f_\la^{-1}(\D(\infty,r))$ containing $p_\la$. 	 By the implicit function theorem  there exists $r>0$, a parameter region $\Lambda'\subset \D$ containing $\lambda_0$, and  a holomorphic motion 
	$h_\la: V_{\la_0}(r) \to V_\la(r)$ defined on $\Lambda'$, 
	such that 
	\begin{equation}\label{eq:comm1}
	f_\la \circ h_\la(z) = f_{\la_0}(z)
	\end{equation}
	 and such that 
	$V_\la(r) \cap S(f_\la) = \emptyset$ for all $\la \in \Lambda'$. 	Using  the assumption that  the orbit of $v_0$ does not contain other singular values, 	 up to reducing $r$ and $\Lambda'$ if necessary,  we can assume that for all $\la \in \Lambda'$
	there exists a well-defined univalent map $k_\la : \D(\infty, r) \to \rs$  which is the inverse branch of $f_\la^n$ 
	mapping $f_\la^n(v_\la) \in \D(\infty, r)$ to $v_\la$. Up to reducing $r$ further, we can assume also  that $\D(\infty,r)\subset W_G$. 
Up to reducing  $\Lambda'$ again  if necessary, there is a constant  $\delta$ given by Lemma \ref{lem:smalltracts}  such that  for all $\la \in \Lambda'$, a suitable  tract $T_\la(\delta)$ above $\D(v_\la, \delta)$ is compactly contained in $\D(\infty, \frac{r}{2})$. 

\begin{figure}[htb!]
\centering
\includegraphics[width=0.35\textwidth]{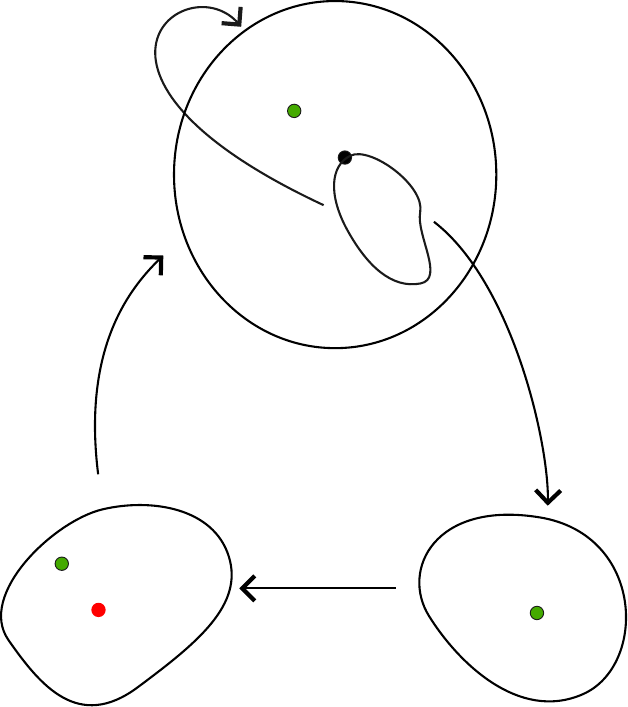}
\setlength{\unitlength}{0.35\textwidth}
\put(-0.28,1.05){$V=\D(\infty,r/2)$}
\put(-0.37,0.87){\small $U_\la$}
\put(-0.14,0.55){\small $f_\la$}
\put(-0.92,0.55){\small $f_\la$}
\put(-0.53,0.23){\small $f_\la^{n+1}$}
\put(-0.04,0.26){\small $V_\la^{n+1}$}
\put(-0.88,0.1){\scriptsize $p_\la$}
\put(-0.18,0.1){\scriptsize $v_\la$}
\put(-0.88,0.22){\scriptsize $f_\la^{n+1}(v_\la)$}
\put(-1.18,0.24){\small$V_\la(r/2)$}
\put(-0.6,0.99){\scriptsize $f_\la^{n+2}(v_\la)$}
\put(-0.88,1.07){\small $f_\la^{n+3}$}
\put(-0.53,0.87){\scriptsize $\infty$}
\caption{\small \label{fig:renormalization2} General sketch of the proof of Theorem \ref{th:existtlfam}}
\end{figure} 

We are now ready to start the actual construction (see Figure~\ref{fig:renormalization2}).  
Since by assumption $p_{\la_0} $ is not a singular value, up to diminishing $r$  and $\Lambda'$ we can find a Jordan domain  $W \subset \C$ such that $V_{\la_0}(r) \Subset W$ and $W \cap S(f_\la) = \emptyset$ for all $\la \in \Lambda'$.  Again by reducing $\Lambda'$, and using the fact that $h_\la,\phi_\la\ra\id$ as $\la\ra\la_0$,  we can assume that  for $\la \in \Lambda'$ the following hold:  
	\begin{enumerate}
		\item $h_\la(V_{\la_0}(r)) \Subset \phi_\la(W)$
		\item $V_{\la_0}(\frac{3}{4}r) \Subset \phi_\la(W)$
		\item $h_\la(V_{\la_0}\left( \frac{r}{2} \right)) \Subset V_{\la_0}(\frac{3}{4}r)$.
	\end{enumerate}
	 
Because infinity is an essential singularity and $W\cap S(f_{\la_0})=\emptyset$,   there exists a sequence of simply  connected domains $\Omega_k\subset\C$, shrinking to $\infty$, that map univalently to $W$ under $f_{\la_0}$. In particular, we may fix $\Omega \subset \C$ a simply connected domain such that 
	\begin{enumerate}
	 \setcounter{enumi}{3}
		\item $f_{\la_0}: \Omega \to W$ is a conformal isomorphism
		\item $\psi_\la(\Omega)\Subset \D(\infty,r/2)\Subset  W_G$ for all $\la\in\Lambda'$
		\item\label{it:estdiam} For all $\la \in \Lambda'$, $\mathrm{diam}(\psi_\la(\Omega)) \leq \frac{\delta C }{100 m}$, where $C:=\sup_{\la \in \Lambda'} | (f_\la^n)'(v_\la)|$ and $m:=\mathrm{mod}(W \setminus V_{\la_0}(\frac{3}{4}r))$
		\item There exists $\hat{\Lambda}$  a connected component of $G^{-1}(\Omega)$ such that $\hat{\Lambda} \Subset \Lambda'$.
	\end{enumerate} 
	To justify item (7), observe that $G(\la_0)=\infty$, hence since $G$ is open, $G(\Lambda')$ covers a neighborhood of $\infty$. If $\Omega$ is compactly contained in such neighborhood and since $\Omega$ does not contain infinity (which is the only possible singular value for $G$), the claim follows.
	
	By the choice of $\Omega$, $f \circ G: \overline{\hat{\Lambda}} \to \overline{W}$ is a 
	homeomorphism. In particular, $\hat{\Lambda}$ is a Jordan domain. Notice that $\la_0\notin\hat{\Lambda}$ since $\infty\notin\Omega$. 
	
	Let 
	$$
	\Lambda = \left\{ \la \in \hat \Lambda : f_\la^{n+2}(v_\la) \in \D(\infty, \frac{r}{2})\right\}.
	$$
	
	We now claim  that for any $y \in \D(\infty, \frac{r}{2})$, there exists a unique $\lambda \in \hat{\Lambda}$ such that $f_\la^{n+1}(v_\la)\in V_\la(r/2)$ and  $f_\la^{n+2}(v_\la) = y$. Since the map $\la \mapsto f_\la^{n+2}(v_\la)$ is holomorphic, the claim implies that it  is  a conformal isomorphism from $\Lambda$ to $\D(\infty,\frac{r}{2})$, and hence that  $\Lambda$ is also a Jordan domain. 
	To prove the claim  let 
	$$
	F(\la):=f_\la^{n+1}(v_\la) = \phi_\la \circ f \circ G(\la),
	$$
	 and for all $z \in V_{\la_0}\left( \frac{r}{2} \right)$, let $\gamma_z(\la):=h_\la(z)$. Both of these maps are holomorphic on $\hat \Lambda$. Using item (1) we have 
	\begin{equation}\label{eq:compactness}
	\gamma_z(\hat{\Lambda}) =V_\la(r/2) \Subset V_{\la_0}(\frac{3}{4}r)\Subset \phi_\la(W) = F(\hat{\Lambda}).
	\end{equation} 
	Therefore, by an application of the Argument Principle to $F-\gamma_z$ (see e.g. \cite[Lemma 2.15]{ABF}), there exists a unique $\la_z \in \hat{\Lambda}$ such that $F(\la_z)= \gamma_z(\la_z) \in V_{\la_z}(r/2)$.
	Then, by definition of $h_\la$, we have 
\begin{equation}\label{eq: singular value}
	f_{\la_z}^{n+2}(v_{\la_z}) ={  f_{\la_z}\circ F(\la_z)}={ 	f_{\la_z}\circ \gamma_z(\la_z)= 	f_{\la_z}\circ h_{\la_z}(z) =} f_{\la_0}(z) \in \D(\infty, \frac{r}{2}).
	\end{equation}

	 This proves the claim, and shows  that $\Lambda \Subset \hat{\Lambda}$.

\begin{figure}[htb!]
\centering
\includegraphics[width=0.4\textwidth]{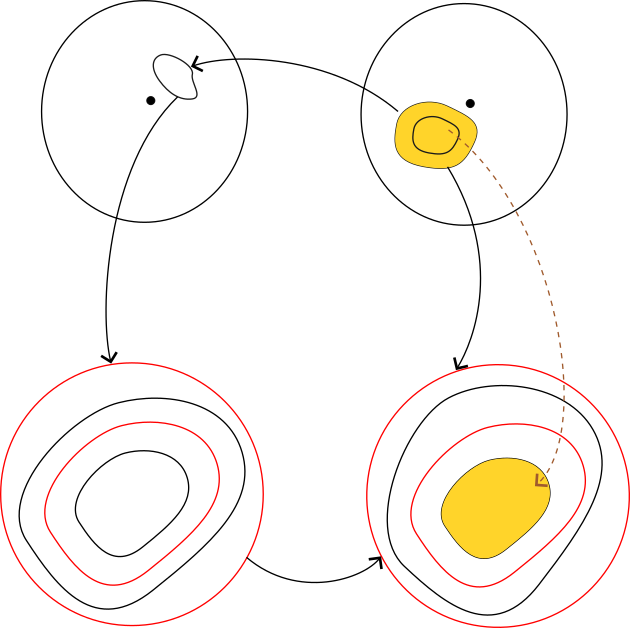}
\setlength{\unitlength}{0.4\textwidth}
\put(-0.24,0.83){\scriptsize $\la_0$}
\put(-0.34,0.84){\small $\hat\Lambda$}
\put(-0.33,0.77){\scriptsize $\Lambda$}
\put(-0.52,0.91){\small $G$}
\put(-0.75,0.92){\small$\Omega$}
\put(-0.83,0.83){\scriptsize$\infty$}
\put(-0.92,0.55){\small$f_{\la_0}$}
\put(-1,0.35){\small \r $W$}
\put(-0.91,0.37){\tiny$V_{\la_0}(r)$}
\put(-0.73,0.33){\r \tiny$V_{\la_0}(\frac{3}{4}r)$}
\put(-0.86,0.2){\tiny$V_{\la_0}(\frac{r}{2})$}
\put(-0.52,0.03){\small  $\phi_\la$}
\put(-0.53,0.35){\small \r $\phi_\la(W)$}
\put(-0.1,0.1){\tiny$V_{\la}(r)$}
\put(-0.28,0.33){\r \tiny$V_{\la_0}(\frac{3}{4}r)$}
\put(-0.26,0.18){\tiny$V_{\la}(\frac{r}{2})$}
\put(-0.23,0.53){\small $F$}
\put(-0.12,0.53){\small $\gamma_z$}
\caption{\label{fig:Lambda}\small The set $\hat \Lambda$ and its images under $F$ and $\gamma_z$. Recall that $G(\la):=\psi_\la^{-1} \circ f_\la^n(v_\la)$ and $F(\la):=f_\la^{n+1}(v_\la)= \phi_\la \circ f \circ G(\la)$. }
\end{figure}

We now proceed to define a family of tangent-like maps over $\Lambda$. We will then show that it is proper and equipped.

Recall that for $\lambda\in\Lambda'$,   $f_\la^{n+1}(v_\la) \in V_{\la}(\frac{r}{2})$ and that  $V_\la(\frac{r}{2})$ is simply connected and disjoint from $S(f_\la)$. Let $V_\la^{(n+1)}$ denote the connected component of $f_\la^{-(n+1)}(V_\la(\frac{r}{2}))$ containing $v_\la$.  Recall that $\delta$ was chosen so that the tracts $T_\la(\delta)$ above $\D(v_\la,\delta)$ are compactly contained in $\D(\infty,\frac{r}{2})$ for $\la\in\Lambda'$.
\begin{lem}\label{lem:smalldiam} 
The set $V_\la^{(n+1)}$ is simply connected and satisfies
$$
\diam(V_\la^{(n+1)}) \leq \delta \text{\ \ \  for all $\la \in \Lambda$.}
$$
 
\end{lem}
\begin{proof}
For all $\la \in \Lambda$, let $V_\la^{(j)}$ denote the connected component of $f_\la^{-j}(V_{\la}\left(\frac{r}{2}\right))$ containing $f_\la^{n+1-j}(v_\la)$. Since $V_{\la}\left(\frac{r}{2}\right) \Subset W$ and $W \cap S(f_\la) = \emptyset$ for all $\la \in \Lambda$, $V_\la^{(1)}$ is simply connected and there is a univalent inverse branch 	$\ell_\la$ of $f_\la^{-1}$, mapping $V_{\la}\left(\frac{r}{2}\right)$ conformally to $V_\la^{(1)}$ and depending holomorphically on $\la$. 
			
	We claim that for all $\la \in \Lambda$, $V_\la^{(1)} \Subset \psi_\la(\Omega).$ 
	Indeed, recall that  for all $\la \in \Lambda$, $G(\la) = \psi_\la^{-1} \circ f_\la^n(v_\la) \in \Omega$; hence $f_\la^n(v_\la) \in \psi_\la(\Omega)$. Moreover, $f_{\la_0}$ maps $\Omega$ univalently to $W$, therefore $\ell_{\la_0}(W)=\Omega$.  
	 Recall also that by items (3) and (2), we have $V_{\la}(\frac{r}{2}) \Subset \phi_\la(W)$ for all $\la \in \Lambda'\supset \Lambda$. 
	
	Therefore 
	
	$$
	V_\la^{(1)}=\ell_\la(V_\la(r/2))\Subset\psi_\la\circ \ell_{\la_0}\circ \phi_\la^{-1}(\phi_\la(W)) =\psi_\la(\Omega),
	$$
	
	as claimed. 

	By   item \eqref{it:estdiam} we then have 
	$$
	\diam(V_{\la}^{(1)}) \leq \diam(\psi_\la(\Omega)) \leq  \frac{\delta C}{100 m} \text{ \ \ \ 	for all $\la \in \Lambda$.}
	$$ 
 Hence, $k_\la \circ \ell_\la : W \to \C$ is a well-defined inverse branch 
	of $f_\la^{-(n+1)}$, mapping $V_{\la}\left(\frac{r}{2}\right)$ to $V_\la^{(n+1)}$.

	Recall that for all $\la \in \Lambda'$, $V_{\la}(\frac{r}{2}) \Subset V_{\la_0}(\frac{3}{4}r) \Subset W$; hence 
	the modulus of the topological annulus $\ell_\la(W) \setminus V_\la^{(1)}$ is at least $m=\mathrm{mod}(W \setminus V_{\la_0}(\frac{3}{4}r))$. Since $k_\la$ is univalent on $\ell_\la(W)$, Koebe's theorem applied to $k_\la$ implies that 
	we have $\diam(V_\la^{(n+1)}) \leq \delta$ for all $\la \in \Lambda$.  
\end{proof}

By assumption, there is a logarithmic tract $U_{\la_0}$ above $V_{\la_0}^{(n+1)}$ which is a quasidisk  (see the observation at the beginning of this section). The choice of tracts for other values of $\lambda$ is done as follows.
 \begin{lem}\label{lem:holomo}
There exists a holomorphic motion $m_\la: U_{\la_0} \to U_\la$ over $\Lambda$, such that $U_\la$ is a logarithmic tract over $V_{\la}^{(n+1)}$, and
$$
f_\la^{n+3} \circ m_\la=  f_{\la_0}^{n+3} \text{on $U_{\la_0}$}.
$$

\end{lem}
\begin{proof}
	The holomorphic motion $m_\la$ will be given by Lemma \ref{lem:tractmoveholo}, hence we need to show that there exists a holomorphic motion of $V_\la^{(n+1)}$ over $\Lambda$.  To that end, 
	let $w_\la:=k_\la \circ \ell_\la \circ h_\la \circ f_{\la_0}^{n+1}: V_{\la_0}^{(n+1)} \to V_\la^{(n+1)}$.
	For all $\la \in \Lambda$, the map $w_\la$ is injective, as a composition of univalent maps. 
	Moreover, since $k_\la, h_\la$ and $\ell_\la$ are all holomorphic in $(\la,z)$, it follows that for all $z \in  V_{\la_0}^{(n+1)}$, the map $z \mapsto w_\la(z)$ is holomorphic. (In fact, more is true: $(\la,z) \mapsto w_\la(z)$ is holomorphic, but we do not require this). In particular, $w_\la$ is a holomorphic motion of $V_\la^{(n+1)}$ over $\Lambda$; and by construction, $w_\la(v_{\la_0}) = v_\la$. Finally, by Lemma \ref{lem:tractmoveholo}, there exists a holomorphic motion $m_\la$ such that
	\[ 
	f_\la \circ m_\la = w_\la \circ f_{\la_0}.
	\]
Hence, by the definition of $w_\la$, \eqref{eq:comm1},  and using that $k_\la \circ \ell_\la : W \to \C$ is a well-defined inverse branch of $f_\la^{-(n+1)}$ (see the proof of Lemma \ref{lem:smalldiam}), mapping $V_{\la}\left(\frac{r}{2}\right)$ to $V_\la^{(n+1)}$, we have that  
\[
f_\la^{n+3}\circ m_\la 
= f_\la^{n+2}\circ k_\la \circ \ell_\la \circ h_\la\circ f_{\la_0}^{n+1}\circ f_{\la_0} 
=f_\la\circ h_\la \circ f_{\la_0}^{n+2}
=f_{\la_0} \circ f_{\la_0}^{n+2}
=f_{\la_0}^{n+3},
\]
as claimed.
\end{proof}

It follows that $U_\la$ is a quasidisk for all $\la\in\Lambda$. 
Since the diameter of $V_{\la}^{(n+1)}$ is smaller than $\delta$ (Lemma \ref{lem:smalldiam}), it follows that  
$U_\la \Subset V:=\D(\infty, \frac{r}{2})$ (Lemma \ref{lem:smalltracts}). Moreover, $f_\la^{n+3} : U_\la \to \D(\infty, \frac{r}{2}) \setminus \{f_\la^{n+2}(v_\la)\}$ is a universal cover, and hence
\[
g_\la:=f_\la^{n+3} :  U_\la \to V \setminus \{f_\la^{n+2}(v_\la)\}
\]
is indeed a TL map.  By construction, $\Lambda$ is a Jordan domain, $f_\la^{n+2}(v_\la) \in V$ and $g_\la$ is holomorphic in both $z$ and $\la$, hence $\{g_\la: U_\la \to V \}_{\la \in \Lambda}$ is a TL family.

It remains to prove that  $\{g_\la: U_\la \to V \}_{\la \in \Lambda}$ is proper and equipped. 	
To see properness, observe that  $\la \mapsto f_\la^{n+2}(v_\la)$ is an isomorphism between $\overline{\Lambda}$ and $\overline{\D}(\infty, \frac{r}{2})$, so $\la \mapsto f_\la^{n+2}(v_\la)$ turns once around $\infty$ along  $\partial V$, as $\la$ turns once along $\partial \Lambda$.   It follows from the fact that $\Lambda\Subset \hat\Lambda$ and the uniqueness of $\la_z$, that $g_\la$ is a TL family over a slightly larger domain, hence the family is proper.
	
	To end the proof we now show that the family is equipped, i.e. that there  exists a basepoint $\la_*\in \Lambda$ and a holomorphic motion $\tilde m_\la: \overline{V} \setminus  U_{\la_*}  \to V \setminus U_\la$ such that for all $\la \in \Lambda$, for all $z \in \partial U_{\la_*}$, 
	$$
	g_\la \circ \tilde m_\la(z)= \tilde m_\la \circ g_{\la_*}(z).
	$$
	From Lemma \ref{lem:holomo}, there is a holomorphic motion $m_\la: \overline{U_{\la_0}} \to \overline{U_\la}$ 
	such that for all $\la \in \Lambda$, for all $z \in \overline{U_{\la_0}}$,
	$$g_\la \circ m_\la(z)=  g_{\la_0}(z).$$
	Let $\la_* \in \Lambda$ be an arbitrary parameter, and set 
	$$
	\tilde m_\la(z) = \left\{
	\begin{array}{ll}
		m_\la \circ m_{\la_*}^{-1}(z)  &\text{ if $z \in \partial U_{\la_*}$} \\
		 z  &\text{ if $z \in \partial \D(\infty, \frac{r}{2})$.}
	\end{array}
	\right.
	$$
	This defines an equivariant holomorphic motion of  $\partial U_\la \cup \partial \D(\infty, \frac{r}{2})$
	based at $\la_*$, i.e. for all $\la \in \Lambda$, for all $z \in \partial U_{\la_*}$,
	$$
	g_\la \circ \tilde m_\la(z)
	=g_\la\circ m_\la\circ m_{\la_*}^{-1}(z)
	= g_{\la_0} \circ  m_{\la_*}^{-1}(z) 
	=    g_{\la_*}(z)
	=\tilde m_\la \circ g_{\la_*}(z).
	$$
	Finally, we can extend $\tilde m_\la$ to a holomorphic motion of $\rs$ over $\Lambda$ by  Slodkowski's  extension Theorem~\ref{thm:Slodkowski}, which we still denote by $\tilde m_\la$.
	By continuity of $\tilde m_\la$ and since $\tilde m_\la(\partial U_{\la_*}) = \partial U_\la$ and $\tilde m_\la(\partial V) = \partial V$, we have $\tilde m_\la(\overline{V} \setminus U_{\la_*}  ) = \overline{V} \setminus U_\la$, as desired.

This concludes the proof of Theorem C. 

\subsection{Density of copies: Proof of Corollary D}\

Following \cite{mcmullen2000mandelbrot}, we recall here the notion of unramified points.

\begin{defi}[Unramified points]
	Let $f$ be a transcendental meromorphic map, and let $zw\in \rs$. We say that $z\in \C$ is an unramified preimage of order $n\in\N_*$ of $w$ if  $f^n(z)=w$ and $f^n$ is locally univalent near $z$.
	If $w$ has unramified preimages of arbitrarily high order, we say that $w$ is unramified.
\end{defi}

A critical point $c$ is called Misiurewicz if there exists $k\in\N$ such that  $f^k(c)$ is a repelling periodic point. 
The next lemma asserts that if $c_\la$ is an active critical point which is not persistently a Picard exceptional value, then after small perturbations we can find an unramified Misiurewicz critical point. The assumption that $c_\la$ is not persistently a Picard exceptional value is necessary, since for instance if $c_\la$ is always an omitted value, then it can never be unramified. The proof is therefore slightly different than in the case of rational maps. Compare with \cite[Prop. 2.4 and 2.5]{mcmullen2000mandelbrot}.

\begin{lem}[Density of Misiurewicz critical points]\label{lem:unramif}
	Let  $\{f_\la\}_{\la \in \D}$ be a natural family of finite type meromorphic maps. Let $c_\la$ be a marked critical point, active at $\la_1 \in \D$, and assume that there exists $\la_0 \in \D$ such that $c_{\la_0}$ is not a Picard exceptional value. Then, there exists some $\la_2$ arbitrarily close to $\la_1$ such that $f_{\la_2}$ has  an unramified  critical point which is Misiurewicz.
\end{lem}
\begin{proof}
	Let $N$ be the number of singular values.
	We start with the following observations: if $c_\la$ has at least $N+1$ unramified preimages of order $k$, then it has infinitely many unramified preimages of order $k+1$. Indeed, there is at least one unramified preimage of order $k$ which is not a singular value, and therefore it itself has infinitely many unramified preimages. In particular, if $c_\la$ has $N+1$ unramified preimages of order 1, then $c_\la$ is unramified.
	
	First, we claim that the set of $\la \in \D$ such that $c_\la$ is a Picard exceptional value is discrete.
		Indeed, let $\mathcal{V}_\la$ denote the set of Picard exceptional values of $f_{\la}$. 
	Since the family $\{f_\la\}_{\la \in \D}$ is natural, we may write $f_\la = \phi_\la \circ f_{\la_0} \circ \psi_\la^{-1}$, with $\phi_{\la_0} = \psi_{\la_0} = \id$, which implies that $\mathcal{V}_\la=\phi_\la(\mathcal{V}_{\la_0})$.
	By Picard's Theorem, $f_{\la_0}$ has at most two Picard exceptional values on the sphere, say $e_1, e_2$. It follows that, if $\la$ is such that $c_\la$ is Picard exceptional,  then either $c_\la=\phi_\la(e_1)$ or  $c_\la=\phi_\la(e_2)$. 
		On the other hand,  $c_\la=\psi_\la(c_{\la_0})$. Hence the set of $\la\in\D$ such that $c_\la$ is a Picard exceptional value satisfy  the equation $\phi_\la(e_i)=\psi_\la(c_{\la_0})$ for either $i=1$ or $i=2$. Both equations are holomorphic in $\la$ and are not satisfied for $\la=\la_0$, hence their set of solutions is discrete in the one-dimensional parameter space $\D$. 

	Hence, by density of Misiurewicz  parameters (Proposition \ref{prop: density Misiurewicz}), there exists $\la_2$ arbitrarily close to $\la_1$ such  that $c_{\la_2}$ is Misiurewicz and not  Picard exceptional.  In particular, $c_{\la_2}$ has infinitely many preimages, only finitely many  of which can be singular values (since the family is of finite type). We may therefore choose a preimage $y_1$  of $c_\la$ such that $y_1$ is not a singular value (but might possibly be a critical point). If $y_1$ is not a critical point, then it is an unramified preimage of $c_{\la_2}$, and since it is not a singular value, it itself has infinitely many unramified preimages. Then, by our observation, $c_{\la_2}$ is unramified and we are done.  Otherwise, $y_1$ is a critical point which is not a singular value. But then $y_1$ has infinitely many unramified preimages of order 1 (since it is not a singular value), hence $y_1$ is sn unramified critical point; and since $f_{\la_2}(y_1)=c_{\la_2}$, $y_1$ is also Misiurewicz.
\end{proof}

We now recall the result from McMullen below. It is stated for rational maps, however the proof is purely local and holds for  transcendental meromorphic maps. The exact statement is also more technical but more precise; we have extracted only what we will use.

\begin{theo}[\cite{mcmullen2000mandelbrot}, Theorem 4.1]\label{th:mcmu}
	If $\{f_\la\}_{\la \in \Lambda}$ is a family of rational or meromorphic maps such that $f_{\la_*}$ has a marked unramified critical point $c_\la$ of constant degree $d \geq 2$ which is Misiurewicz, then for any neighborhood $W$ of $\la_*$, there is a quasiconformal embedding $\chi : \partial \mcal_d \to \bif \cap W$.
\end{theo}

Recall that in the case of a natural family, all critical points are always marked and have constant local degree
throughout the family. We are now ready to prove Corollary D.

\begin{proof}[Proof of Corollary D]
	Let $\la_1$ be a parameter in the bifurcation locus and let $W$ be an arbitrary neighborhood of $\la_1$. By characterization of $\JJ$-stability (Theorem \ref{thm:J stability}) 
	there is at least one active singular value $v_{\la_1}$ at $\la_1$: this active singular value must be either critical or asymptotic. 
 If $v_{\la_1}$ is a critical value then there is an active critical point $c_{\la_1}$. By Lemma \ref{lem:unramif}, there exists $\la_2 \in W$ such that $f_{\la_2}$ has an unramified critical point which is Misiurewicz. By Theorem \ref{th:mcmu}, there is a topological embedding $\chi: \partial \mcal_d \to \bif \cap W$ for some $d \geq 2$.
 Otherwise, $v_{\la_1}$ is an asymptotic value. If its forward orbit does not persistently contain a critical point, then we can apply directly Theorem C, in which case, we have a topological embedding $\chi: \partial \tcal \to \bif \cap W$. If its forward orbit does contain a critical point, then that critical point is active, and we are back to the previous case.
\end{proof}

\bibliographystyle{amsalpha}	
\bibliography{ABFTangentLike}

\end{document}